\documentclass{aims}
\usepackage{fixltx2e}
\usepackage{float}
\usepackage{mathtools}
\usepackage{amssymb,amsmath,amsthm}
\usepackage{fixmath}
\usepackage{braket}
\usepackage{caption}
\usepackage{subcaption}
\usepackage{enumitem}
\usepackage{units}
\usepackage[final]{graphicx}
\usepackage{color}
\usepackage{xcolor}
\usepackage{booktabs}
\usepackage[english]{babel}
\usepackage[utf8]{inputenc}
\usepackage{lmodern}
\usepackage{pgfplots}
\pgfplotsset{compat=1.3}
\pgfplotsset{every tick label/.append style={font=\tiny}}
\usepackage[final,stretch=10]{microtype}
\usepackage[foot]{amsaddr}
\usepackage{cite}
\usepackage[colorlinks=true]{hyperref}
\hypersetup{urlcolor=blue, citecolor=red}
\def\currentvolume{X}
 \def\currentissue{X}
  \def\currentyear{200X}
   \def\currentmonth{XX}
    \def\ppages{X--XX}
     \def\DOI{10.3934/xx.xx.xx.xx}

\def\half{{\textstyle \frac{1}{2}}}

\def\rh{\mathrm{h}}

\DeclareMathOperator{\var}{var}
\numberwithin{equation}{section}

\newcommand{\M}{\mathcal M_T}
\newcommand{\C}{\mathcal C_T}
\newcommand{\ML}{\mathcal M(\Omega,L^2(I))}
\newcommand{\CL}{\mathcal C_0(\Omega,L^2(I))}
\newcommand{\LwM}{L_{w^*}^2(I,\mathcal M(\Omega))}

\renewcommand{\c}{\mathcal C_0(\Omega)}
\newcommand{\m}{\mathcal M(\Omega)}

\newcommand{\R}{\mathbb R}
\renewcommand{\L}{L^2(I\times\Omega)}
\renewcommand{\l}{L^2(\Omega)}

\newcommand{\lt}{L^2(I)}

\newcommand{\diag}{\operatorname{diag}}

\renewcommand{\rh}{\vartheta}
\newcommand{\spa}{\operatorname{span}}

\def\beal{\begin{equation}\begin{aligned}}
\def\eeal{\end{aligned}\end{equation}}
\def\bealn{\begin{equation}\left\{\begin{aligned}}
\def\eealn{\end{aligned}\right.\end{equation}}

\newtheorem{theorem}{Theorem}[section]

\newtheorem{lemma}[theorem]{Lemma}
\newtheorem{proposition}{Proposition}

\theoremstyle{definition}
\newtheorem{definition}[theorem]{Definition}
\newtheorem{remark}{Remark}
\newtheorem{assumption}{Assumption}

\definecolor{green}{rgb}{0,0.7,0}

\newcommand\Blue[1]{\textcolor{black}{#1}}

\title[FEM for hyperbolic control problems]{Finite element error analysis for measure-valued optimal control problems governed by a 1D wave equation with variable coefficients}

\author[Philip Trautmann and Boris Vexler and Alexander Zlotnik]{}



\email{philip.trautmann@uni-graz.at}
\email{vexler@ma.tum.de}
\email{azlotnik@hse.ru}

\thanks{The first and the second author were supported by FWF and DFG through the International Research Training Group IGDK 1754 `Optimization and Numerical Analysis for Partial Differential Equations with Nonsmooth Structures'. The third has been funded within the framework of the Academic Fund Program at the National Research University Higher School of Economics in 2016-2017 (grant no. 16-01-0054) and by the Russian Academic Excellence Project `5-100'. He also thanks the Technical University of Munich for its hospitality in 2014-2015 years.}

\begin{document}
\maketitle

\centerline{\scshape Philip Trautmann}
\medskip
{\footnotesize
 \centerline{Department of Mathematics and Scientific Computing}
 \centerline{University of Graz}
   \centerline{Heinrichstra{\ss}e 36}
   \centerline{8010 Graz, Austria}
}
\medskip
\centerline{\scshape Boris Vexler}
\medskip
{\footnotesize
 \centerline{Zentrum Mathematik}
 \centerline{Technische Universit\"at M\"unchen}
   \centerline{Boltzmannstra{\ss}e 3}
   \centerline{85748 Garching bei M\"unchen, Germany}
}
\medskip
\centerline{\scshape Alexander Zlotnik}
\medskip
{\footnotesize
 \centerline{Department of Mathematics at Faculty of Economic Sciences}
 \centerline{National Research University Higher School of Economics}
   \centerline{Myasnitskaya 20}
   \centerline{101000 Moscow, Russia}
}
\bigskip
\centerline{(Communicated by \Blue{the associate editor name})}
\begin{abstract}
This work is concerned with the optimal control problems governed by a 1D wave equation with variable coefficients and the control spaces $\mathcal M_T$ of either measure-valued functions
{$\LwM$}
or vector measures $\mathcal M(\Omega,L^2(I))$.
The cost functional involves the standard quadratic {tracking} terms and the regularization term $\alpha\|u\|_{\mathcal M_T}$ {with} $\alpha>0$.
We construct and study three-level in time bilinear finite element discretizations for {this class of} problems.
The main focus lies on the derivation of error estimates for the optimal state variable and the error measured in the cost functional. The analysis is mainly based on some previous results of the authors.
The numerical results are included.
\end{abstract}
\section{Introduction}
This work is concerned with the discretization and numerical analysis of optimal control problems involving
{a 1D linear wave equation with variable coefficients} and controls taking values in certain measure spaces.
{The combination of variable coefficients and irregular data leads to significant technical problems.}

\par Motivated by industrial applications as well as applications in the natural sciences, in which one is interested to place actuators in form of point sources in an optimal way, see, e.g., \cite{BermudezGamalloRoriguez:2004,Brunner:12} or in the reconstruction of point sources from given measurements, see, e.g., \cite{KunischTrautmannVexler14,PieperHelm2016}, measure valued optimal control problems involving PDEs gained attention in the last years. These problems can be translated into optimization problems in terms of the coordinates and coefficients of the point sources. However, these optimization problem are non-convex since the solution of the state equation (PDE) depends in a non-linear way on the coordinates of the point sources. Thus one has to deal with multiple local minima. Several authors suggested to cast the control problem resp. inverse problem in form of an optimization problem over a suitable measure space $\M$ involving a convex regularization functional $R$ which favors point sources as solutions. In our case we introduce the following problem formulation involving the 1D wave equation
\begin{equation}\label{proto_prob}
\begin{aligned}
J(u)&=F(y)+R(u) \to \min_{u\in \M}\\
\text{subject to}
\quad &\rho\partial_{tt}y-\partial_x(\kappa\partial_xy)=u\quad \text{for}~(t,x)\in I\times\Omega=(0,T)\times (0,L)
\end{aligned}
\end{equation}
with additional initial and boundary conditions. The functional $F$
is given by a quadratic tracking functional involving $y|_{I\times\Omega}$, $y(T,\cdot)|_\Omega$ and $\partial_ty(T,\cdot)|_\Omega$. The regularization functional $R$ and the control space $\M$ are chosen in a way such that $\M$ contains point sources of the desired form and $R$ promotes controls of such a form, i.e. linear combinations of point sources with time-dependent intensities or more general controls with a small spatial support. Since problem \eqref{proto_prob} is convex, one {does not} need to deal with several local minima. However, it is not longer guaranteed that the solution consists of a sum of point sources.
We enforce such controls via the regularization functional $R$. Problems of the form \eqref{proto_prob} (also involving other PDEs) have been analysed from theoretical, numerical and algorithmic points of view, see
\cite{CasasClasonKunisch2013,CasasClasonKunisch:2012,ClasonKunisch:2011a,
ClasonKunisch:2011b,PieperVexler:2013,KunischPieperVexler2014,KunischTrautmannVexler14,CasasKunisch14,CasasKunisch15,
BrediesPikkarainen:2013,
PieperHelm2016,CasasZuazua13,CasasVexlerZuazua13}.
Optimal control problems governed by the linear wave equation were discussed in several different aspects, see \cite{KunischWachsmuth13,KroenerKunisch14,KroenerKunischVexler11,LasieckaSokolowski91,Zuazua04,GugatTrelatZuazua16,GugatKeimerLeugering09,
MordukhovichRaymond04,MordukhovichRaymond05,Kroener11,Kroener13}.
In our particular case we consider the control spaces  {$\M$ of measure-valued functions ${\LwM}$
and vector measures $\ML$ with $R(u)=\alpha\|u\|_{\M}$.}
These two different choices imply different structural properties of the optimal controls.
{A typical {non-regular} element from the space $\ML$
is given by
\begin{equation}
 u=\sum_{i=1}^nu_i(t)\delta_{x_i},\quad u_i\in L^2(I),~x_i\in \Omega,
\label{control}
\end{equation}
where $\delta_{x_i}$ are the Dirac delta functions.}
Point sources {of such type}  with fixed positions and time-dependent intensities are of interest in acoustics or geology, see \cite{KunischTrautmannVexler14,PieperHelm2016}.
If one is interested in controls involving moving point sources of the form
\begin{equation}
u=\sum_{i=1}^nu_i(t)\delta_{x_i(t)},\ \ u_i\in L^2(I),\ \ x_i: I\to\Omega\ \ \text{is measurable},
\label{control 2}
\end{equation}
{then the control space ${\LwM}$ rather than $\ML$ is more appropriate.}
{The space $\ML$ and the functional $\|\cdot\|_{\ML}$ are also related to the term directional sparsity resp. joint sparsity, see \cite{HerzogStadlerWachsmuth2012,FornasierRauhut08}.}

{For the discretization of optimal control problem~\eqref{proto_prob}, we discretize the state equation by space-time finite element method as introduced in \cite{Zlotnik94}.
Related methods are also discussed and analyzed in \cite{Bales94,FrenchPeterson96}, see also \cite{BangerthGeigerRannacher10}.
The measure-valued control is not directly discretized, cf.
the variational control discretization from \cite{Hinze:2005}. However, there exists optimal controls consisting of Dirac measures in the spatial grid points which can be computed, see also \cite{CasasClasonKunisch2013,KunischPieperVexler2014}. The numerical analysis of the control problem is based on {FEM} error estimates for the {second order hyperbolic equations}
from \cite{Zlotnik94} and techniques developed in \cite{CasasClasonKunisch2013,KunischPieperVexler2014}.
It requires to overcome significant technical difficulties caused by non-smoothness of controls and states.
To the best of our knowledge, this is the first paper providing such numerical analysis for the studied control problems.
}
\par The problem like \eqref{proto_prob} for {a parabolic/heat state equation is analyzed {for the case} $\M=\ML$ in \cite{KunischPieperVexler2014} and
{for the case} $\M={\LwM}$ in \cite{CasasClasonKunisch2013}.
In particular, {in both papers} the authors prove existence of optimal controls and derive optimality conditions and {FEM} error estimates.} Our analysis is partly based on these results of \cite{KunischPieperVexler2014}.
In \cite{KunischTrautmannVexler14} a problem {similar to~\eqref{proto_prob}} involving the linear wave equation with constant coefficients as state equation is analyzed. In particular, existing regularity results for a Dirac right{-}hand side are extended to sources from $\ML$. Based on these regularity results existence of optimal controls is proved as well as optimal conditions are derived in the 3D case.
\par {Now we briefly}
sum up the contents of this work. First of all we collect and {partially prove required}
existence and regularity results for the linear wave equation in the 1D setting. In particular, we check that the notions of a weaker solution defined in \cite{Zlotnik94} and more commonly used very weak solution, e.g. \cite{lions1972non}, are equivalent. Most importantly we prove that the solution of
the linear wave equation with variable coefficients from $H^1(\Omega)$ for any source term $u\in \ML$ is an element of {$\mathcal C(\bar I,H^1_0(\Omega))\cap \mathcal C^1(\bar I,L^2(\Omega))$ provided that the initial data have relevant regularity.}
The proof is based on a non-standard energy type bound in space, not only in time, cf. \cite{Lions1987,Fabre1994}.
In \cite{KunischTrautmannVexler14} the same result is proved {for the wave equation with constant coefficients} using duality techniques. {This proof in~\cite{KunischTrautmannVexler14} provides also corresponding results for multidimensional case but can not be directly extended for treating variable coefficients.
{This is due to the fact that it uses estimates of the solution of the wave equation in the whole space with a Dirac measure on the right hand-side which are proven using the Fourier-and Laplace-transformation or explicit solution formulas.}}

The existence of optimal controls and the derivation of optimality conditions are discussed on the basis of results from \cite{KunischTrautmannVexler14,KunischPieperVexler2014}. In the case $\M=\ML$ we prove that
the optimal control $\bar u$ belongs to {$\mathcal C^1(\bar I,\m)$.}
\par {Further, the FEM} discretization of the state equation is introduced. The state variable
$y_{h,\tau}$ belongs to the space of bilinear finite elements and is defined by the regularized Galerkin method.
The resulting numerical scheme is a three-level method in time {(i.e., its main equation relates the approximate solution values at three consecutive grid time levels)}.
Moreover, we {pose and prove the FEM error estimates in $\mathcal C(\bar I,L^2(\Omega))$ for} the discrete state equation which we need for the numerical analysis of the control problem.
We base this study mainly on the results from \cite{Zlotnik94} concerning error analysis of {FEMs} for the second order hyperbolic equations in the classes of the data having integer Sobolev or fractional Nikolskii order of smoothness.
Note that their sharpness in a strong sense was stated in \cite{Zlotnik92}.

Then we consider a semi-discrete optimal control problem in which the continuous state equation is replaced by its discretized version whereas the controls are not discretized. We
{prove convergence of the discrete optimal controls to the continuous one and}
derive optimality conditions based on the Lagrange techniques. Most importantly we derive the discrete adjoint state equation.
We can conclude that the first-discretize-then-optimize and first-optimize-then-discretize approaches commute.
Therefore {an analysis of the discrete adjoint state equation including the error estimates in
$\mathcal C(\bar{I}\times\bar{\Omega})$ and $L^2(I,\mathcal{C}_0(\Omega))$} can also be based on techniques
from \cite{Zlotnik94}. Then we use results from \cite{KunischPieperVexler2014} to represent the numerical error of state variable and of the cost functional in terms of {FEM} errors of the state equation and the adjoint state equation.
Let $\bar u$ and $\bar y$ be the optimal control and the corresponding optimal state, and the variables
$\bar u_{\tau,h}$ and $\bar y_{\tau,h}$
be their discrete counterparts. As the main result of this paper
we prove the error estimates
\[
\|\bar y-\bar y_{\tau,h}\|_{L^2(I\times \Omega)}=\mathcal O\big(({\tau+h})^\alpha\big),\quad
|J(\bar u)-J(\bar u_{\tau,h})|=\mathcal O\big({(\tau+h)^{2/3}}
\big)
\]
{where $\tau$ is the step in time, $h$ is the maximal step in space and}
$\alpha=1/3$ for $\M={\LwM}$ or $\alpha=2/3$ for $\M=\ML$.
The latter higher order is due to the above mentioned improved regularity results for the state and optimal control.
{ Such estimates are proved for the measure-valued controls in the hyperbolic case for the first time.
Similar estimates are impossible in multidimensional settings due to much less fractional Sobolev regularity of optimal states and controls.}
\par Finally we discuss the numerical computation of the discrete control $\bar u_{h,\tau}$. Based on a control discretization $u_{h,\tau}$ that given by the sum {like \eqref{control} with $x_i$ at the spatial grid points and $u_i$ in the space of linear finite elements,}
a solution of the semi-discrete control problem can be calculated similarly to \cite{KunischPieperVexler2014}. For the actual numerical computation of the optimal control we add the term $(\gamma/2)\|u\|_{L^2(I\times\Omega)}^2$, $\gamma>0$, to \eqref{proto_prob}. This regularized problem is solved by a semi-smooth Newton method, see \cite{Pieper:2015}. In a continuation strategy the regularization parameter $\gamma$ is made sufficiently small. We {complete
this work with a numerical example for $\M=\ML$.
}
\par The paper is organized in the following way. In Section \ref{Problem setting} we introduce the problem setting and the control spaces resp. the regularization functionals. Section \ref{sec:regularity} is concerned with regularity properties of the linear wave equation with variable coefficients in the 1D setting. In Section \ref{sec:existence_optimality} the control problem is analyzed from a theoretical point of view.
Section \ref{sec:disc state} deals with discretization of the state equation.
Then we obtain stability bounds and error estimates
for the discrete state equation in Section \ref{sec:error_state}.
Section \ref{sec:disc prob} is concerned with the analysis of the semi-discrete optimal control problem. The next section discusses stability bounds and error estimates for the discrete adjoint sate equation.
In Sections \ref{sec:error_opt_state} resp. \ref{sec:error_cost_func} error estimates for the optimal state and cost functional are derived being the main theoretical results of the study.
Section \ref{tistfo} deals with the time stepping formulation of the discrete state equation.
{In Section \ref{sec:control disc}}
we discuss the control discretization with Dirac measures {at} the grid points.
Then we introduce the $L^2(I\times\Omega)$ regularized problem and describe its solutions by a semi-smooth Newton method.
Finally Section \ref{sec:numerics} provides a numerical example.
\section{Problem setting}
\label{Problem setting}
We consider optimal control problems of the following form
\begin{equation}\label{measure_control_problem}\tag{$\mathcal P$}
J(y,u)=F(y)+\alpha\|u\|_{\M}\to\min_{u,y}
\end{equation}
with the parameter $\alpha>0$ and the tracking functional
\[
 F(y):=\half\big(\left\|y-z_1\right\|_{L^2(I,{H_\rho})}^{2}
+\left\|y(T)-z_2\right\|_{{H_\rho}}^{2}+\left\|\rho\partial_ty(T)-z_3\right\|_{{\mathcal{V}_\kappa^*}}^{2}\big)
\]
using $\mathbf{z}:=(z_1,z_2,z_3)\in\mathcal{Y}:=\L\times \l\times H^{-1}(\Omega)$,
subject to \textit{the state equation} which is an initial-boundary value problem for {a 1D linear wave equation with variable coefficients}
\begin{equation}\label{state_equation}
\left\{\begin{aligned}
\rho\partial_{tt}y-\partial_x(\kappa\partial_{x}y)&=u&&\text{in}~I\times \Omega:=(0,T)\times(0,L)
\\
y&=0&&\text{on}~I\times\partial\Omega\\
y=y^0,~\partial_t y&=y^1&&\text{in}~\{0\}\times\Omega.\\
\end{aligned}\right.
\end{equation}
Here, in particular, the initial data
$\mathbf{y}:=(y^0,y^1)\in H^1_0(\Omega)\times L^2(\Omega)$, and $L>0$ and $T>0$.
The
coefficients $\rho,\kappa\in {H^1(\Omega)}$
satisfy $\rho(x)\geq\nu>0$ and $\kappa(x)\geq\nu$ on $\Omega$.
\par For brevity we denote $H=\l$, $V=H^1_0(\Omega)$, $V^2=H^2(\Omega)\cap V$ and $V^3=\{v\in V|\partial_{x}(\kappa \partial_x v)\in V\}$ equipped with the norms
\[
 \|\cdot\|_{V}=\|\partial_x\cdot\|_H,\ \ \|\cdot\|_{V^2}=\|\partial_{xx}\cdot\|_H,\ \
 {\|\cdot\|_{V^3}=\|\partial_x(\kappa \partial_x\cdot)\|_V.}
\]
Moreover, we utilize the equivalent coefficient-dependent {Hilbert} norms on $H$, $V$, $V^\ast$ and $\mathcal Y$
\begin{gather*}
 \|w\|_{{H_\rho}}=\|\sqrt{\rho}w\|_{H},\quad
 \|w\|_{{\mathcal{V}_\kappa}}=\|\sqrt{\kappa}\partial_xw\|_{H},\quad
 \|w\|_{{\mathcal{V}_\kappa^*}}=\sup_{\|v\|_{{\mathcal{V}_\kappa}}\leq 1}\langle w,v\rangle_{\Omega},
\\
\|\mathbf{z}\|_{\mathcal{Y}}=\big(\|z_1\|_{L^2(I,{H_\rho})}^2
 +\|z_2\|_{{H_\rho}}^2+\|z_3\|_{{\mathcal{V}_\kappa^*}}^2\big)^{1/2},
\end{gather*}
where $\langle\cdot,\cdot\rangle_{\Omega}$ is the duality relation on $V^*\times V$.
\par For the control space $\M$ we consider two choices, either the space of vector measures $\mathcal M(\Omega,L^2(I))$ or the space of weak{-star} measurable, $\m$-valued functions {$\LwM$}. 
\Blue{Recall that $\|u\|_{\LwM}=\|\|u(\cdot)\|_{\m}\|_{L^2(I)}$ where
$\|u(\cdot)\|_{\m}\in L^2(I)$ for any $u\in \LwM$.}
%{\color{green} Note that the function $t\mapsto \|u(t)\|_{\m}$ is an element of $L^2(I)$ for a $u\in \LwM$.}
Let correspondingly $\C$ be chosen as $\mathcal C_0(\Omega,L^2(I))$ or $L^2(I,\mathcal C_0(\Omega))$
where $\mathcal{C}_0(\Omega)=\{v\in \mathcal{C}(\bar{\Omega})|\,v|_{x=0,L}=0\}$.
The following identifications of dual spaces hold
\[
 \mathcal C_0(\Omega,L^2(I))^*\cong\mathcal M(\Omega,L^2(I)),\ \
 L^2(I,\mathcal C_0(\Omega))^*\cong
 {\LwM},
\]
with the duality pairings \Blue{respectively
\begin{multline*}
\langle u,v\rangle_{\M,\C}:=\int_\Omega\int_0^Tv(x,t)\,\mathrm du(x)\,\mathrm dt,\ \
\langle u,v\rangle_{\M,\C}:=\int_0^T\int_\Omega v(t,x)\,\mathrm du(t)\,\mathrm dt\\
%\text{resp.}\quad
%\langle u,v\rangle_{\M,\C}:=\int_\Omega\int_0^Tv(x,t)\,u'(x,t)\,\mathrm d|u|\,\mathrm dt
\end{multline*}
for any $u\in \M$ and $v\in \C$.}
%{, where $|u|$ is the total variation measure of $u$ and $u'\in L^1(\Omega,|u|,L^2(I))$ its Randon-Nikodym derivative. In particular, there holds $\M=\C^\ast$.} 
\Blue{See \cite{CasasClasonKunisch2013,CembranoMendoza97,Edwards65,KunischPieperVexler2014},
where more details on the properties of these spaces can be found.}
%Tulcea69} 
%resp. \cite{KunischPieperVexler2014}, 
%where more details on the properties of these spaces and the norm $\|\cdot\|_{\M}$ can be found.
In particular, the following embeddings hold
\begin{equation}
\label{embedding1}
 \mathcal M(\Omega,L^2(I))\hookrightarrow {\LwM}\hookrightarrow L^2(I,V^*).
\end{equation}
\section{Existence and regularity of the state}\label{sec:regularity}
\subsection{Weak formulations and preliminary existence, uniqueness and regularity results}
In this section we introduce our solution concepts for the state equation \eqref{state_equation}.
We begin with defining a weak formulation of \eqref{state_equation}.
\begin{definition}\label{def:weak1}
Let $(u,y^0,y^1)\in X\times V\times H$ with $X=\L$ or $H^1(I,V^\ast)$ or $\ML$. Then $y\in \mathcal{C}(\bar{I},V)\cap \mathcal{C}^1(\bar{I},H)$ is called a \textit{weak solution} of \eqref{state_equation} if it satisfies
the integral identity
\begin{equation}
 {B(y,v)}+\big(\rho\partial_ty(T),v(T)\big)_{H}
 =\int_I \langle u,v\rangle_{\Omega}~\mathrm dt+\big(\rho y^1,v(0)\big)_{H}~~\forall v\in L^2(I,V)\cap H^1(I,H)
\label{int_id1}
\end{equation}
{with the indefinite symmetric bilinear form
\begin{equation}
 B(y,v):=-(\rho\partial_ty,\partial_t v)_{\L}
 +(\kappa\partial_xy,\partial_xv)_{L^2(I\times\Omega)},
\label{dif bilinear_form}
\end{equation}
}
and the initial condition $y(0)=y^0$.
\end{definition}
The right-hand side in \eqref{int_id1} is well defined for $X=\ML$ too due to embeddings \eqref{embedding1}.
\begin{remark}\label{rem:vT_zero}
It is possible (and more common) to suppose that $v(T)=0$ in \eqref{int_id1}
when the last term on the left disappears (for example, see \cite{Zlotnik94}).
This leads to an equivalent formulation.
To check this, it is enough to replace there $v$ by $v\beta_\delta$, where $\beta_\delta(t)=\min\big(1,(T-t)/\delta\big)$, $0<\delta<T$.
Then $\partial_t(v\beta_\delta)=(\partial_tv)\beta_\delta-(1/\delta)v\chi_{(T-\delta,T)}$, where $\chi_{(T-\delta,T)}$ is the characteristic function of $(T-\delta,T)$.
Passing to the limit as $\delta\to 0$ with the help of the dominated convergence theorem and the properties of $y$ and $v$ leads to the result.
\end{remark}
\par Another definition of the weak solution is possible.
\begin{definition}\label{def:weak2}
Let $(u,y^0,y^1)\in X\times V\times H$ with $X=\L$ or $H^1(I,V^\ast)$ or $\ML$. A function $y\in \mathcal C(\bar I,V)\cap H^2(I,V^\ast)\hookrightarrow H^1(I,H)$ is called a weak solution of \eqref{state_equation} if it satisfies
\begin{equation}\label{weak_form_lions}
\int_I\langle\rho\partial_{tt}y,v\rangle_{\Omega}+(\kappa\partial_xy,\partial_xv)_H~\mathrm dt=\int_I\langle u,v\rangle_{\Omega}~\mathrm dt\quad~\forall v\in L^2(I,V)
\end{equation}
and $y(0)=y^0$ as well as $\partial_ty(0)=y^1$.
\end{definition}
\begin{proposition}
Definitions \ref{def:weak1} and \ref{def:weak2} (up to the property $y\in\mathcal{C}^1(\bar{I},H)$) are equivalent.
\end{proposition}
\begin{proof}
The weak solution from Definition \ref{def:weak1} has $\partial_{tt}y\in L^2(I,V^*)$ according to the integral identity \eqref{int_id1}.
Then the equivalence of \eqref{weak_form_lions} and \eqref{int_id1} can be proved using integration by parts in time and the density of $\mathcal C^\infty(\bar I,V)$ in $L^2(I,V)\cap H^1(I,H)$, cf. \cite[Chapter 1, Theorem 2.1]{lions1972non}.
\end{proof}
\begin{proposition}\label{prop:exist weak}
\begin{enumerate}
\item Let $(u,y^0,y^1)\in X\times V\times H$ with $X=\L$ or $H^1(I,V^\ast)$. Then \eqref{state_equation} has a unique weak solution satisfying
$y\in \mathcal{C}(\bar{I},V)\cap \mathcal{C}^1(\bar{I},H)\cap H^2(I,V^\ast)$
and
\begin{gather}
 \|y\|_{\mathcal{C}(\bar{I},V)}+\|\partial_ty\|_{\mathcal{C}(\bar{I},H)}+\|\partial_{tt}y\|_{L^2(I,V^\ast)}
 \leq c\,\big(\|u\|_{X}+\|\mathbf{y}\|_{V\times H}\big).
\label{pset13}
\end{gather}
Hereafter $c>0$, $c_1>0$, etc., are independent of $y$ and the data.
\par In the case $X=H^1(I,V^\ast)$ there even holds $y\in \mathcal C^2(\bar I,V^\ast)$ as well as
\[
\|\partial_{tt}y\|_{\mathcal C(\bar I,V^\ast)}\leq c\,\big(\|u\|_{H^1(I,V^\ast)}+\|\mathbf{y}\|_{V\times H}\big).
\]
\item Let $(u,y^0,y^1)\in X \times V^2\times V$ with $X=L^2(I,V)$ or $H^1(I,H)$. Then the weak solution y satisfies
$y\in \mathcal{C}(\bar{I},V^2)\cap \mathcal{C}^1(\bar{I},V)\cap H^2(I,H)$
and
\begin{equation}
 \|y\|_{\mathcal{C}(\bar{I},V^2)}+\|\partial_ty\|_{\mathcal{C}(\bar{I},V)}+\|\partial_{tt}y\|_{L^2(I,H)}\leq c\,\big(\|u\|_{X}
 +\|\mathbf{y}\|_{V^2\times V}\big).
\label{pset15}
\end{equation}
\par In the case $X=H^1(I,H)$ there even holds $y\in \mathcal C^2(\bar I,H)$ as well as
\[
\|\partial_{tt}y\|_{\mathcal C(\bar I,H)}\leq c\,\big(\|u\|_{H^1(I,H)}+\|\mathbf{y}\|_{V^2\times V}\big).
\]
Moreover, $y$ satisfies the equation $\rho\partial_{tt}y-\partial_x(\kappa\partial_{x}y)=u$ in $L^2(I\times\Omega)$,
i.e. it is the strong solution.
\end{enumerate}
\end{proposition}
\begin{proof}
For example, see \cite[Propositions 1.1 and 1.3]{Zlotnik94}.
\end{proof}
Item 2 ensures the regularity of weak solution for more regular data.
\par For less regular data $(u,y^0,y^1)\in L^2(I,V^\ast)\times H\times V^\ast$ one can use other weak formulations.
To state the first of them, we define the integration operator
$(\mathcal{I}_tv)(t):=\int_0^tv(s)~\mathrm ds$ and its adjoint $(\mathcal{I}_t^\ast v)(t):=\int_t^Tv(s)~\mathrm ds$ on $\bar{I}$.
\begin{definition}\label{def:weaker1}
Let $(u,y^0,y^1)\in L^2(I,V^\ast)\times H\times V^\ast$. A function $y\in \mathcal{C}(\bar{I},H)$ with $\mathcal{I}_ty\in \mathcal{C}(\bar{I},V)$ is called a \textit{weaker solution} of \eqref{state_equation} if it satisfies
\begin{multline}
 \int_I-(\rho y,\partial_t v)_H+(\kappa\partial_x\mathcal{I}_ty,\partial_xv)_H~\mathrm dt
 +\big(\rho y(T),v(T)\big)_{H}\\
 =\int_I\big\langle u,\mathcal{I}_t^\ast v\big\rangle_{\Omega}~\mathrm dt
 +\big(\rho y^0,v(0)\big)_{H}
 +\langle \rho y^1, (\mathcal{I}_t^\ast v)(0)\rangle_{\Omega}\quad \forall v\in L^2(I,V)\cap H^1(I,H).
\label{int_id2}
\end{multline}
\end{definition}
As in the case of Definition \ref{def:weak1},
it is sufficient to take $v(T)=0$ in \eqref{int_id2}, cf. Remark \ref{rem:vT_zero}.
\begin{proposition}\label{prop:exist weaker}
Let $(u,y^0,y^1)\in L^2(I,V^\ast)\times H\times V^\ast$. Then there exists a unique weaker solution $y\in \mathcal C(\bar I,H)\cap \mathcal{C}^1(\bar{I},V^\ast)$ and it satisfies the bound
\begin{equation*}
 \|y\|_{\mathcal{C}(\bar{I},H)}+\|\mathcal{I}_ty\|_{\mathcal{C}(\bar{I},V)}+\|\partial_ty\|_{\mathcal{C}(\bar{I},V^\ast)}\leq c\,\big(\|u\|_{L^2(I,V^\ast)}+\|\mathbf{y}\|_{H\times V^*}\big).
\label{pset11}
\end{equation*}
\end{proposition}
\begin{proof}
See \cite[Proposition 1.2]{Zlotnik94}.
\end{proof}
\par We infer that there are other weak formulations of \eqref{state_equation} for solutions $y\in \mathcal C(\bar I,H)\cap \mathcal C^1(\bar I,V^\ast)$. One can use the concept of very weak solutions.
\begin{definition}\label{def:weaker2}
Let $(u,y^0,y^1)\in L^2(I,V^\ast)\times H\times V^\ast$. A function $y\in \mathcal C(\bar I,H)\cap \mathcal C^1(\bar I,V^\ast)$ satisfying
\begin{multline}
 \int_I\big(y,\rho\partial_{tt} v-\partial_x(\kappa\partial_xv)\big)_H~\mathrm dt
 -\big(\rho y(T),\partial_tv(T)\big)_{H}+\langle \rho\partial_ty(T),v(T)\rangle_{\Omega}\\
 =\int_I \langle u,v\rangle_{\Omega}~\mathrm dt
 -\big(\rho y^0,\partial_tv(0)\big)_{H}+\langle\rho y^1,v(0)\rangle_{\Omega}
\label{int_id4}
\end{multline}
for any $v\in L^2(I,V^2)\cap H^2(I,H)\hookrightarrow H^1(I,V)$ is called a \textit{very weak solution} of \eqref{state_equation}.
\end{definition}
Actually, these two {last} solution concepts are equivalent for the considered data spaces.
\begin{theorem}
Definitions \ref{def:weaker1} and \ref{def:weaker2} are equivalent.
\end{theorem}
\begin{proof}
First of all, we consider the auxiliary integrated in $t$ problem \eqref{state_equation}:
\begin{equation}\label{equation_primitive}
\left\{\begin{aligned}
\rho\partial_{tt}\tilde y-\partial_x(\kappa\partial_{x}\tilde y)&=\mathcal I_tu+\rho y^1&&\text{in}~I\times \Omega\\
\tilde y&=0&&\text{on}~I\times\partial\Omega\\
\tilde y=0,~\partial_t \tilde y&=y^0&&\text{in}~\{0\}\times\Omega\\
\end{aligned}\right.
\end{equation}
for $(u,y^0,y^1)\in L^2(I,V^\ast)\times H\times V^\ast$. Thus, we have $\mathcal I_tu\in H^1(I,V^\ast)$. According to Proposition \ref{prop:exist weak} problem \eqref{equation_primitive} has a unique weak solution $\tilde y\in \mathcal C(\bar I,V)\cap \mathcal C^1(\bar I,H)$. Moreover, we set $y=\partial_t\tilde y$. Thus the weak formulation of \eqref{equation_primitive} involving $\tilde y$ coincides with the weaker formulation of \eqref{state_equation} involving $y$. Furthermore there holds $y=\partial_t \tilde y\in \mathcal C(\bar I,H)$ and
\begin{equation}\label{derivative_primitive}
\partial_ty
=\partial_{tt}\tilde y=(1/\rho)\big(\mathcal I_tu+\partial_x(\kappa\partial_x\tilde y)+\rho y^1\big)\in \mathcal C(\bar I,V^\ast).
\end{equation}
\par Now we take any $v\in \mathcal C^\infty(\bar I,V^2)$
and test \eqref{int_id2} with $-\partial_tv$ in the role of $v$:
\begin{multline*}
 \int_I(\rho y,\partial_{tt} v)_H
 -\big(\kappa\partial_x\mathcal{I}_ty,\partial_x\partial_tv\big)_H~\mathrm dt
 -\big(\rho y(T),\partial_tv(T)\big)_{H}\\
 =\int_I\big\langle u,-\mathcal I^\ast_t(\partial_tv)\big\rangle_{\Omega}~\mathrm dt
 -\big(\rho y^0,\partial_tv(0)\big)_{H}+\langle \rho y^1,-(\mathcal I^\ast_t\partial_tv)(0)\rangle_{\Omega}.
\end{multline*}
Next we rearrange a term on the left integrating by parts in $x$ and $t$:
\begin{multline}
-\int_I\big(\kappa\partial_x\mathcal{I}_ty,\partial_x\partial_tv\big)_H~\mathrm dt
=\int_I\big(\mathcal I_ty,L\partial_tv\big)_H~\mathrm dt\\
=-\int_I\big(y,Lv\big)_H~\mathrm dt
+\big((\mathcal I_ty)(T),Lv(T)\big)_H
\label{aux equal}
\end{multline}
with $Lv:=\partial_x(\kappa\partial_xv)$. Since $\mathcal I_ty\in \mathcal C(\bar I,V)$, we get
\begin{multline*}
 \int_I\big(y,\rho\partial_{tt} v-\partial_x(\kappa\partial_xv)\big)_H~\mathrm dt
 -\big(\rho y(T),\partial_tv(T)\big)_{H}+\langle \partial_x(\kappa\partial_x\mathcal{I}_ty)(T),v(T)\rangle_{\Omega}\\
 =\int_I\big\langle u,v\big\rangle_{\Omega}~\mathrm dt-\langle (\mathcal I_tu)(T),v(T)\rangle_{\Omega}
 -\big(\rho y^0,\partial_tv(0)\big)_{H}+\langle \rho y^1,v(0)\rangle_{\Omega}-\langle \rho y^1,v(T)\rangle_{\Omega}.
\end{multline*}
Since formula \eqref{derivative_primitive} implies that
\[
 \langle\rho\partial_ty(T),\varphi\rangle_{\Omega}=((\mathcal I_tu)(T),\varphi)_H
 +\langle\partial_x(\kappa\partial_x\mathcal{I}_ty),\varphi\rangle_{\Omega}+\langle\rho y^1,\varphi\rangle_{\Omega}\ \ \forall
 \varphi\in V,
\]
by the density of $\mathcal C^\infty(\bar I,V^2)$ in $L^2(I,V^2)\cap H^2(I,H)$
we find that $y$ is a very weak solution of \eqref{state_equation}.
\par Now let $y\in \mathcal C(\bar I,H)\cap \mathcal C^1(\bar I,V^\ast)$ be a very weak solution of \eqref{state_equation}. Then we take any $v\in \mathcal C^\infty(\bar I,V^2)$ and test \eqref{int_id4} with $\mathcal I_t^\ast v$. Thus, we get
\begin{multline}
 \int_I\big(y,-\rho\partial_{t}v)_H-(y,\partial_x(\kappa\partial_x\mathcal I_t^\ast v)\big)_H~\mathrm dt+\big(\rho y(T),v(T)\big)_{H}
\\
 =\int_I\big\langle u,\mathcal{I}_t^\ast v\big\rangle_{\Omega}~\mathrm dt
 +\big(\rho y^0,v(0)\big)_{H}+\langle \rho y^1, (\mathcal{I}_t^\ast v)(0)\rangle_{\Omega}
\label{aux equal 2}
\end{multline}
and then
\begin{multline*}
\int_I(y,-\rho\partial_{t}v)_H+(\mathcal I_ty,-\partial_x(\kappa\partial_xv))_H~\mathrm dt+\big(\rho y(T),v(T)\big)_{H}\\
=\int_I\big\langle \mathcal I_tu+\rho y^1,v\big\rangle_{\Omega}~\mathrm dt+\big(\rho y^0,v(0)\big)_{H}.
\end{multline*}
The last equation yields that $L\mathcal I_ty=-\rho\partial_t y+\mathcal I_tu+\rho y^1\in \mathcal C(\bar I,V^\ast)$. Thus $\mathcal I_ty\in \mathcal C(\bar I,V)$
and we can transform a term on the left in \eqref{aux equal 2} by replacing $v$ by $\mathcal I_t^\ast v$ in \eqref{aux equal}:
\begin{equation*}
\int_I\big(y,-\partial_x(\kappa\partial_x\mathcal I_t^\ast v)\big)_H~\mathrm dt=\int_I\big(\kappa\partial_x\mathcal I_ty,\partial_xv\big)_H~\mathrm dt.
\end{equation*}
Then the density of $\mathcal C^\infty(\bar I,V^2)$ in $L^2(I,V)\cap H^1(I,H)$ shows that $y$ is a weaker solution of \eqref{state_equation}.
\end{proof}
Moreover, there is the concept of solutions by transposition.
\begin{definition}\label{def:weaker3}
Let $(u,y^0,y^1)\in L^2(I,V^\ast)\times H\times V^\ast$. A \textit{solution by transposition} $y\in \mathcal C(\bar I,H)\cap \mathcal C^1(\bar I,V^\ast)$ of \eqref{state_equation} is defined by
\begin{multline}
 \int_I(\rho y,\phi)_H~\mathrm dt-\big(\rho y(T),p^1\big)_{H}
 +\langle \rho\partial_ty(T),p^0\rangle_{\Omega}\\
 =\int_I \langle u,p\rangle_{\Omega}~\mathrm dt+\langle\rho y^1,p(0)\rangle_{\Omega}-\big(\rho y^0,\partial_tp(0)\big)_{H}
\label{solution_transposition}
\end{multline}
for all $(\phi,p^0,p^1)\in L^2(I\times\Omega)\times V\times H$
where $p\in \mathcal C(\bar I,V)\cap \mathcal C^1(\bar I,H)$ is the weak solution of \textit{the adjoint problem} \begin{equation}\label{dual_state_equation}
\left\{\begin{aligned}
\rho\partial_{tt}p-\partial_x(\kappa\partial_{x}p)&=\rho\phi&&\text{in}~I\times \Omega\\
p&=0&&\text{on}~I\times\partial\Omega\\
p=p^0,~\partial_t p&=p^1&&\text{in}~\{T\}\times\Omega\\
\end{aligned}\right.
\end{equation}
\end{definition}
\begin{proposition}
Definitions \ref{def:weaker2} and \ref{def:weaker3} are equivalent too.
\end{proposition}
\begin{proof}
For $\phi \in H^1(I,H)$ or $L^2(I,V)$, $p^0\in V^2$ and $p^1\in V$ there holds $p\in \mathcal C(\bar I,V^2)\cap \mathcal C^1(\bar I,V)\cap H^2(I,H)$, see Proposition \ref{prop:exist weak}. Due to the density of $H^1(I,H)$ resp. $L^2(I,V)$ in $L^2(I\times \Omega)$ as well as $V^2$ in $V$ and $V$ in $H$ a very weak solution is a solution by transposition. Now let $p\in \mathcal C^{\infty}(\bar I,V^2)$ and set $\phi=\partial_{tt}p-(1/\rho)\partial_x(\kappa\partial_{x}p)\in \mathcal C^{\infty}(\bar I,H)$, $p^0=p(T)\in V^2$ and $p^1=\partial_tp(T)\in V^2$. Thus $p$ is the solution of \eqref{dual_state_equation}. Then the density of $\mathcal C^{\infty}(\bar I,V^2)$ in $L^2(I,V^2)\cap H^2(I,H)$ implies that a solution by transposition is a very weak solution.
\end{proof}
\begin{remark}
For $(u,y^0,y^1)\in L^2(I,H)\times V\times H$, the weaker solution coincides with the weak one.
\end{remark}
\subsection{Existence and regularity of the state}
In this section we study the existence, uniqueness and regularity of solution of the state equation for measure valued source terms. We will carry out the analysis for both control spaces.
We use the distinct properties of each space in order to show improved regularity of the state.
\subsubsection{The control space $\mathcal M(\Omega,L^2(I))$}
\label{pros1}
The space $\mathcal M(\Omega,L^2(I))$ is not so broad as ${\LwM}$ and contains no moving point sources
but
contains the standing $\delta$-sources \eqref{control}.
Therefore, we expect that the state has better regularity properties in this case and prove that
$y\in \mathcal C(\bar{I},V)\cap \mathcal C^1(\bar{I},H)$.
The proof will be based on a priori bound and a density argument. First we state the following density result.
\begin{lemma}\label{weakstardensity}
Let $u\in \ML$. Then there exists a sequence
\linebreak
$\{u_n\}\subset {C_c^\infty(\Omega,L^2(I))}$
such that
\begin{equation}\label{weakstarseq}
u_n\rightharpoonup^\ast u~~\text{in}~~\ML~~\text{as}~~n\to\infty,~~\|u_n\|_{\ML}\leq\|u\|_{\ML}~\forall n\geq 1.
\end{equation}
\end{lemma}
\begin{proof}
We denote by $X$ the locally convex space $\ML$ endowed with its weak-star topology and define the absolutely convex set
\[
 E=\{u\in \mathcal C_c^\infty(\Omega,L^2(I))|\|u\|_{L^1(\Omega,L^2(I))}\leq 1\}\subset X.
\]
\par Assume that \eqref{weakstarseq} is wrong.
Then there exists $u_0\in\ML$ with
\linebreak
$\|u_0\|_{\ML}=1$ such that $u_0\not\in\bar{E}$ where $\bar{E}$ is the closure of $E$ in $X$.
Owing to the corollary of a theorem on the separation of convex sets \cite[Ch. III, Theorem 6]{Kantorovich82} there exists
$v\in \CL$ such that
\begin{multline}\label{weakstarseq 2}
 |\langle u,v\rangle_{\ML,\CL}|\leq 1\ \ \forall u\in E,\\
 1<\langle u_0,v\rangle_{\ML,\CL}\leq \|v\|_{\CL}.
\end{multline}
On the other hand, $\mathcal C_c^\infty(\Omega,L^2(I))$ is dense in $L^1(\Omega,L^2(I))$ thus
\[
 \sup_{u\in E}|\langle u,v\rangle_{\ML,\CL}|=\|v\|_{\CL}
\]
that contradicts \eqref{weakstarseq 2}. Thus $u_0\in\bar{E}$. Since $\CL$ is separable, the weak-star topology on $E$ is metrizable. {Therefore} the closure of $E$ is equal to its sequential closure, see \cite[Theorem 3.28, Corollary 3.30]{Brezis:2011}.
\end{proof}
Note that clearly $C_c^\infty(\Omega,L^2(I))\subset L^2(I,V)$.

\par Preliminarily we prove the following crucial a priori bound.
\begin{lemma}\label{aprioriest}
Let {$(u,y^0,y^1)\in L^1(\Omega,L^2(I))\times V\times H$}
and $y$ be the corresponding strong solution of problem \eqref{state_equation}.
Then $y$ satisfies the following a priori bound
\begin{multline}
 \|y\|_{\mathcal{C}(\bar I,V)}+\|\partial_t y\|_{\mathcal{C}(\bar I,H)}
 +\|\kappa\partial_xy\|_{\mathcal{C}(\bar{\Omega},L^2(I))}+\|\partial_t y\|_{\mathcal{C}_0(\Omega,L^2(I))}\\
 \leq c\big(\|u\|_{L^1(\Omega,L^2(I))}+\|\mathbf{y}\|_{V\times H}\big).
\label{eneq1C}
\end{multline}
\end{lemma}
\begin{proof}
We first remind the energy equality for problem \eqref{state_equation}
\[
 \|\sqrt{\rho}\partial_ty(t)\|_H^2+\|\sqrt{\kappa}\partial_xy(t)\|_H^2
 =\|\sqrt{\rho}y^1\|_H^2+\|\sqrt{\kappa}\partial_xy^0\|_H^2
 +2\mathcal I_t(u(t),\partial_ty(t))_H\ \ \text{on}\ \ I.
\]
After setting
\[
 E(t):=\|\partial_ty(t)\|_H^2+\|\partial_xy(t)\|_H^2,~~ E^0:=\|y^1\|_H^2+\|\partial_xy^0\|_H^2
 \]
and $c_0:=\max\big(\|\rho\|_{L^\infty(\Omega)},\,\|\kappa\|_{L^\infty(\Omega)}\big)$, the energy equality implies
\begin{multline}
 \nu\|E\|_{\mathcal{C}(\bar{I})}\leq c_0E^0 +2\max_{\theta\in\bar{I}}\Big|\int_\Omega\int_0^\theta u\partial_ty~\mathrm dt \mathrm dx\Big|\\
 \leq c_0E^0+2\nu^{-1}\|u\|_{L^1(\Omega,\,L^2(I))}\|\sqrt{\rho\kappa}\partial_ty\|_{\mathcal{C}_0(\Omega,\,L^2(I))}.
\label{eneq3}
\end{multline}
\par We also multiply the equation in \eqref{state_equation} by $-2\kappa\partial_xy$ and integrate over $I$. Integration by parts in $t$ yields the equality
\begin{multline}\label{eneq5}
\rho\kappa\partial_x\Big(\|\partial_ty\|_{L^2(I)}^2\Big)+\partial_x\Big(\|\kappa\partial_xy\|_{L^2(I)}^2\Big)\\
=2\rho\kappa\Big(\partial_ty(T)\partial_xy(T)-y^1\partial_xy^0\Big)-2(u,\kappa\partial_xy)_{L^2(I)}\ \text{on}\ \Omega.
\end{multline}
We define a function $P:=\rho\kappa\|\partial_ty\|_{L^2(I)}^2+\|\kappa\partial_xy\|_{L^2(I)}^2$ on $\Omega$.
Since the left-hand side of \eqref{eneq5} equals $\partial_xP-\big(\partial_x(\rho\kappa)\big)\|\partial_ty\|_{L^2(I)}^2$,
taking the modulus and integrating over any $(a,b)\subset\Omega$ we derive
\begin{multline}
\|\partial_xP\|_{L^1(a,b)}
\leq c_0^2(E(T)+E^0)+2\|u\|_{L^1(\Omega,\,L^2(I))}\|\kappa\partial_xy\|_{\mathcal{C}(\bar{\Omega},\,L^2(I))}\\
+\|\big(\partial_x(\rho\kappa)\big)\|\partial_ty\|_{L^2(I)}^2\|_{L^1(a,b)}\\
\leq c_0^2(\|E\|_{\mathcal{C}(\bar{I})}+E^0)
+2\|u\|_{L^1(\Omega,\,L^2(I))}\|P\|_{\mathcal{C}(\bar{\Omega})}^{1/2}
+\nu^{-2}\|\partial_x(\rho\kappa)\|_{L^1(a,b)}\|P\|_{\mathcal{C}(\bar{\Omega})}.
\label{eneq7}
\end{multline}
Let $x_0\in\bar{\Omega}$ be such that $\|P\|_{\mathcal{C}(\bar{\Omega})}=P(x_0)$ hold and let now $[a,b]\ni x_0$.
Then the mean value theorem for integrals implies
\begin{equation}
 \|P\|_{\mathcal{C}(\bar{\Omega})}\leq(b-a)^{-1}\|P\|_{L^1(a,b)}+\|\partial_xP\|_{L^1(a,b)}.
\label{eneq8}
\end{equation}
By the above definitions we clearly have
\begin{equation}\label{estPE}
 \|P\|_{L^1(\Omega)}\leq c_0^2\|E\|_{L^1(I)}\leq c_0^2T\|E\|_{\mathcal{C}(\bar{I})}.
\end{equation}
Inserting \eqref{eneq7} into \eqref{eneq8} and using \eqref{estPE}, we obtain
\begin{multline}
 \|P\|_{\mathcal{C}(\bar{\Omega})}
 \leq c_0^2\big(1+T(b-a)^{-1}\big)\big(\|E\|_{\mathcal{C}(\bar{I})}+E^0\big)
 +2\|u\|_{L^1(\Omega,\,L^2(I))}\|P\|_{\mathcal{C}(\bar{\Omega})}^{1/2}\\
 +\nu^{-2}{(b-a)^{1/2}\|\rho\kappa\|_{H^1(\Omega)}}\|P\|_{\mathcal{C}(\bar{\Omega})}.
\label{eneq9}
\end{multline}
Owing to \eqref{eneq3} we can write
\begin{equation}
\nu\|E\|_{\mathcal{C}(\bar{I})}\leq c_0^2E^0+2\nu^{-1}\|u\|_{L^1(\Omega,\,L^2(I))}\|P\|_{\mathcal{C}(\bar{\Omega})}^{1/2}.
\label{eneq11}
\end{equation}
Using this in \eqref{eneq9} and choosing a small enough $(a,b)$ such that
\[\nu^{-2}{(b-a)^{1/2}\|\rho\kappa\|_{H^1(\Omega)}}\leq 1/2,\]
we derive
\begin{gather}
 \|P\|_{\mathcal{C}(\bar{\Omega})}\leq c_1\big(E^0+\|u\|_{L^1(\Omega,\,L^2(I))}^2\big).
\label{eneq13}
\end{gather}
Inserting the last bound in \eqref{eneq11}, we also get
\[
\|E\|_{\mathcal{C}(\bar{I})}\leq c_2\big(E^0+\|u\|_{L^1(\Omega,\,L^2(I))}^2\big).
\]
Finally, this yields bound \eqref{eneq1C}.
\end{proof}
\begin{remark}
Lemma \ref{aprioriest} remains valid for $\rho,\kappa\in W^{1,1}(\Omega)$.
Owing to the absolute continuity of the Lebesgue integral we have
$\|\partial_x(\rho\kappa)\|_{L^1(a,b)}\leq\mu(b-a)$, where  $\lim_{\theta\to+0}\mu(\theta)=0$, thus
one can replace $(b-a)^{1/2}\|\rho\kappa\|_{H^1(\Omega)}$ by $\mu(b-a)$ in \eqref{eneq9} and below in the proof.
\end{remark}
\begin{theorem}\label{regML2}
Let $(u,y^0,y^1)\in \ML\times V\times H$. Then there exists a unique weak solution $y$ and it satisfies the bound
\begin{equation}\label{aprioriy}
 \|y\|_{\mathcal C(\bar I,V)}+\|\partial_t y\|_{\mathcal C(\bar I,H)}
 \leq c\,\big(\|u\|_{\mathcal M(\Omega,L^2(I))}+\|\mathbf{y}\|_{V\times H}\big).
\end{equation}
\end{theorem}
\begin{proof}
1. Let first $u=0$. According to Proposition \ref{prop:exist weak} were exists a unique weak solution $y$ of \eqref{state_equation} for any $\mathbf{y}\in V\times H$ and it satisfies
\[
\|y\|_{\mathcal C(\bar I,V)}+\|\partial_t y\|_{\mathcal C(\bar I,H)}\leq c\|\mathbf{y}\|_{V\times H}.
\]
\par 2. Now it suffices to consider the case $y^0=y^1=0$. Let first $u\in\mathcal M(\Omega,H^1(I))\hookrightarrow H^1(I,V^\ast)$ since $\partial_t u\in\ML\hookrightarrow L^2(I,V^\ast)$.
Then according to Proposition \ref{prop:exist weak}
there exists a unique weak solution $y\in \mathcal C(\bar I,V)\cap \mathcal C^1(\bar I,H)$
of \eqref{state_equation} and it satisfies bound \eqref{pset13}. Moreover, it is also a weaker solution.
\par So it remains to prove the bound
\begin{equation}\label{aprioriy1}
 \|y\|_{\mathcal C(\bar I,V)}+\|\partial_t y\|_{\mathcal C(\bar I,H)}
 \leq c\,\|u\|_{\mathcal M(\Omega,\,L^2(I))}~~\text{for}~~ u\in\mathcal M(\Omega,H^1(I)).
\end{equation}
To this end, according to Lemma \ref{weakstardensity} we approximate $u$ by functions
$\{u_n\}\subset L^2(I,V)$
satisfying \eqref{weakstarseq}.
The strong solution $y_n$ of \eqref{state_equation} corresponding to $u=u_n$ satisfies the bound like \eqref{eneq1C} and in particular
\[
 \|y_n\|_{\mathcal C(\bar I,V)}+\|\partial_t y_n\|_{\mathcal C(\bar I,H)}\leq c\,\|u_n\|_{\mathcal M(\Omega,\,L^2(I))}
 \leq c\|u\|_{\ML}.
\]
Therefore there exists a subsequence of $\{y_n\}$ (not relabeled) and $\tilde y\in L^\infty(I,V)\cap W^{1,\infty}(I,H)$
such that $y_n$ converges to $\tilde{y}$ in the weak-star sense of $L^\infty(I,V)\cap W^{1,\infty}(I,H)$.
This is sufficient to pass to the limit in the last bound and in \eqref{int_id2} for $y=y_n$, $u=u_n$ and $v(T)=0$, see Remark \ref{rem:vT_zero}.
Thus $\tilde y$ both satisfies the bound
\[
 \|\tilde{y}\|_{L^\infty(\bar I,V)}+\|\partial_t\tilde{y}\|_{W^{1,\infty}(\bar I,H)}\leq c\,\|u\|_{\mathcal M(\Omega,\,L^2(I))}
\]
and is a weaker solution of \eqref{state_equation}. Due to its uniqueness there holds $\tilde y=y$, and bound \eqref{aprioriy1} is proved.
\par 3. Let now $u\in\mathcal M(\Omega,L^2(I))$ and $y$ be the corresponding weaker solution of \eqref{state_equation}, see Proposition \ref{prop:exist weaker}.
The space $\mathcal M(\Omega,H^1(I))$ is dense in $\mathcal M(\Omega,L^2(I))$, cf. \cite[Proposition 2.1]{KunischTrautmannVexler14}; 
\Blue{this also holds since $\ML$ is the projective closure of the tensor product between $\m$ and $L^2(I)$, see \cite{Raymond},} and  $H^1(I)$ is dense in $L^2(I)$. Thus
%{\color{green} The space $\mathcal M(\Omega,H^1(I))$ is dense in $\mathcal M(\Omega,L^2(I))$, cf. \cite[Proposition 2.1]{KunischTrautmannVexler14} or which follows from the facts that $\ML$ is the projective closure of the tensor product between $\m$ and $L^2(I)$ and the density of $H^1(I)$ in $L^2(I)$, see \cite{Raymond}. Thus} 
there exists
a sequence $\{u_n\}\subset \mathcal M(\Omega,H^1(I))$ such that $u_n\to u$ in $\mathcal M(\Omega,L^2(I))$ as $n\to\infty$.
Let $y_n\in \mathcal C(\bar I,V)\cap \mathcal C^1(\bar I,H)$ be the above weak solution of \eqref{state_equation} corresponding to $u=u_n$. %
Since $\{u_n\}$ is a Cauchy sequence in $\mathcal M(\Omega,L^2(I))$, $\{y_n\}$ is a Cauchy sequence in
$\mathcal C(\bar I,V)\cap \mathcal C^1(\bar I,H)$ too due to bound \eqref{aprioriy1} for $u=u_n$.
Thus $y_n\to\hat{y}$ in $\mathcal C(\bar I,V)\cap \mathcal C^1(\bar I,H)$ and
\[
 \|\hat{y}\|_{\mathcal C(\bar I,V)}+\|\partial_t\hat{y}\|_{\mathcal C(\bar I,H)}\leq c\,\|u\|_{\mathcal M(\Omega,\,L^2(I))}.
\]
Then we pass to the limit in \eqref{int_id1} for $y=y_n$, $u=u_n$ and $v(T)=0$ and see that $\hat{y}$ is a weak solution of \eqref{state_equation}.
{Due to uniqueness of the weaker solution we get $\hat{y}=y$, and the proof is complete.}
\end{proof}
\subsubsection{Some function spaces and embeddings}
We set
\[
 H^{(-1)}=V^*,\ \ H^{(0)}=H,\ \ H^{(1)}=V,\ \ H^{(2)}=V^2,\ \ H^{(3)}=V^3
\]
and introduce the interpolation spaces
\[
H^{(\lambda)}
:=\big(H^{(\ell)},H^{(\ell+1)}\big)_{\lambda-\ell,\infty},\ \
\ell:=\lfloor\lambda\rfloor,
\]
for non-integer $\lambda\in (-1,3)$ using the real $K_{\lambda,q}$-interpolation method of Banach spaces for $q=\infty$, see \cite{BergLofstrom76}.
Recall that the value  $q=\infty$ leads to the broadest intermediate spaces.
Their explicit description in terms of the Nikolskii spaces or their subspaces is known, see \cite{Nikolskii75,Triebel78,Zlotnik79}.
In particular,
\begin{gather*}
 H^{(\lambda)}=H^{\lambda,2}(\Omega)\ \ \text{for}\ \ 0<\lambda<\half,
\\
 H^{(\lambda)}=\tilde{H}^{1/2,2}(\Omega):=\{w\in L^2(\Omega)|\,
 \textsl{o}w\in H^{1/2,2}(\tilde{\Omega})\}\ \
 \text{for}\ \ \lambda=\half,
\\
 H^{(\lambda)}=H_0^{\lambda,2}(\Omega):=\{w\in H^{\lambda/2,2}(\Omega)|\,w|_{x=0,L}=0\}\ \ \text{for}\ \ \half<\lambda<1,
\end{gather*}
where $\|w\|_{\tilde{H}^{1/2,2}(\Omega)}=\|\textsl{o}w\|_{H^{1/2,2}(\tilde{\Omega})}$ and
$\textsl{o}w$ is the odd extension of $w$ with respect to $x=0$ and $L$ from $\Omega$ to $\tilde{\Omega}:=(-L,2L)$.
Hereafter equalities of Banach spaces are understood up to the equivalence of their norms.
It is well known that the space $\tilde{H}^{1/2,2}(\Omega)$ contains discontinuous but piecewise $\mathcal{C}^1$-functions.
\par In addition, let $D_x$ be the distributional derivative and, for a Banach space $B(\Omega)\subset L^1(\Omega)$, $B_\perp(\Omega)$ denote the subspace of $W\in B(\Omega)$ with the mean value
\[
\langle W\rangle_\Omega:=\frac{1}{L}\int_\Omega W\,\mathrm{dx}=0.
\]
Define the space $H^{-1/2,2}(\Omega)$ of distributions $w=D_xW$ with $W\in H_\perp^{1/2,2}(\Omega)$ equipped with the norm
$\|w\|_{H^{-1/2,2}(\Omega)}=\|W\|_{H^{1/2,2}(\Omega)}$.
Then $H^{(-1/2)}=H^{-1/2,2}(\Omega)$, see Item 3 of the proof of Lemma \ref{lem:embedding} below.
(Actually a quite similar result is valid for $H^{(\lambda)}$ for any $-1<\lambda<0$.)
Note that, in particular, the Dirac delta-function $\delta_a(x)=D_x\big(H(x-a)-(1-a/L)\big)\in H^{(-1/2)}$ for any $a\in\Omega$, where $H(\xi)=0$ for $\xi<0$ and $H(\xi)=1$ for $\xi>0$ is the Heaviside function.
\par Let $Q=\Omega\times I$ and $\Delta_hW(x)=W(x+h)-W(x)$ be the forward difference in $x$.
Define the spaces $H^{1/2,0;2}(Q)$ and $SHW^{1/2,1;2}(Q)$ of functions $W\in L^2(Q)$ such that respectively $|W|_{H^{1/2,0;2}(Q)}:=\sup_{0<h<L}h^{-1/2}\|\Delta_hW\|_{L^2((0,L-h)\times I)}<\infty$
and $\partial_tW\in H^{1/2,0;2}(Q)$
equipped with the norms
\begin{gather*}
\|W\|_{H^{1/2,0;2}(Q)}=\|W\|_{L^2(Q)}+|W|_{H^{1/2,0;2}(Q)},
\\
\|W\|_{SHW^{1/2,1;2}(Q)}=\|W\|_{L^2(Q)}+\|\partial_tW\|_{H^{1/2,0;2}(Q)}.
\end{gather*}
Here $H^{1/2,0;2}(Q)$ is a particular anisotropic Nikolskii space (of the order $1/2$ in $x$ only) and $SHW^{1/2,1;2}(Q)$ is a particular space of functions having the dominating mixed smoothness (of the order $1/2$ in $x$ in the Nikolskii sense and $1$ in $t$ in the Sobolev sense).
Note that $SHW^{1/2,1;2}(Q)\hookrightarrow H^{1/2,0;2}(Q)$.

\par For a Banach space $B(Q)\subset L^1(Q)$, let $B_\perp(Q)$ be the subspace of $W\in B(Q)$ such that $\langle W(\cdot,t)\rangle_\Omega=0$ on $I$.
Define the spaces $H^{-1/2,0;2}(Q)$ and $SHW^{-1/2,1;2}(Q)$ of distributions $w=D_xW$ with respectively
$W\in H_\perp^{1/2,0;2}(Q)$ and $W\in SHW_\perp^{1/2,1;2}(Q)$ equipped with the norms
\[
 \|w\|_{H^{-1/2,0;2}(Q)}=\|W\|_{H^{1/2,0;2}(Q)},\ \ \|w\|_{SHW^{-1/2,1;2}(Q)}=\|W\|_{SHW^{1/2,1;2}(Q)}.
\]
Note that all the spaces defined above and below in this subsection are Banach ones.
\par The next technical lemma plays an essential role below.
\begin{lemma}\label{lem:embedding}
\par The following equalities and embeddings hold
\begin{gather}
 \big(L^2(I,V^*),L^2(I,H)\big)_{1/2,\infty}=H^{-1/2,0;2}(Q),
\label{interpsp1}\\
 \big(H^1(I,V^*),H^1(I,H)\big)_{1/2,\infty}=SHW^{-1/2,1;2}(Q),
\label{interpsp2}\\
 \LwM\hookrightarrow H^{-1/2,0;2}(Q),
\label{embedding2}\\
 \mathcal{C}^1(\bar{I},\m)\hookrightarrow SHW^{-1/2,1;2}(Q).
\label{embedding2 1}
\end{gather}
\end{lemma}
\begin{proof}
1. Define the anisotropic Sobolev spaces
$W^{1,0;2}(Q)=\{W\in L^2(Q)|\,\partial_xW\in L^2(Q)\}$ and
$W^{0,1;2}(Q)=\{W\in L^2(Q)|\,\partial_tW\in L^2(Q)\}$ equipped with the norms
\[
 \|W\|_{W^{1,0;2}(Q)}=\|W\|_{L^2(Q)}+\|\partial_xW\|_{L^2(Q)},\
 \|W\|_{W^{0,1;2}(Q)}=\|W\|_{L^2(Q)}+\|\partial_tW\|_{L^2(Q)}.
\]
%These spaces together with all the spaces defined below in the proof are Banach ones.
The following equalities hold
\begin{gather}
 \big(L^2(Q),W^{1,0;2}(Q)\big)_{1/2,\infty}=H^{1/2,0;2}(Q),
\label{interpsp3a}\\
 \big(L_\perp^2(Q),W_\perp^{1,0;2}(Q)\big)_{1/2,\infty}=H_\perp^{1/2,0;2}(Q),
\label{interpsp3b}
\end{gather}
for example, see \cite[Ch. 1.2]{Zlotnik79}.
Recall that the corresponding $\hookleftarrow$-embeddings are proved by the classical techniques of approximation by the Steklov averages and the opposite $\hookrightarrow$-embeddings are rather simple.
Moreover, equality \eqref{interpsp3a} involving three spaces implies \eqref{interpsp3b} since it concerns the closed subspaces of one and the same type for all of these three spaces.
The same is valid for pairs of embeddings \eqref{interpsp4a}-\eqref{interpsp4b} and \eqref{interpsp5} below.
\par The elements in $L^2(I,V^*)$ and $L^2(I,H)=L^2(Q)$ can be uniquely identified as distributions $w=D_xW$ such that respectively
$W\in L_\perp^2(Q)$ with $\|w\|_{L^2(I,V^*)}=\|W\|_{L^2(Q)}$ and
$W\in W_\perp^{1,0;2}(Q)$ with $\|w\|_{L^2(I,H)}=\|\partial_xW\|_{L^2(Q)}$, where $\|\partial_xW\|_{L^2(Q)}$ is an equivalent norm in $W_\perp^{1,0;2}(Q)$ (in the latter case, of course, $D_xW=\partial_xW$).
In particular, for $w\in L^2(Q)$, clearly
$W(x,t)=\int_0^xw(\xi,t)\,d\xi-\big\langle\int_0^xw(\xi,t)\,d\xi\big\rangle_\Omega$.
Taking into account that one and the same operator establishes the one-to-one correspondence between respectively three spaces involved in equalities \eqref{interpsp3b} and \eqref{interpsp1}, the latter one is valid too.

\par 2. Define the space $SW^{1,1;2}(Q)=\{W\in W^{1,2}(Q)|\,\partial_x\partial_tW\in L^2(Q)\}$ equipped with the norm
$\|W\|_{SW^{1,1;2}(Q)}=\|W\|_{W^{1,2}(Q)}+\|\partial_x\partial_tW\|_{L^2(Q)}$.
The following equalities
\begin{gather}
 \big(W^{0,1;2}(Q),SW^{1,1;2}(Q)\big)_{1/2,\infty}=SHW^{1/2,1;2}(Q),
\label{interpsp4a}\\
 \big(W_\perp^{0,1;2}(Q),SW_\perp^{1,1;2}(Q)\big)_{1/2,\infty}=SHW_\perp^{1/2,1;2}(Q),
\label{interpsp4b}
\end{gather}
can be proved similarly to \eqref{interpsp3a}-\eqref{interpsp3b}.
\par The elements in $H^1(I,V^*)$ and $H^1(I,H)=W^{0,1;2}(Q)$ can be uniquely identified as the distributions $w=D_xW$ such that respectively
$W\in W_\perp^{0,1;2}(Q)$, with the equivalent norms $\|w\|_{H^1(I,V^*)}$ and $\|W\|_{W^{0,1;2}(Q)}$, and
$W\in SW_\perp^{1,1;2}(Q)$, with the equivalent norms $\|w\|_{H^1(I,H)}$ and $\|W\|_{SW^{1,1;2}(Q)}$.
Thus equality \eqref{interpsp4b} implies \eqref{interpsp2}.

\par 3. The following equalities hold
\begin{gather}
\hspace{-8pt}
 \big(L^2(\Omega),{H^1(\Omega)}\big)_{1/2,\infty}=H^{1/2,2}(\Omega),\,
 \big(L_\perp^2(\Omega),{H^1_\perp(\Omega)}\big)_{1/2,\infty}=H_\perp^{1/2,2}(\Omega)
\label{interpsp5}
\end{gather}
which are simpler 1D versions of \eqref{interpsp3a}-\eqref{interpsp3b}, for example, see \cite{BergLofstrom76,Triebel78} and \cite[Ch. 1.2]{Zlotnik79}.
The second equality implies the above mentioned one $H^{(-1/2)}=H^{-1/2,2}(\Omega)$.
\par Let $NBV(\bar{\Omega})$ be the space of normalized functions of bounded variation on $\bar{\Omega}$ that are continuous from the right at $x=0$ and continuous from the left at any $x\in (0,L]$; we equip it with the norm
$\|W\|_{NBV(\bar{\Omega})}=\sup_{\bar{\Omega}}W+\var_{\bar\Omega}W$.
Any $w\in\mathcal{M}(\Omega)$ can be represented as $w=D_xW$ with $W\in NBV(\bar{\Omega})$ and $\|w\|_{\mathcal{M}(\Omega)}=\var_{\bar\Omega}W$, for example, see \cite[Ch. 2]{Butazzo98}.
The representation is clearly unique for $W\in NBV_\perp(\Omega)$; in this subspace $\var_{\bar\Omega}W$ serves as an equivalent norm.

\par Notice that the following inequalities hold
\begin{gather}
 \sup_{\bar{\Omega}}|W(x)-\langle W\rangle_\Omega|\leq\var_{\bar\Omega}W,\ \
 \sup_{0<h<L}h^{-1}\|\Delta_hW\|_{L^1(0,L-h)}\leq\var_{\bar\Omega}W
\label{ineqNBV}
\end{gather}
for any $W\in NBV(\bar{\Omega})$;
the definition of the Riemann integral implies the latter one.
Then for any $W\in NBV_\perp(\bar{\Omega})$ we get the inequalities
\begin{gather}
 \sup_{0<h<L}h^{-1/2}\|\Delta_hW\|_{L^2(0,L-h)}
\nonumber\\
 \leq\big(\sup_{0<h<L}h^{-1}\|\Delta_hW\|_{L^1(0,L-h)}\big)^{1/2}(2\|W\|_{L^\infty(\Omega)}\big)^{1/2}
 \leq\sqrt{2}\var_{\bar\Omega}W
\label{ineqvarW}
\end{gather}
(they remain valid for $W\in NBV(\bar{\Omega})$ with $\var_{\bar\Omega}W$ replaced by $\|W\|_{NBV(\bar{\Omega})}$).
Thus
\[
 NBV(\bar{\Omega})\hookrightarrow H^{1/2,2}(\Omega),\ \
 NBV_\perp(\bar{\Omega})\hookrightarrow H_\perp^{1/2,2}(\Omega),\ \ \mathcal{M}(\Omega)\hookrightarrow H^{-1/2,2}(\Omega).
\]
\par Let $w\in\LwM$. Then $w\in L_w^2(I,V^*)=L^2(I,V^*)$,
where the equality is valid due to the classical Pettis theorem \cite[Theorem 8.15.2]{Edwards65} (since $V^*$ is separable), and
$w=D_xW$ with $W\in L_\perp^2(Q)$ and $\|W\|_{L^2(Q)}\leq c\|w\|_{\LwM}$.
\par {Moreover, we have $W(t)\in NBV(\bar\Omega)$ for a.e. $t\in I$. By} applying \eqref{ineqvarW} to $W(\cdot,t)$, omitting $\sup_{0<h<L}$ on the left, integrating the squared result over $I$ and taking back $\sup_{0<h<L}$ on the left, we obtain
\[
 |W|_{H^{1/2,0;2}(Q)}^2
 \leq 2\int_I\big(\var_{\bar\Omega}W(\cdot,t)\big)^2\,dt{=2\int_I\|w(t)\|_{\m}^2\,dt}=2\|w\|_{\LwM}^2
\]
that completes the proof of embedding \eqref{embedding2}.

\par 4. Let $w\in \mathcal{C}^1(\bar{I},\m)$. Then $w=D_xW$ and $\partial_tw=D_xZ$ with $W,Z\in H_\perp^{1/2,0;2}(Q)$ and
\begin{equation}
 \|W\|_{H^{1/2,0;2}(Q)}+\|Z\|_{H^{1/2,0;2}(Q)}\leq c\|w\|_{\mathcal{C}^1(\bar{I},\m)}
\label{ineqWZ}
\end{equation}
according to embedding \eqref{embedding2}.
\par Moreover, define the forward difference quotients in time $\Delta_\tau^{(1)}w(t)=(w(t+\tau)-w(t))/\tau$ for $0\leq t<t+\tau\leq T$.
Then for the same $t$ and $\tau$ owing to the first inequality \eqref{ineqNBV} we get
\[
 \sup_{\bar{\Omega}}|\Delta_\tau^{(1)}W(x,t)-Z(x,t)|
 \leq\var_{\bar{\Omega}}|\Delta_\tau^{(1)}W(x,t)-Z(x,t)|
 =\|\Delta_\tau^{(1)}w(t)-\partial_tw(t)\|_{\mathcal{M}(\Omega)}.
\]
Therefore
\[
 \|\Delta_\tau^{(1)}W-Z\|_{L^2(\Omega\times(0,T-\tau))}\leq\sqrt{L}\|\Delta_\tau^{(1)}w-\partial_tw\| _{\LwM}\to 0\ \text{as}\ \tau\to +0.
\]
Consequently there exists the derivative $\partial_tW=Z\in L^2(Q)$, and inequality \eqref{ineqWZ} implies
embedding \eqref{embedding2 1}.
\end{proof}

\par We also set $V^0(Q)=L^2(Q)$ and define the anisotropic Sobolev subspaces
\[
 V^\ell(Q)=\{w\in L^2(Q)|\,\partial_x^\ell w\in L^2(Q),\ w|_{\partial\Omega\times I}=0\},\
 \|w\|_{V^\ell(Q)}=\|\partial_x^\ell w\|_{L^2(Q)}
\]
for $\ell=1,2$
and the anisotropic Nikolskii subspaces
\[
 \tilde{H}^{\ell+1/2,0;2}(Q)=\{w\in L^2(Q)|\,\partial_x^\ell\textsl{o}w\in H^{1/2,0;2}(\tilde{Q})
 \,\ \text{and (if}\ \ell=1)\,\ w|_{\partial\Omega\times I}=0\}
\]
equipped with the norm
$\|w\|_{\tilde{H}^{\ell+1/2,0;2}(Q)}=\|\partial_x^\ell\textsl{o}w\|_{H^{1/2,0;2}(\tilde{Q})}$
for $\ell=0,1$, where $\tilde{Q}=\tilde{\Omega}\times I$.
Then the following equality holds
\begin{gather}
 \big(V^\ell(Q),V^{\ell+1}(Q)\big)_{1/2,\infty}=\tilde{H}^{\ell+1/2,0;2}(Q),\ \ \ell=0,1,
\label{interpspVH}
\end{gather}
which is similar to equality \eqref{interpsp3a}.

\section{Analysis of the control problem}\label{sec:existence_optimality}

According to Theorem \ref{regML2} {and Proposition \ref{prop:exist weaker}}
the state equation \eqref{state_equation} is uniquely solvable for any $u$ in either $\ML$ or ${\LwM}$ and the solution $y$ depends continuously on the data. Therefore, we can introduce the linear and bounded operator $\hat S\colon (u,y^0,y^1)\mapsto (y,y(T),{\rho}\partial_ty(T))$. The control-to-state mapping
\[
S\colon \M\rightarrow\mathcal Y,
~~u\mapsto (y,y(T),{\rho}\partial_ty(T))
\]
is given by $Su=\hat S(u,0,0)+\hat S(0,y^0,y^1)$ for fixed $y^0$ and $y^1$ and it is an affine and bounded operator. So we can rewrite the original control problem \eqref{measure_control_problem} in its reduced form
\begin{equation*}
j(u)=\half\left\|Su-\mathbf{z}\right\|_{\mathcal Y}^{2}+\alpha\|u\|_{\M}\to\min_{u\in \M}.
\end{equation*}
\begin{proposition}
Problem \eqref{measure_control_problem} has a unique solution $\bar u\in\M$.
\end{proposition}
\begin{proof}
The control-to-state operator $S$ is weak-star-to-strong sequential continuous, i.e., if $\{u_n\}\subset \M$ and $u_n\rightharpoonup^\ast u$ in $\M$, then $Su_n\rightarrow Su$ in $\mathcal Y$.  The proof of this continuity property is similar to \cite[Lemma 6.1]{KunischTrautmannVexler14} in the case of solutions by transposition resp. very weak solutions. The strong continuity follows from the compact embeddings and well known Aubin-Lions-Lemma. Then the direct method of calculus of variations combined with the sequential Banach-Alaoglu theorem ($\C$ is separable) can be applied to show existence of an optimal control. Additionally the control is unique since the control-to-state operator $S$ is injective and the data tracking functional is strictly convex.
\end{proof}
Owing to Proposition \ref{prop:exist weaker}
the optimal control $\bar u\in\M$ satisfies the inequalities
\begin{equation}
 \alpha\|\bar{u}\|_{\M}\leq j(\bar u)\leq j(0)=\half\|S(0)-\mathbf{z}\|_{\mathcal Y}^2
 \leq c\big(\|y^0\|_H+\|y^1\|_{V^*}+\|\mathbf{z}\|_\mathcal{Y}\big)^2
\label{uopt0}
\end{equation}
and thus
\begin{equation}
 \|\bar u\|_{\M}\leq c\big(\|\mathbf{y}\|_{H\times V^*}+\|z\|_{\mathcal Y}\big)^2\leq C.
\label{uopt}
\end{equation}
Hereafter $C>0$ depends on the norms of data.
\par Next we discuss first order optimality conditions.
{We introduce the adjoint control-to-solution operator
$S^\star\colon \mathcal{Y}\rightarrow C(\bar I,V)\hookrightarrow\C$, $(\phi,p^1,p^0)\mapsto p$ where $p$ is a weak solution of \eqref{dual_state_equation}.
This operator is well defined and bounded according to Proposition \ref{prop:exist weak}.
\par We also need the operator
$A^{-1}\colon V^\ast\rightarrow V$, $f\mapsto w$}
where $w\in V$ is the unique solution of
\begin{equation}\label{ode}
(\kappa\partial_x w,\partial_x v)_H=\langle f,v\rangle_{\Omega}\quad \forall v\in V.
\end{equation}
\par The next result provides the necessary and sufficient optimality condition for the optimal pair $(\bar{p},\bar{u})$.
\begin{proposition}\label{prop:opt_cond}
An element $\bar u\in\M$ is an optimal control of \eqref{measure_control_problem} if and only if
\begin{equation}\label{subgradient_cond_p}
-\bar p\in \alpha\partial\|\bar u\|_{\M},
\end{equation}
or {equivalently}
\begin{equation}\label{sub_gradient_definition}
\langle -\bar p,u-\bar u\rangle_{\C,\,\M}+\alpha\|\bar u\|_{\M}\leq \alpha\|u\|_{\M}~~\forall u\in\M
\end{equation}
where
{$\bar p=S^*\big(\bar y-z_1,-(\bar y(T)-z_2),A^{-1}(\rho\partial_t\bar y-z_3)\big)$}
with $(\bar{y},\bar{y}(T),{\rho}\partial_t\bar{y}(T))=\hat S(\bar u,y^0,y^1)$.
\end{proposition}
\begin{proof}
For $\M=\ML$ the proof of
{\cite[Theorem 7.1]{KunischTrautmannVexler14}} remains valid; for $\M={\LwM}$ it is similar {to~\cite[Theorem 3.2]{CasasClasonKunisch2013}.}
\end{proof}
To discuss further the properties
of the optimal control $\bar u$, we introduce the Jordan decomposition of a signed measure $\mu \in \m$, see \cite{Bogachev07}. There exists unique elements $\mu^\pm\in \m^+$ such that $\mu=\mu^+-\mu^-$. Moreover, we recall the polar decomposition of a vector measure
$\mu\in\ML$: $\mathrm d\mu=\mu'\mathrm d|\mu|$, where $\mu'$ is the Radon-Nikodym-derivative of $\mu$ with respect to $|\mu|$.
\par The subgradient condition in Proposition \ref{prop:opt_cond} implies the following conditions.
\begin{proposition}
Let $\bar u\in\M$ be the optimal control of \eqref{measure_control_problem} and $\bar p\in \C$ {be} the corresponding adjoint state. Then there holds
$\|\bar p\|_{\C}\leq \alpha$.

In the cases $\M={\LwM}$ and $\M=\ML$ there respectively hold
\[
\operatorname{supp}\bar u^{\pm}(t)\subset\{x\in \Omega\,|\,\bar p(t,x)=\mp\|\bar p(t,\cdot)\|_{\c}\}\quad\text{for a.a.}~t\in I
\]
and
\begin{equation}\label{supp ML2}
\operatorname{supp}|\bar u|\subset \{x\in \Omega\,|\,\|\bar p(\cdot,x)\|_{\lt}=\alpha\},\ \ \bar u'=-\alpha^{-1}\bar p~~\text{in}~~L^1\big((\Omega,|\bar u|),L^2(I)\big).
\end{equation}
\end{proposition}
\begin{proof}
A detailed discussion of the proof of these results can be found in \cite{CasasClasonKunisch2013,KunischPieperVexler2014}.
\end{proof}
\par The regularity of the adjoint state $\bar{p}$ is now applied to show improved regularity of the optimal control $\bar u$.
\begin{theorem}\label{thm:imp_reg_control}
Let $\M=\ML$, $\mathbf{z}\in \mathcal{Y}^1:=L^2(I,V)\times V\times H$, $\mathbf{y}\in V\times H$
and $\bar u$ be the optimal control of \eqref{measure_control_problem}. Then $\bar u\in \mathcal C^1(\bar I,\m)$ and the following bound holds
\[
\|\bar u\|_{\mathcal C^1(\bar I,\m)}\leq {C=}C\big(\|\mathbf{y}\|_{V\times H},\|\mathbf{z}\|_{\mathcal{Y}^1}\big).
\]
\end{theorem}
\begin{proof}
There holds $\bar y \in {\mathcal C(\bar{I},V)\cap \mathcal C^1(\bar{I},H)}$ according to Theorem \ref{regML2}. Thus, the optimal adjoint state has the following regularity
$\bar p \in \mathcal C(\bar I,V^2)\cap \mathcal C^1(\bar{I},V)$
by Proposition \ref{prop:exist weak}. We have $\bar u=-{\alpha^{-1}}
\bar p\,|\bar u|$ according to \eqref{supp ML2}.
Moreover, we define the function
\[
 w=-\alpha^{-1}(\partial_t\bar p)|\bar u|
\]
and show that it serves the time derivative of $\bar u$.
For any $t_0,t\in\bar{I}$ and $t_0\neq t$, we define the difference quotient
$\bar{u}{(t_0,t)}=\big(\bar{u}(t)-\bar{u}(t_0)\big)/(t-t_0)$. Then we consider
\begin{align*}
 \|\bar{u}{(t_0,t)}-w(t_0)\|_{\m}&
 =\alpha^{-1}\sup_{\|\phi\|_{\c}\leq 1}\int_\Omega\big(\bar{p}(t_0,t)-\partial_t{\bar{p}}(t_0)\big)\phi~\mathrm d|\bar u|\\
&\leq\alpha^{-1}\|\bar{p}(t_0,t)-\partial_t{\bar{p}}(t_0)\|_{\c}\|\bar u\|_{\ML}\\
&\leq c\alpha^{-1}\|\bar{p}(t_0,t)-\partial_t{\bar{p}}(t_0)\|_{V}\|\bar u\|_{\ML}\rightarrow 0
\end{align*}
as $t\to t_0$ since $\bar p\in \mathcal C^1(\bar I,V)$.
Next, quite similarly we get
\[
\|w(t)-w(t_0)\|_{\m}\leq c\alpha^{-1}\|\partial_t\bar p(t)-\partial_t\bar p(t_0)\|_{V}\|\bar u\|_{\ML}\to 0
\]
as $t\to t_0$. Consequently $\partial_t\bar u=w\in \mathcal{C}(\bar{I},\m)$.
Finally, we bound $\partial_t \bar u$ as follows
\begin{multline*}
\|\partial_t\bar u\|_{\mathcal C(\bar I,\m)}\leq c\alpha^{-1}\|\partial_t \bar p\|_{\mathcal C(\bar I,V)}\|\bar u\|_{\ML}\\
\leq c_1\alpha^{-1}\big(\|\bar y-z_1\|_{L^2(I,V)}+\|\bar y(T)-z_2\|_{V}+\|\rho\partial_t \bar y(T)-z_3\|_H\big)\|\bar u\|_{\ML}\\
\leq c_2\alpha^{-1}\big(\|\bar u\|_{\ML}+\|\mathbf{y}\|_{V\times H}+\|\mathbf{z}\|_{\mathcal{Y}^1}\big)\|\bar u\|_{\ML}
\end{multline*}
owing to Proposition \ref{prop:exist weak} and Theorem \ref{regML2}. Utilizing bound \eqref{uopt} for $\bar{u}$, we complete the proof.
\end{proof}
\section{Discretization of the state equation}
\label{sec:disc state}
We introduce the uniform {grid $t_m=m\tau$ in time with the step $\tau=T/M$ and a non-uniform grid $0=x_0 < x_1 <\ldots < x_N=L$ in space
with the steps $h_j=x_j-x_{j-1}$, where $M\geq 2$ and $N\geq 2$.
Let also} $h=\max_{j=1,\ldots,N}h_j$, $h_{\rm\min}=\min_{j=1,\ldots,N}h_j$ and $\rh=(\tau,h)$.
We assume that the space grid is \textit{quasi-uniform}, i.e., $h\leq c_1h_{\rm\min}$.
Hereafter $c,c_1,C$, etc., are grid-independent.
\par Let $V_\tau\subset H^1(I)$ {and $V_h\subset V$ be the spaces of piecewise linear finite elements with respect to the introduced grids on $\bar{I}$ and $\bar{\Omega}$.}
\par {We approximate the state variable $y$ by $y_\rh\in V_\rh:=V_\tau\otimes V_h\subset H^1(I,V)$ and additionally $\partial_ty(T)$ by $y_{Th}^1\in V_h$.}
For $(u,y^0,y^1)\in \M\times H\times V^*$ \textit{the discrete state equation} has the following form
\begin{align}
B_\sigma(y_\rh,v)+(\rho {y_{Th}^1},v(T))_{H}&=\langle u,v\rangle_{\M,\,\C}+\langle\rho y^1,v(0)\rangle_{\Omega}~~\forall v\in V_\rh,
\label{discrete_state_equation1}\\
(\rho y_\rh(0),\varphi)_{H}&=(\rho y^0,\varphi)_{H}~~\forall \varphi\in V_h,
\label{discrete_state_equation2}
\end{align}
involving the indefinite symmetric bilinear form
\begin{multline}
B_\sigma(y,v):=-(\rho\partial_ty,\partial_t v)_{\L}
-\big(\sigma-\frac16\big)\tau^2(\kappa\partial_x\partial_ty,\partial_x\partial_tv)_{L^2(I\times\Omega)}\\
+(\kappa\partial_xy,\partial_xv)_{L^2(I\times\Omega)},
\label{bilinear_form}
\end{multline}
{with the grid independent parameter $\sigma$, cf. \eqref{int_id1}.
This definition follows \cite{Zlotnik94} but notice carefully that normally
$y_\rh$ is uniquely defined by \eqref{discrete_state_equation1} with $v(T)=0$ and \eqref{discrete_state_equation2}.
To treat general $v$, we need $y_{Th}^1$.
}
\begin{remark}
The second term on the right {hand-side of} \eqref{bilinear_form}
{regularizes the Galerkin (i.e. projection) method}
with respect to {bilinear form \eqref{dif bilinear_form}.}
It is included to ensure unconditional stability for {suitable values of $\sigma$.}
Moreover, the term \[-(1/6)\,\tau^2(\kappa\partial_x\partial_ty,\partial_x\partial_tv)_{L^2(I\times\Omega)}\] is the error term of {the compound} trapezoidal rule applied for the calculation of the temporal integral in $(\kappa\partial_xy,\partial_xv)_{L^2(I\times\Omega)}$. So that, in particular, for $\sigma=0$ in \eqref{bilinear_form} this temporal integral is calculated using {this}
rule whereas for $\sigma=1/6$ it is not approximated.
\end{remark}

Next we recall the inverse inequality
\begin{equation}
 {\|\varphi\|_{\mathcal{V}_\kappa}}\leq\alpha_h\|\varphi\|_{{H_\rho}}~~\forall\varphi\in V_h
\label{inv_ineq}
\end{equation}
where the least constant satisfies $c_1h^{-1}\leq\alpha_h\leq c_2h^{-1}$
for the quasi-uniform grid.
For $\sigma\leq 1/4$ we need to state conditions linking the temporal and spatial grids to ensure stability of the numerical method.
{
\begin{assumption}\label{as:stability}
In what follows, let
\begin{gather}
\textstyle{ \text{if}~~ \sigma<\frac14,~~ \text{then}~~ \tau^2\alpha_h^2\big(\frac14-\sigma\big)\leq 1-\varepsilon_0^2~~\text{for some}~~0<\varepsilon_0<1,}
\label{stab cond1}\\
\textstyle{  \text{if}~~ \sigma\leq\frac14,~~ \text{then}~~ \tau^2\alpha_h^2\big(\frac{1+\varepsilon_1^2}{4}-\sigma\big)\leq 1~~\text{for some}~~0<\varepsilon_1\leq 1.}
\label{stab cond2}
\end{gather}
\end{assumption}
}
\begin{remark}
{The parameters $\varepsilon_0$ and $\varepsilon_1$ can be chosen arbitrarily small
but then constants in the stability and error estimates for our {FEM}
can tend to infinity.}
\end{remark}
\begin{remark}
As we see {below in Section \ref{tistfo},
the method is}
related to well known time-stepping methods, in particular, to the explicit Leap-Frog-method for $\sigma=0$.
Then conditions \eqref{stab cond1} and \eqref{stab cond2} reduce to a CFL-type one
$\tau\alpha_h\leq 2\sqrt{1-\varepsilon_0^2}$.
For $\sigma=1/4$ the method is related to the Crank-Nicolson
scheme and is unconditionally stable but in a weaker norm than we need to derive our error estimates so that we impose a very weak CFL-type condition
$\tau\alpha_h\leq 2/\varepsilon_1$.
\end{remark}

\par Below in proofs we utilize the auxiliary squared norms
\begin{gather*}
\|\varphi\|_{H_\tau^0}^2
:=\|\varphi\|_{{H_\rho}}^2+\big(\sigma-\frac14\big)\tau^2\|\varphi\|_{{\mathcal{V}_\kappa}}^2,\\
\|y\|_{\mathcal C_\tau(H_E)}^2
=\max_{1\leq m\leq M}\Big(\frac 1\tau\|y_m-y_{m-1}\|_{H_\tau^0}^2+\frac 1 2\|y_m+y_{m-1}\|_{{\mathcal{V}_\kappa}}^2\Big)
\end{gather*}
for $\varphi\in V_h$ and $y\in V_\tau\otimes V_h$.
We
need to bound them by
standard norms.
\begin{lemma}\label{lem:stab_est}
{Under conditions \eqref{stab cond1} and \eqref{stab cond2} the following inequalities
hold
\begin{gather}
\varepsilon_0\|\varphi\|_{{H_\rho}}\leq\|\varphi\|_{H_\tau^0}\ \ \forall\varphi\in V_h,
\label{stab ineq 1}\\
\textstyle{ \|y\|_{\mathcal C_\tau(\mathcal{V}_\kappa)}:=\max_{0\leq m\leq M}\|y(t_m)\|_{\mathcal{V}_\kappa}\leq \frac{\sqrt{2}}{\varepsilon_1}\|y\|_{\mathcal C_\tau(H_E)}\ \ \forall y\in V_\rh}
\nonumber
\end{gather}
with $\varepsilon_0:=1$ for $\sigma\geq 1/4$ and $\varepsilon_1:=\sqrt{4\sigma-1}$ for $\sigma>1/4$.}
\end{lemma}
\begin{proof}
For $\sigma\geq1/4$, the first inequality is obvious; for $\sigma<1/4$ it can be checked by a direct calculation {using \eqref{inv_ineq}}. The proof of the second inequality is covered in \cite[Corollary 2.1]{Zlotnik94}.
\end{proof}
Now we discuss some properties of $y_{Th}^1$ and $\partial_ty(T)$ that are essential below.
\begin{proposition}
Let $(y_\rh,{y_{Th}^1})\in V_\rh\times V_h$ be the solution of \eqref{discrete_state_equation1}-\eqref{discrete_state_equation2}. Then there holds
\begin{equation}
(\rho {y_{Th}^1},\varphi)_{H}=-\Big(\kappa\partial_x\int_Iy_\rh~\mathrm dt,\partial_x\varphi\Big)_{H}
+{ \int_I\langle u,\varphi\rangle_\Omega~\mathrm dt}
+(\rho y^1,\varphi)_{H}~~\forall \varphi\in V_h.
\label{dsefordty}
\end{equation}
\end{proposition}
\begin{proof}
This is proved by testing \eqref{discrete_state_equation1} with time constant functions $v=\varphi\in V_h$.
\end{proof}
The non-local in time identity \eqref{dsefordty} is convenient for our error analysis but not for the implementation; for the latter issue see Section \ref{tistfo}.
Identities similar to
\eqref{dsefordty} also hold on the continuous level.
\begin{proposition}
\begin{enumerate}
\item Let $y\in \mathcal C(\bar I,V)\cap \mathcal C^1(\bar I,H)$ be the weak solution of \eqref{state_equation} for $\M=\ML$.
Then there holds
\begin{equation}
 (\rho\partial_ty(T),\varphi)_H=-\Big(\kappa\partial_x\int_Iy~\mathrm dt,\partial_x\varphi\Big)_{H}
+{\Big\langle \int_Iu~\mathrm dt,\varphi\Big\rangle_\Omega}
 +(\rho y^1,\varphi)_{H}
 ~~\forall \varphi\in V.
\label{sefordty1}
\end{equation}
\item Let {$y\in \mathcal C(\bar I,H)\cap \mathcal C^1(\bar I,V^\ast)$}
be the weaker {(very weak)}
solution of \eqref{state_equation} for $\M={\LwM}$. Then there holds
\begin{equation}
\langle\rho\partial_ty(T),\varphi\rangle_{\Omega}=-\Big(\kappa\partial_x\int_Iy~\mathrm dt,\partial_x\varphi\Big)_{H}
+{\int_I\langle u,\varphi\rangle_\Omega~\mathrm dt}
+\langle\rho y^1,{\varphi}\rangle_{\Omega}
 ~~\forall \varphi\in V.
\label{sefordty}
\end{equation}
\end{enumerate}
\end{proposition}
\begin{proof}
For $\M=\ML$ identity \eqref{sefordty1} is proved by testing \eqref{int_id1} with time constant function $v=\varphi\in V$.
For
$\M={\LwM}$ we test \eqref{int_id4} with any $\varphi\in V^2$
and get
\begin{equation*}
\langle \rho\partial_ty(T),\varphi\rangle_{\Omega}=((\mathcal{I}_{t}y)(T),\partial_x(\kappa\partial_x\varphi))_H
+\langle u,\varphi\rangle_{{\M,\,\C}}
\end{equation*}
According to Proposition \ref{prop:exist weaker} we have $\mathcal{I}_{t}y\in\mathcal C(\bar I,V)$. Thus there holds
\begin{equation*}
\langle \rho\partial_ty(T),\varphi\rangle_{\Omega}
=-\big(\kappa\partial_x(\mathcal{I}_{t}y)(T),\partial_x\varphi\big)_H
+\langle u,\varphi\rangle_{{\M,\,\C}}
+\langle\rho y^1,\varphi\rangle_{\Omega}.
\end{equation*}
The density of $V^2$
in $V$ implies \eqref{sefordty}.
\end{proof}
\par {For our analysis, we need some projection and interpolation operators.}
We introduce the standard projectors $\pi_h^0$: ${H_\rho}\to V_h$ and $\pi_h^1$: ${\mathcal{V}_\kappa}\to V_h$ defined by
\begin{gather}
(\rho\pi_h^0w,{\varphi})_{H}=(\rho w,{\varphi})_{H}\quad \forall {\varphi}\in V_h,
\label{pi0}\\
(\kappa\partial_x\pi_h^1w,\partial_x{\varphi})_{H}
=(\kappa\partial_xw,\partial_x{\varphi})_{H}\quad \forall {\varphi}\in V_h.
\label{pi1}
\end{gather}
{Clearly $\|\pi_h^0w\|_{H_\rho}\leq \|w\|_{H_\rho}$ and
$\|\pi_h^1w\|_{\mathcal{V}_\kappa}\leq \|w\|_{\mathcal{V}_\kappa}$.
Identity \eqref{discrete_state_equation2} means that $y_\rh(0)=\pi_h^0y^0$.

\par Moreover the following property holds
\begin{equation}
{ (w,\tilde{w})_{\mathcal{V}_\kappa}-(\pi_h^1w,\pi_h^1\tilde{w})_{\mathcal{V}_\kappa}
 =(w-\pi_h^1w,\tilde{w}-\pi_h^1\tilde{w})_{\mathcal{V}_\kappa}\ \ \forall w,\tilde{w}\in V.}
\label{galerkin ort}
\end{equation}
\par Following \cite{Zlotnik94}, we also introduce \textit{the regularized ${H_\rho}$ projector} $\pi_{h,\sigma_0}^0$: $V\to V_h$ defined by
\begin{equation}
 (\rho\pi_{h,\sigma_0}w,\varphi)_H+\sigma_0\tau^2(\kappa\partial_x\pi_{h,\sigma_0}w,\partial_x\varphi)_{H}=(\rho w,\varphi)_H\quad
\forall \varphi\in V_h.
\label{reg_projection}
\end{equation}
with the grid independent parameter $\sigma_0\geq\sigma-1/4$.
{ Clearly $\pi_{h,\sigma_0}=\pi_h^0$ for $\sigma_0=0$.}
\par {Let $i_\tau$: $\mathcal C(\bar{I})\to V_\tau$ be the interpolation operator such that $i_\tau w(t_m)=w(t_m)$ for all $m=0,\ldots,M$.}

Next we {define the operator}
$A_h^{-1}\colon V^\ast\rightarrow V_h$, {$f\mapsto w_h$} where $w_h\in V_h$ is the unique solution of
\begin{equation}\label{discrete_laplace_equation}
(\kappa\partial_xw_h,\partial_x{\varphi})_H=\langle f,{\varphi}\rangle_{\Omega}\quad
\forall {\varphi}\in V_h.
\end{equation}
{Clearly $A_h^{-1}=\pi_h^1A^{-1}$, see \eqref{ode} with $w=A^{-1}f$, and
t}he norm in {$\mathcal{V}_\kappa^\ast$} and its discrete counterpart can be written as
\[
\|f\|_{\mathcal{V}_\kappa^\ast}=\|A^{-1}f\|_{{\mathcal{V}_\kappa}}{=\|w\|_{\mathcal{V}_\kappa}},\quad \|f\|_{\mathcal{V}_{\kappa,h}^*}:=\|A^{-1}_hf\|_{{\mathcal{V}_\kappa}}{=\|w_h\|_{\mathcal{V}_\kappa}
\leq\|w\|_{\mathcal{V}_\kappa}.}
\]

\par Moreover, we
set $r_hA^{-1}:=A^{-1}-A_h^{-1}=A^{-1}-\pi_h^1A^{-1}$. {First}
we note that
\begin{equation}\label{regularity_poisson}
A^{-1}\colon H^{(\lambda)}\rightarrow H^{(\lambda+2)},\quad -1\leq\lambda\leq 1.
\end{equation}
Then by the standard FEM error analysis \cite{BrennerScott:2008}
and operator interpolation theory we have
\begin{equation}\label{V_est_poisson}
\|r_hA^{-1}f\|_V=\|w-{\pi_h^1w}\|_V\leq ch^{1+\lambda}\|f\|_{H^{(\lambda)}}\ \ {\forall f\in H^{(\lambda)}},\quad -1\leq\lambda\leq 0,
\end{equation}
\begin{equation}\label{H_est_poisson}
\|r_hA^{-1}f\|_H=\|w-{\pi_h^1w}\|_H\leq ch^{2+\lambda}\|f\|_{H^{(\lambda)}}\ \ {\forall f\in H^{(\lambda)}},\quad -1\leq\lambda\leq 0.
\end{equation}
\section{Stability and error estimates for the discrete state equation}\label{sec:error_state}
In this section we present error estimates for the state equation.
{We begin with an auxiliary result.
\begin{lemma}
\label{est dif init}
For $\sigma_0\geq\sigma-1/4\geq 0$, the following estimate holds
\begin{gather}
\|{\pi_{h,\sigma_0}w-\pi_h^0w}\|_{H_\tau^0}
\leq c(\tau+h)^\lambda\|{w}\|_{H^{(\lambda)}}\ \ {\forall w\in H^{(\lambda)}},~~ \text{for}~~ 1\leq\lambda\leq2.
\label{disp12b}
\end{gather}
\end{lemma}
\begin{proof}
We recall the well known estimates
\begin{gather}
 \|\pi_h^0w\|_{V}\leq c\|w\|_{V}~~{\forall w\in V},
\label{est s0_1}\\[1mm]
 \|w-\pi_h^0w\|_{V}\leq ch\|w\|_{V^2}~~{\forall w\in V^2},
\label{est s0_2}
\end{gather}
which are valid {using the inverse inequality \eqref{inv_ineq}.}
We also remind inequality \eqref{stab ineq 1} and notice also that for $\sigma_0\geq 0$
the following additional inequality holds
\begin{gather}
 \sqrt{\sigma_0}\tau\|\varphi\|_{\mathcal{V}_\kappa}\leq\|\varphi\|_{H_\tau^0}\ \ \forall\varphi\in V_h.
\label{stab ineq 3}
\end{gather}

\par Let {$w\in V$ and $\varphi\in V_h$}.
We apply identities \eqref{pi0} and \eqref{reg_projection} and get
\begin{multline*}
 \big({\rho(\pi_{h,\sigma_0}^0w-\pi_h^0w)},\varphi\big)_{H}
 +\sigma_0\tau^2\big({\kappa}\partial_x(\pi_{h,\sigma_0}^0{w}-\pi_h^0{w}),
 \partial_x\varphi\big)_{H}\\
 =-\sigma_0\tau^2\big({\kappa}\partial_x\pi_h^0{w},\partial_x\varphi\big)_{H}
\\
 =\sigma_0\tau^2\big({\kappa}\partial_x({w}-\pi_h^0{w}),
 \partial_x\varphi\big)_{H}
 +\sigma_0\tau^2{\langle\partial_x(\kappa\partial_x{w}),\varphi\rangle_\Omega}.
\end{multline*}
Now we set $\varphi=\pi_{h,\sigma_0}^0{w}-\pi_h^0{w}$,
and from the former equality and estimate \eqref{est s0_1} for $\lambda=1$ as well as the latter equality and estimate \eqref{est s0_2}  for $\lambda=2$ we obtain the estimate
\begin{gather*}
 \|\pi_{h,\sigma_0}^0{w}-\pi_h^0{w}\|_{H_\tau^0}\leq c\tau(\tau+h)^{\lambda-1}\|{w}\|_{H^{(\lambda)}}\ \ \text{for}\ \ \lambda=1,2.
\end{gather*}
\par By using the $K_{\lambda,\infty}$-method, we complete the proof.
\end{proof}
Now we get a stability bound and error estimates in $\mathcal{C}(\bar{I},H)\times \mathcal{V}_{\kappa,h}^*$ for the discrete state equation.
\begin{proposition}
\label{prop: est stat eq}
{Let $y$ and $(y_\rh,y_{Th}^1)$ be the solutions to the state equation \eqref{state_equation} and the discrete state equation \eqref{discrete_state_equation1}-\eqref{discrete_state_equation2}.}
\begin{enumerate}
\item {For $(u,y^0,y^1)\in L^2(I,V^\ast)\times V\times V^*$, t}he following stability bound holds:
\begin{equation}
 \|y_\rh\|_{\mathcal C(\bar I,H)}+\|\rho {y_{Th}^1}\|_{\mathcal{V}_{\kappa,h}^*}
 \leq c\,\big(\|u\|_{L^2(I,V^\ast)}+{\|\mathbf{y}\|_{V\times V^*}}
\big).
\label{disp11}
\end{equation}
\item For $(u,y^0,y^1)\in {H^{-1/2,0;2}(Q)}
\times V\times H^{(-1/2)}$, the following error estimate holds:
\begin{equation}
 \|y-y_\rh\|_{\mathcal C(\bar I,H)}+\|\rho(\partial_ty(T)-{y_{Th}^1})\|_{\mathcal{V}_{\kappa,h}^*}
 \leq c\,(\tau+h)^{1/3}\big(\|u\|_{{H^{-1/2,0;2}(Q)}}
 +{\|\mathbf{y}\|_{V\times H^{(-1/2)}}}
\big).
\label{eesv11}
\end{equation}
\item For $(u,y^0,y^1)\in {SHW^{-1/2,1;2}(Q)}\times V\times H$, the higher order {error} estimate holds:
\begin{equation}
 \|y-y_\rh\|_{\mathcal C(\bar I,H)}+\|\rho(\partial_ty(T)- {y_{Th}^1})\|_{\mathcal{V}_{\kappa,h}^*}
 \leq c\,(\tau+h)^{2/3}\big(\|u\|_{{SHW^{-1/2,1;2}(Q)}}
 +{\|\mathbf{y}\|_{V\times H}}
\big).
\label{disp13}
\end{equation}
\end{enumerate}
\end{proposition}
\begin{proof}
1. According to \cite[Theorem 2.1 (1)]{Zlotnik94},
the bound
\begin{equation}
 \|y_\rh\|_{\mathcal C(\bar I,\,H)}+\Big\|{\int_I}y_\rh~\mathrm dt\Big\|_{V}
 \leq c\,\big(\|u\|_{L^2(I,V^\ast)}+\|y_\rh(0)\|_{H_\tau^0}+\|y^1\|_{V^\ast}\big)
\label{stab L2}
\end{equation}
is valid for any $y_\rh(0)\in V_h$. We have $y_\rh(0)=\pi_h^0y^0$. In the case $\sigma\leq 1/4$, there {clearly} holds
\[
\|\pi_h^0y^0\|_{H^0_\tau}\leq \|\pi_h^0y^0\|_{{H_\rho}}\leq {\|y^0\|_{H_\rho}.}
\]
For $\sigma>1/4$, {we alternatively get}
using \eqref{disp12b} for $\lambda=1$
\begin{equation*}
\|\pi_h^0y^0\|_{H^0_\tau}\leq\|\pi_h^0y^0-\pi_{h,\sigma_0}^hy^0\|_{H^0_\tau}+\|\pi_{h,\sigma_0}^0y^0\|_{H^0_\tau}
\leq c(\tau+h)\|y^0\|_{V}+\|y^0\|_{H_\rho}
\end{equation*}
for any ${\sigma_0\geq\sigma-1/4}$.
\par We proceed with the bound for {$y_{Th}^1$}.
{Identity \eqref{dsefordty} and bound \eqref{stab L2} together with the generalized Minkowski inequality
}
imply
\begin{multline}
\|\rho y_{Th}^1\|_{\mathcal{V}_{\kappa,h}^*}\leq c\left(\Big\|{\int_I}y_\rh~\mathrm dt\Big\|_{V}
+{\Big\|\int_Iu~\mathrm dt\Big\|_{V^\ast}}+\|\rho y^1\|_{V^\ast}\right)\\
\leq c_1\,\big(\|u\|_{L^2(I,V^\ast)}+{\|\mathbf{y}\|_{V\times H}}\big).
\label{stab L2 a}
\end{multline}
{Finally we derive}
bound \eqref{disp11}.
\par 2. Let $\tilde y_\rh$ be the solution of equation \eqref{discrete_state_equation1} for $\tilde y_\rh(0)=\pi_{h,\sigma_0}^0y^0$.
\Blue{Bound \eqref{stab L2} together with \eqref{est dif init} for $\lambda=1$, the bound in Proposition \ref{prop:exist weaker} and the stability of $\pi_h^1$ in $V$ imply
\[
\|y-\tilde y_\rh\|_{\mathcal C(\bar I,H)}
 +\Big\|\int_I(\pi_h^1y-\tilde y_\rh)~\mathrm dt\Big\|_{V}\leq c\big(\|u\|_{L^2(I,V^\ast)}+\|\mathbf{y}\|_{H\times V^\ast}\big).
\]
Owing to \cite[Theorem 4.1]{Zlotnik94} the following error estimate holds
\[
\|y-\tilde y_\rh\|_{\mathcal C(\bar I,H)}
 +\Big\|\int_I(\pi_h^1y-\tilde y_\rh)~\mathrm dt\Big\|_{V}\leq c(\tau+h)^{2/3}\big(\|u\|_{L^2(Q)}+\|\mathbf{y}\|_{V\times H}\big).
\]
Using the $K_{1/2,\infty}$-method}
%{\color{green} The inequality \eqref{stab L2} together with \eqref{est dif init}, the a priori estimate in Proposition \ref{prop:exist weaker} and the stability of $\pi_h^1$ imply
%\[
%\|y-\tilde y_\rh\|_{\mathcal C(\bar I,H)}
% +\Big\|\int_I(\pi_h^1y-\tilde y_\rh)~\mathrm dt\Big\|_{V}\leq c\big(\|u\|_{L^2(I,V^\ast)}+\|\mathbf{y}\|_{H\times V^\ast}\big).
%\]
%Then \cite[Theorem 4.1]{Zlotnik94} implies
%\[
%\|y-\tilde y_\rh\|_{\mathcal C(\bar I,H)}
% +\Big\|\int_I(\pi_h^1y-\tilde y_\rh)~\mathrm dt\Big\|_{V}\leq c(\tau+h)^{2/3}\big(\|u\|_{L^2(Q)}+\|\mathbf{y}\|_{V\times H}\big).
%\]
%Using the $K_{1/2,\infty}$-method}
\Blue{and equality \eqref{interpsp1} we get the intermediate} error estimate
\begin{multline}
 \|y-\tilde y_\rh\|_{\mathcal C(\bar I,H)}
 +\Big\|\int_I(\pi_h^1y-\tilde y_\rh)~\mathrm dt\Big\|_{V}\\
\leq c(\tau+h)^{1/3}\big(\|u\|_{{H^{-1/2,0;2}(Q)}}+{\|\mathbf{y}\|_{H^{(1/2)}\times H^{(-1/2)}}}\big).
\end{multline}
In the case $\sigma\leq 1/4$ we can choose $\sigma_0=0$, then $y_\rh(0)=\pi_{h,\sigma_0}y^0=\pi_h^0y^0$ and $\tilde y_\rh=y_\rh$.
In the case $\sigma\geq 1/4$ we can use the stability bound \eqref{stab L2} and estimate \eqref{disp12b} to get
\begin{equation}
\|\tilde y_\rh-y_\rh\|_{\mathcal C(\bar I,H)}+\Big\|{\int_I}(\tilde y_\rh-y_\rh)~\mathrm dt\Big\|_{V}\leq c\,\|\pi_{h,\sigma_0}^0y^0-\pi_h^0y^0\|_{H_\tau^0}\leq c_1(\tau+h)\|y^0\|_{V}.
\label{diftyy}
\end{equation}

\par Then by subtracting \eqref{dsefordty} from \eqref{sefordty} and applying identity \eqref{pi1} we find
\begin{multline*}
 \langle\rho\big(\partial_ty(T)-{y_{Th}^1}\big),\varphi\rangle_{\Omega}
 =-\left(\kappa\partial_x{\int_I}(y-y_{\rh})~\mathrm dt,\partial_x\varphi\right)_{H}\\
 =-\left(\kappa\partial_x{\int_I}(\pi_h^1y-y_{\rh})~\mathrm dt,\partial_x\varphi\right)_{H}
 ~~\forall \varphi\in V_h,
\end{multline*}
consequently
\begin{equation}
 \|\rho\big(\partial_ty(T)-{y_{Th}^1}\big)\|_{\mathcal{V}_{\kappa,h}^*}
 \leq c\Big\|\int_I(\pi_h^1y-y_\rh)~\mathrm dt\Big\|_{V}.
\label{estfordty}
\end{equation}
Thus we obtain \eqref{eesv11}.

\par 3. Once again we apply \cite[Theorem 4.1]{Zlotnik94}
and first get the estimate
\[
 \|y-\tilde{y}_\rh\|_{\mathcal C(\bar I,H)}+\Big\|\int_I(\pi_h^1y-\tilde{y}_\rh)~\mathrm dt\Big\|_{V}
 \leq c\,(\tau+h)^{2/3}\big(\|u\|_{L^2(I,H)}+\|\mathbf{y}\|_{V\times H}\big).
\]
Combining it together with \eqref{diftyy}, we derive
\begin{equation}
 \|y-y_\rh\|_{\mathcal C(\bar I,H)}+\Big\|\int_I(\pi_h^1y-y_\rh)~\mathrm dt\Big\|_{V}
 \leq c\,(\tau+h)^{2/3}\big(\|u\|_{L^2(I,H)}+\|\mathbf{y}\|_{V\times H}\big).
\label{est2over3}
\end{equation}
In this proof, we apply this estimate in the case $u=0$ only (but in general case below).

\par In the remaining case $\mathbf{y}=0$, from \cite[Theorem 4.1]{Zlotnik94} we also get the higher order error estimate
\begin{equation}
 \|y-y_\rh\|_{\mathcal C(\bar I,H)}+\Big\|\int_I(\pi_h^1y-y_\rh)~\mathrm dt\Big\|_{V}
 \leq c\,(\tau+h)^{4/3}\|u\|_{H^1(I,H)}\ \ \forall u\in H^1(I,H).
\label{est4over3}
\end{equation}
Moreover owing to Proposition \ref{prop:exist weaker} and bound \eqref{disp11} (both for $\mathbf{y}=0$) we have
\begin{multline*}
 \|y-y_\rh\|_{\mathcal C(\bar I,H)}+\Big\|\int_I(\pi_h^1y-y_\rh)~\mathrm dt\Big\|_{V}\\
 \leq \|y\|_{\mathcal C(\bar I,H)}+c\|\mathcal{I}_ty\|_{\mathcal{C}(\bar{I},V)}
 +\|y_\rh\|_{\mathcal C(\bar I,H)}+\Big\|\int_Iy_\rh~\mathrm dt\Big\|_{V}
\\
 \leq c_1\|u\|_{L^2(I,V^*)}\ \ \forall u\in L^2(I,V^*).
\end{multline*}
\par The last bound and estimate \eqref{est4over3} imply by the $K_{1/2,\infty}$-method {and equality \eqref{interpsp2}:
\begin{multline*}
 \|y-y_\rh\|_{\mathcal C(\bar I,H)}+\Big\|\int_I(\pi_h^1y-y_\rh)~\mathrm dt\Big\|_{V}\\
 \leq c\,(\tau+h)^{2/3}\|u\|_{(L^2(I,V^\ast),H^1(I,H))_{1/2,\infty}}
 \leq c_1\,(\tau+h)^{2/3}\|u\|_{SHW^{-1/2,1;2}(Q)}
%\|u\|_{B_{1/2}}\ \ \forall u\in B_{1/2}:=(L^2(I,V^\ast),H^1(I,H))_{1/2,\infty}.
\end{multline*}
for any $u\in SHW^{-1/2,1;2}(Q)$.
%Due to the simple embedding
%\[
%H^{1}(I,H^{(-1/2)})=(H^1(I;V^*),H^1(I;H))_{1/2,\infty}
%\hookrightarrow (L^2(I,V^\ast),H^1(I,H))_{1/2,\infty}
%\]
Applying} inequality \eqref{estfordty} we complete the proof.
\end{proof}

\begin{remark}
A priori stability bound \eqref{disp11} implies the unique solvability of the discrete state equation \eqref{discrete_state_equation1}-\eqref{discrete_state_equation2}.
\end{remark}
\begin{remark}
According to the given proof, for $\tilde{y}_\rh$ in place of $y_\rh$
the norms of $\mathbf{y}$ in \eqref{disp11} and \eqref{eesv11} can be weakened down to respectively
$\|\mathbf{y}\|_{H\times V^*}$ and $\|\mathbf{y}\|_{H^{(1/2)}\times H^{(-1/2)}}$.
For $\sigma\leq 1/4$, we have $\tilde{y}_\rh=y_\rh$.
The same can be shown for $y_\rh$ also for $\sigma>1/4$
provided that $\tau\alpha_h\leq c_0$ with any $c_0>0$.
\end{remark}
}
\section{Discrete control problem}\label{sec:disc prob}
First we introduce the discrete mapping \[\hat S_{\rh}\colon (u,y_0,y_1)\mapsto (y_{\rh},y_{\rh}(T),{\rho y_{Th}^1})\] and the discrete affine linear control-to-state mapping
\begin{equation*}\label{discrete_control_to_state_mapping}
S_{\rh}\colon \M\rightarrow \mathcal{Y}_\rh
={V_\rh\times V_h\times (\rho\times V_h)},~u\mapsto (y_\rh,y_\rh(T),{\rho y_{Th}^1})
\end{equation*}
defined by $S_{\rh}{u}=\hat S_{\rh}(u,0,0)+\hat S_{\rh}(0,y_0,y_1)$, with $\rho\times V_h=\{\rho\varphi; \varphi\in V_h\}$.
The mapping $S_{\rh}$ is a composition of
\[
 u\mapsto \vec u=\{\langle u,{e_{m,n}^\rh}
 \rangle_{\M,\,\C}\}_{m,n=1}^{M,N-1},\quad \M\rightarrow \R^{M(N-1)},
\]
{where $\{e_{m,n}^\rh\}$ is a basis in $V_\rh$,}
and $\vec u\mapsto (y_\rh,y_\rh(T),{\rho y_{Th}^1})$.
The {former} mapping is bounded due to {$e_{m,n}^\rh\in\C$}
and the {latter one} is finite dimensional. Thus $S_{\rh}$ is a bounded operator. Then we consider the following semi-discrete optimal control problem
\begin{equation}\label{semi_discrete_problem}\tag{$\mathcal P_\rh$}
 j_\rh(u)=\half\left\|S_{\rh}u-{\mathbf{z}}\right\|_{{\mathcal Y_h}}^{2}
 +\alpha\|u\|_{\M}\to\min_{u\in \M}
\end{equation}
with the squared semi-norm
corresponding to the inner product
\[
{
 ({\mathbf{z}},\tilde{{\mathbf{z}}})_{\mathcal Y_h}
 =(\rho z_1,\tilde{z}_1)_{\L}+(\rho z_2,\tilde{z}_2)_{H}+\big(A_h^{-1}z_3,A_h^{-1}\tilde{z}_3\big)_{\mathcal{V}_\kappa}\quad \forall {\mathbf{z}},\tilde{{\mathbf{z}}} \in \mathcal{Y}.}
\]
\par Using the similar argument as in the continuous case it can be shown that \eqref{semi_discrete_problem} has a solution $\bar u_\rh$ which is not unique in general, and due to {the optimality, the stability bound \eqref{disp11} and property \eqref{regularity_poisson} (for $\lambda=-1$)} one gets
\[
\alpha\|\bar u_\rh\|_{\M}\leq j_\rh(\bar u_\rh)\leq j_\rh(0)=\half\|S_\rh(0)-z\|_{\mathcal Y_h}^{2}
\leq c\,\big(\|\mathbf{y}\|_{V\times V^*}+\|z\|_{\mathcal Y}\big)^{2},
\]
cf. \eqref{uopt0}, and consequently
\begin{equation}
 {\|\bar u_\rh\|_{\M}\leq c\big(\|\mathbf{y}\|_{V\times V^*}+\|z\|_{\mathcal Y}\big)^2\leq C.}
\label{est u_h}
\end{equation}
\begin{theorem}
{Let $\mathbf{z}\in \mathcal{Y}$, $\mathbf{y}\in V\times H^{(-1/2)}$ and
$\bar u,\bar u_\rh\in \M$ be the optimal controls of respectively problems \eqref{measure_control_problem} and \eqref{semi_discrete_problem}.}
Then there holds
\[
\bar u_{\rh}\rightharpoonup^\ast\bar u~~\text{in}~~\M,\quad \|\bar u_{\rh}\|_{\M}\rightarrow \|\bar u\|_{\M}~~ \text{as}~~ {\rh}
\rightarrow 0.
\]
\end{theorem}
\begin{proof}
{Owing to \eqref{est u_h}}
there exists a sequence {$\{\rh_n\}$,
$\rh_n\rightarrow 0$,} and $u\in \M$ such that $\bar u_{\rh_n}\rightharpoonup^\ast u$ in $\M$ as $n\rightarrow\infty$.
%Next we prove that
{This implies the limit relation}
\begin{equation}\label{conv_tracking}
\|S_{\rh_n} \bar u_{\rh_n}-{\mathbf{z}}\|_{\mathcal Y_{\rh_n}}\rightarrow \|Su-{\mathbf{z}}\|_{\mathcal Y}.
\end{equation}
To {prove it}, we write the chain of inequalities
\begin{multline*}
\big|\|S_{\rh_n} \bar u_{\rh_n}-{\mathbf{z}}\|_{\mathcal Y_{\rh_n}}-\|Su-{\mathbf{z}}\|_{\mathcal Y}\big|\\
\leq \big|\|S_{\rh_n} \bar u_{\rh_n}-{\mathbf{z}}\|_{\mathcal Y_{\rh_n}}-\|Su-{\mathbf{z}}\|_{\mathcal Y_{\rh_n}}\big|+\big|\|Su-{\mathbf{z}}\|_{\mathcal Y_{\rh_n}}-\|Su-{\mathbf{z}}\|_{\mathcal Y}\big|
\\
\leq \|S_{\rh_n} \bar u_{\rh_n}-Su\|_{\mathcal Y_{\rh_n}}+\big|\|Su-{\mathbf{z}}\|_{\mathcal Y_{\rh_n}}-\|Su-{\mathbf{z}}\|_{\mathcal Y}\big|
\\
\leq \|S_{\rh_n} \bar u_{\rh_n}-S \bar u_{\rh_n}\|_{\mathcal Y_{\rh_n}}+\|S\bar u_{\rh_n}-Su\|_{\mathcal Y_{\rh_n}}+\big|\|Su-{\mathbf{z}}\|_{\mathcal Y_{\rh_n}}-\|Su-{\mathbf{z}}\|_{\mathcal Y}\big|.
\end{multline*}
The first term {on the right in the last inequality} converges to zero according to {the error estimate} \eqref{eesv11}. The convergence of the second term follows from the weak-star-to-strong continuity of $S\colon \M\rightarrow {\mathcal{Y}}$
and the stability of {$\pi_h^1$ in $V$.}
Finally, property \eqref{galerkin ort} for $\tilde{w}=w$
implies the convergence of the last term.

\par Then \eqref{conv_tracking} and the weak-star lower semicontinuity of $\|\cdot\|_{\M}$ in $\M$ implies
\[
{
j(u)\leq \liminf_{n\rightarrow\infty}j_{\rh_n}(\bar u_{\rh_n})\leq\limsup_{n\rightarrow\infty}j_{\rh_n}(\bar u_{\rh_n})
\leq \limsup_{n\rightarrow\infty}j_{\rh_n}(\bar u)=j(\bar u).}
\]
Thus, the uniqueness of $\bar u$ means that $u=\bar u$
and in addition
implies the convergence of the whole sequence $\bar u_{\rh}\rightharpoonup^\ast \bar u$ {in $\M$} as $\rh\rightarrow 0$.  Moreover, we have $j_{\rh}(\bar u_{\rh})\rightarrow j(\bar u)$. This and \eqref{conv_tracking} {lead to} $\|\bar u_{\rh}\|_{\M}\rightarrow \|\bar u\|_{\M}$.
\end{proof}
\par For convenience we set $F_h({\mathbf{z}})={(1/2)}\|{\mathbf{z}}\|_{{\mathcal Y_h}}^2$. In the following
the {directional} derivative of a functional $g\colon \M\rightarrow \R$ at $u\in\M$ in direction $\delta u\in\M$ {is denoted
by} $Dg(u)\delta u$.
In the case $Dg(u)\in \M^\ast$, $g$ is {the} Gateaux differentiable in $u$. Moreover, we make use of the convex subdifferential of $\|\cdot\|_{\M}$. Let $\hat u\in \M$ and $p\in\C$. Then there holds $p\in \partial\|\hat u\|_{\M}$ {if and only if}
\[
\langle p,u-\hat u\rangle_{\C,\M}+\alpha\|\hat u\|_{\M}\leq \alpha\|u\|_{\M}~~\forall u\in\M.
\]
An element $\bar u_\rh\in\M$ is an optimal solution of \eqref{semi_discrete_problem} if and only if $-D((F_h\circ S_\rh)(\bar u_\rh))\in \alpha\partial\|\bar u_\rh\|_{\M}$.
To calculate $D((F_h\circ S_\rh)(u))$ for $u\in \M$, we apply the Lagrange technique
and define the Lagrange functional by
\begin{multline*}
L(u,y_\rh, {y_{Th}^1},p_\rh,{p_{0h}^1})=F_{h}(y_\rh,y_\rh(T),{y_{Th}^1})-B_\sigma(y_\rh,p_\rh)-(\rho {y_{Th}^1},p_\rh(T))_{H}
\\[1mm]
+\langle u,p_\rh\rangle_{\M,\,\C}+{(\rho y^1,p_\rh(0))_H+(\rho(y_\rh(0)-y^0),p_{0h}^1)_H}
\end{multline*}
with $(p_\rh,{p_{0h}^1})\in V_\rh\times V_h$
{(where we base on identities \eqref{discrete_state_equation1}-\eqref{discrete_state_equation2}).}
We obviously have
\[
(F_h\circ S_\rh)(u)=L(u,S_\rh u,p_\rh,{p_{0h}^1})\quad\forall (p_\rh,{p_{0h}^1})\in V_\rh\times V_h.
\]
Thus there holds
\[
D((F_h\circ S_\rh)(u))\delta u=D_uL(u,y_\rh,{y_{Th}^1},p_\rh,{p_{0h}^1})\delta u =\langle p_\rh,\delta u\rangle_{\C,\,\M}~~\forall\delta u\in\M
\]
{provided that} $(p_\rh,{p_{0h}^1)}
\in {V}_\rh\times V_h$ is the solution of the discrete problem
\begin{multline*}
-D_{y_\rh}L(u,y_\rh,{y_{Th}^1},p_\rh,{p_{0h}^1})v
= B_\sigma(v,p_\rh)-(\rho(y_\rh-z_1),v)_{\L}\\
-(\rho(y_\rh(T)-z_2),v(T))_{H}-(\rho v(0),{p_{0h}^1})_{H}=0~~\forall v\in V_\rh
\end{multline*}
and
\[
-D_{{y_{Th}^1}}L(u,y_\rh,{y_{Th}^1},p_\rh,{p_{0h}^1})\varphi
= (\rho\varphi,p_\rh(T))_{H}-\big(\rho A_h^{-1}(\rho y^1_{Th}-z_3),\varphi\big)_{H}=0~~\forall\varphi\in V_h.
\]
Therefore \textit{the discrete optimality system} consists of \textit{the discrete state equation}
\begin{equation}
\begin{aligned}\label{bar eq state}
B_\sigma(\bar y_\rh,v)+(\rho{\bar{y}_{Th}^1},v(T))_{H}&
=\langle \bar u_\rh,v\rangle_{\M,\,\C}+(\rho y^1,v(0))_{H}&&\forall v\in {V}_\rh,\\
(\rho\bar y_\rh(0),\varphi)_{H}&=(\rho y^0,\varphi)_{H}&&\forall \varphi\in V_h,
\end{aligned}
\end{equation}
\textit{the discrete adjoint state equation}
\begin{equation}
\begin{aligned}\label{bar eq adj}
B_\sigma(v,\bar p_\rh)-(\rho v(0),\bar{p}_{0h}^1)_{H}&=(\rho(\bar y_\rh-z_1),v)_{\L}\\
&\hfill+(\rho(\bar{y}_\rh(T)-z_2),v(T))_{H}&&\forall v\in {V}_\rh,\\[1mm]
(\rho\varphi,\bar p_\rh(T))_{H}&=\big(\rho A_h^{-1}(\rho{\bar{y}_{Th}^1}-z_3),\varphi\big)_{H}&&\forall \varphi\in V_h
\end{aligned}
\end{equation}
and \textit{the discrete variational inequality}
\begin{equation}\label{discrete_sub_gradient_definition}
\langle -\bar p_\rh,u-\bar u_\rh\rangle_{\C,\,\M}+\|\bar u_\rh\|_{\M}\leq \|u\|_{\M}~~\forall u\in\M.
\end{equation}
\section{Stability and error estimates for the discrete adjoint state equation}\label{sec:error_adjoint}
{
We define \textit{the general discrete adjoint state equation}
\begin{align}
B_\sigma(v,p_\rh)-(\rho v(0),p_{0h}^1)_{H}&=(\rho(y-z_1),v)_{\L}\nonumber\\
&\hfill+(\rho(y(T)-z_2),v(T))_{H}&&\forall v\in V_\rh,
\label{bar eq adj 2a}\\[1mm]
(\rho\varphi,p_\rh(T))_{H}&=\big(\rho A_h^{-1}(\rho\partial_ty(T)-z_3),\varphi\big)_{H}&&\forall \varphi\in V_h.
\label{bar eq adj 2b}
\end{align}
Here $y$ is the solution to the state equation \eqref{state_equation}.
Clearly identity \eqref{bar eq adj 2b} means simply that $p_\rh(T)=A_h^{-1}q_T=\pi_1^hA^{-1}q_T$ with $q_T:=\rho\partial_ty(T)-z_3$.
\par Now we get a stability bound and error estimates in $\mathcal{C}(\bar{I},H)\times \mathcal{V}_{\kappa,h}^*$ and $\C$ for the discrete adjoint state equation.
\begin{proposition}\label{prop_est_adjoint}
Let $p=S^*\big(y-z_1,-(y(T)-z_2),A^{-1}(\rho\partial_ty(T)-z_3)\big)$ and $(p_\rh,p_{0h}^1)$ be the solution of
the corresponding general discrete adjoint state equation \eqref{bar eq adj 2a}-\eqref{bar eq adj 2b}.
\begin{enumerate}
\item If $y\in\mathcal{C}(\bar{I},H)\cap\mathcal{C}^1(\bar{I},V^*)$ and $\mathbf{z}\in\mathcal{Y}$, then the following stability bound holds
\begin{equation}
\|p_\rh\|_{\mathcal C(\bar I,V)}+\|\rho p_{0h}^1\|_{\mathcal{V}_{\kappa,h}^*}
\leq c\,\big(\|y-z_1\|_{L^2(I\times\Omega)}+\|y(T)-z_2\|_H+\|\rho\partial_ty(T)-z_3\|_{V^*}\big).
\label{eesv12}
\end{equation}
\item If $u\in L^2(I,V^*)$, $\mathbf{z}\in\mathcal{Y}$ and $\mathbf{y}\in H\times V^*$, then the following error estimate holds
\begin{equation}
 \|p-p_\rh\|_{\mathcal C(\bar I,H)}+\|\rho(\partial_tp(0)-p_h^0)\|_{\mathcal{V}_{\kappa,h}^*}\leq c(\tau+h)^{2/3}
 \big(\|u\|_{L^2(I,V^*)}+\|\mathbf{z}\|_\mathcal{Y}+\|\mathbf{y}\|_{H\times V^*}\big).
\label{eesv12a}
\end{equation}
\item If $u\in {H^{-1/2,0;2}(Q)}$,
$\mathbf{z}\in\mathcal{Y}^{1/2}:={\tilde{H}^{1/2,0;2}(Q)}\times H^{(1/2)}\times H^{(-1/2)}$ and $\mathbf{y}\in H^{(1/2)}\times H^{(-1/2)}$, then the following error estimate holds
\begin{gather}
\hspace{-5pt}
\|p-p_\rh\|_{\mathcal C(\bar{I}\times\bar{\Omega})}
\leq c(\tau+h)^{2/3}\big(\|u\|_{{H^{-1/2,0;2}(Q)}}
+\|\mathbf{z}\|_{\mathcal{Y}^{1/2}}
+\|\mathbf{y}\|_{H^{(1/2)}\times H^{(-1/2)}}\big).
\label{eesv13}
\end{gather}
\item If $u\in {SHW^{-1/2,1;2}(Q)}$,
$\mathbf{z}\in\mathcal{Y}^{3/2}:={\tilde{H}^{3/2,0;2}(Q)}\times H^{(3/2)}\times H^{(1/2)}$ and $\mathbf{y}\in H^{(3/2)}\times H^{(1/2)}$, then the following higher order error estimate holds
\begin{equation}
 \|p-p_\rh\|_{L^2(I,\mathcal{C}_0(\Omega))}
 \leq c(\tau+h)^{4/3}\big(\|u\|_{{SHW^{-1/2,1;2}(Q)}}
+\|\mathbf{z}\|_{\mathcal{Y}^{3/2}}
+\|\mathbf{y}\|_{H^{(3/2)}\times H^{(1/2)}}\big).
\label{eesv33}
\end{equation}
\end{enumerate}
\end{proposition}
\begin{proof}
1. According to \cite[Theorem 2.1 (2)]{Zlotnik94} the following energy bound hold
\[
 \|p_\rh\|_{\mathcal C(\bar I,V)}+\|\partial_tp_\rh\|_{L^\infty(\bar I,H)}
 \leq c\,\big(\|y-z_1\|_{L^2(I\times\Omega)}+\|p_\rh(T)\|_V+\|y(T)-z_2\|_H)
\]
for any $p_\rh(T)\in V_h$.
Using \eqref{est s0_1}, $A_h^{-1}=\pi_h^1A^{-1}$ and \eqref{regularity_poisson}
we get
\begin{equation}
 \|p_\rh(T)\|_V\leq c\|A^{-1}q_T\|_V\leq c_1\|q_T\|_{V^*}.
\label{estpT}
\end{equation}
By applying also the counterpart of inequalities \eqref{stab L2 a} we derive bound \eqref{eesv12}.

\par 2. The counterpart of the error estimate \eqref{est2over3} for the adjoint state equation case and bound \eqref{estpT} give
\begin{multline*}
 \|p-p_\rh\|_{\mathcal C(\bar I,H)}+\Big\|\int_I(\pi_h^1p-p_\rh)~\mathrm dt\Big\|_{V}\\
 \leq c\,(\tau+h)^{2/3}\big(\|y-z_1\|_{L^2(I,H)}+\|A_h^{-1}q_T\|_V+\|y(T)-z_2\|_H\big)
\\
 \leq c_1\,(\tau+h)^{2/3}\big(\|y\|_{\mathcal{C}(\bar{I},H)}+\|\partial_ty\|_{\mathcal{C}^1(\bar{I},V^*)}
 +\|\mathbf{z}\|_\mathcal{Y}\big).
\end{multline*}
Owing to inequality \eqref{estfordty} and Proposition \ref{prop:exist weaker} we obtain estimate \eqref{eesv12a}.

\par 3. Below we need the multiplicative inequalities
\begin{gather}
 \|w\|_{\mathcal C(\bar I\times\bar\Omega)}
 \leq c\|w\|_{\mathcal C(\bar I,H)}^{1/2}\|w\|_{\mathcal C(\bar I,V)}^{1/2}\ \ \forall w\in C(\bar I,V),
\label{inter_ineq}\\
 \|w\|_{L^2(I,\,\mathcal{C}_0(\Omega))}
 \leq c\|w\|_{L^2(I,H)}^{1/2}\|w\|_{L^2(I,V)}^{1/2}\ \ \forall w\in L^2(I,V).
\label{inter_ineq 2}
\end{gather}
\par Let $\check{p}_\rh$ be the auxiliary solution to \eqref{bar eq adj 2a} for $\check{p}_\rh(T)=\pi_h^0A^{-1}q_T$.
Owing to inequality \eqref{inter_ineq} and the stability bounds \cite[Theorem 2.1]{Zlotnik94} we get
\[
 \|p_\rh-\check{p}_\rh\|_{\mathcal C(\bar I\times\bar\Omega)}
 \leq c\|(p_\rh-\check{p}_\rh)(T)\|_{H_\tau^0}^{1/2}\|(p_\rh-\check{p}_\rh)(T)\|_V^{1/2}.
\]
Consequently, for $q_T\in H^{(\alpha-2)}$, by \eqref{est s0_1}, \eqref{V_est_poisson} and \eqref{H_est_poisson} the following chain of inequalities hold
\begin{multline*}
 \|p_\rh-\check{p}_\rh\|_{\mathcal C(\bar I\times\bar\Omega)}
 \leq c\big(\|(p_\rh-\check{p}_\rh)(T)\|_H^{1/2}\|(p_\rh-\check{p}_\rh)(T)\|_V^{1/2}+\tau^{1/2}\|(p_\rh-\check{p}_\rh)(T)\|_V\big)
\\
 \leq c_1\big(\|r_hA^{-1}q_T\|_H^{1/2}\|r_hA^{-1}q_T\|_V^{1/2}+\tau^{1/2}\|r_hA^{-1}q_T\|_V\big)\\
 \leq c_2(\tau+h)^{\alpha-1/2}\|q_T\|_{H^{(\alpha-2)}}
\end{multline*}
for $1\leq\alpha\leq 2$.
Thus it is enough to prove error estimates \eqref{eesv13} and \eqref{eesv33} for $\check{p}_\rh$ instead of $p_\rh$.
\par According to \cite[Theorem 5.3 and estimate (5.18)]{Zlotnik94} we have the error estimate
\begin{multline}\label{eesv17}
 \|i_\tau p-\check{p}_\rh\|_{\mathcal C(\bar{I}\times\bar{\Omega})}
 =\|p-\check{p}_\rh\|_{\mathcal C_\tau(\bar{I},\,\mathcal C(\bar{\Omega}))}
 :=\max_{0\leq m\leq M}\|(p-\check{p}_\rh)(t_m)\|_{\mathcal C(\bar{\Omega})}
\\[1mm]
 \leq c(\tau+h)^{2(\alpha-1/2)/3}\bigl(\|y-z_1\|_{L^2(I,\,H^{(\alpha-1)})}
 +\|y(T)-z_2\|_{H^{(\alpha-1)}}
 +\|q_T\|_{H^{(\alpha-2)}}\bigr)
\end{multline}
for $\alpha=1,2$. We emphasize that due to \cite[Theorem 4.3 (2) (e)]{Zlotnik94} and \eqref{disp12b}
this estimate holds for $\check{p}_\rh(T)=\pi_h^0A^{-1}q_T$.
\par Inequality \eqref{inter_ineq}, Proposition \ref{prop:exist weak} (applied to the adjoint state problem) and property \eqref{regularity_poisson} imply the following error estimate for the time interpolation
\begin{multline}\label{eesv23}
 \|p-i_\tau p\|_{\mathcal C(\bar{I}\times\bar{\Omega})}
 \leq c\bigl(\tau\|\partial_tp\|_{\mathcal C(\bar{I},H)}\bigr)^{1/2}
 \|(\tau\partial_t)^{\alpha-1}p\|_{\mathcal C(\bar{I},V)}^{1/2}
\\[1mm]
 \leq c_1\tau^{\alpha/2}\bigl(\|y-z_1\|_{L^2(I,\,H^{(\alpha-1)})}
 +\|y(T)-z_2\|_{H^{(\alpha-1)}}
 +\|q_T\|_{H^{(\alpha-2)}}\bigr),
\end{multline}
for $\alpha=1,2$.
Owing to estimates \eqref{eesv17} and \eqref{eesv23} {and} Propositions \ref{prop:exist weaker} and \ref{prop:exist weak} we get
\begin{multline}
 \|p-\check{p}_\rh\|_{\mathcal C(\bar{I}\times\bar{\Omega})}
 \leq c(\tau+h)^{2(\alpha-1/2)/3}\big(\|y\|_{C(\bar{I},\,H^{(\alpha-1)})}+\|\partial_ty\|_{C(\bar{I},\,H^{(\alpha-2)})}
 +\|\mathbf{z}\|_{\mathcal{Y}^{(\alpha-1)}}\big)
\\
 \leq c_1(\tau+h)^{2(\alpha-1/2)/3}\big(\|u\|_{L^2(I,H^{(\alpha-2)})}
 +\|\mathbf{y}\|_{H^{(\alpha-1)}\times H^{(\alpha-2)}}
+\|\mathbf{z}\|_{\mathcal{Y}^{(\alpha-1)}}\big),
% ,~~\alpha=1,2,
\label{eesv25}
\end{multline}
{for $~\alpha=1,2$}, \Blue{where $\mathcal{Y}^{(0)}:=\mathcal{Y}$ and $\mathcal{Y}^{(1)}:=L^2(I,H)\times V\times H$.}
%{\color{green} where $\mathcal{Y}^{(\eta)}:=L^2(I,H^{(\eta)})\times H^{(\eta)}\times H^{(\eta-1)}$.}
\par Applying the $K_{1/2,\infty}$-method {together with equalities \eqref{interpsp1} and \eqref{interpspVH} for $\ell=0$},
we get \eqref{eesv13} for $\check{p}_\rh$ in the role of $p_\rh$.

\par 4. First notice that the multiplicative inequality \eqref{inter_ineq 2},
Proposition \ref{prop:exist weak} (2) (applied for the adjoint state problem) and property \eqref{regularity_poisson} imply another error estimate for the time interpolation
\begin{multline*}
 \|p-i_\tau p\|_{L^2(I,\,\mathcal{C}_0(\Omega))}\leq c\bigl(\tau^2\|\partial_{tt}p\|_{L^2(I,H)}\bigr)^{1/2}
 \big(\tau\|\partial_tp\|_{\mathcal C(\bar{I},V)}\big)^{1/2}
\\[1mm]
 \leq c_1\tau^{3/2}\big(\|y-z_1\|_{L^2(I,V)}+\|y(T)-z_2\|_V+\|q_T\|_H\big).
\end{multline*}
Then Proposition \ref{prop:exist weak} (1) leads to
\begin{gather}
 \|p-i_\tau p\|_{L^2(I,\,\mathcal{C}_0(\Omega))}
 \leq c\tau^{3/2}\big(\|u\|_{H^1(I,V^*)}+\|z\|_\mathcal{Y}+\|\mathbf{y}\|_{V\times H}\big).
\label{eesv31}
\end{gather}
\par Next we derive the error estimate
\begin{gather}
 \|i_\tau p-\check{p}_\rh\|_{\mathcal C(\bar{I}\times\bar{\Omega})}
 \leq c(\tau+h)^{4/3}\big(\|u\|_{{SHW^{-1/2,1;2}(Q)}}
 +\|\mathbf{z}\|_{\mathcal{Y}^{3/2}}+\|\mathbf{y}\|_{H^{(3/2)}\times H^{(1/2)}}\big).
\label{eesv33a}
\end{gather}
According to \cite[Theorem 5.3 and estimate (5.18)]{Zlotnik94} {and equality \eqref{interpspVH} for $\ell=1$ together with} Propositions \ref{prop:exist weaker} and \ref{prop:exist weak} the following three estimates hold
\begin{equation}
 \|i_\tau p-\check{p}_\rh\|_{\mathcal C(\bar{I}\times\bar{\Omega})}
 \leq c(\tau+h)^{4/3}\|\mathbf{z}\|_{\mathcal{Y}^{3/2}}\ \ \text{for}\ \ u=0,\ \mathbf{y}=0,
\label{eesv33b}
\end{equation}
\begin{multline*}
 \|i_\tau p-\check{p}_\rh\|_{\mathcal C(\bar{I}\times\bar{\Omega})}
  \leq c(\tau+h)\|y\|_{H^1(I,H)}
\\
  \leq c_1(\tau+h)\big(\|u\|_{H^1(I,V^*)}+\|\mathbf{y}\|_{V\times H}\big) \ \text{{for}}\ \ \mathbf{z}=0,
\end{multline*}
\begin{multline*}
 \|i_\tau p-\check{p}_\rh\|_{\mathcal C(\bar{I}\times\bar{\Omega})}
 \leq c(\tau+h)^{5/3}\big(\|\partial_{tt}y\|_{L^2(I,H)}+\|\mathbf{y}\|_{V\times H}\big)\
\\
 \leq c_1(\tau+h)^{5/3}\big(\|u\|_{H^1(I,H)}+\|\mathbf{y}\|_{V^2\times V}\big)
 \ \text{for}\ \ \mathbf{z}=0
\nonumber
\end{multline*}
and for $\check{p}_\rh(T)=\pi_h^0A^{-1}q_T$ (for the same reason as above).
Then applying the $K_{1/2,\infty}$-method
to the two last estimates {and using equality \eqref{interpsp2}} we get
\[
 \|i_\tau p-\check{p}_\rh\|_{\mathcal C(\bar{I}\times\bar{\Omega})}
 \leq c(\tau+h)^{4/3}\big(\|u\|_{{SHW^{-1/2,1;2}(Q)}}+\|\mathbf{y}\|_{H^{(3/2)}\times H^{(1/2)}}\bigr)\ \ \text{for}\ \ \mathbf{z}=0.
\]
By combining this estimate and \eqref{eesv33b} we obtain \eqref{eesv33a}.
\par Estimates \eqref{eesv31} and \eqref{eesv33a} imply
\[
  \|p-\check{p}_\rh\|_{L^2(I,\,\mathcal{C}_0(\Omega))}
 \leq c(\tau+h)^{4/3}\big(\|u\|_{{SHW^{-1/2,1;2}(Q)}}
 +\|\mathbf{z}\|_{\mathcal{Y}^{3/2}}+\|\mathbf{y}\|_{H^{(3/2)}\times H^{(1/2)}}\big)
\]
that completes the proof of \eqref{eesv33} for $\check{p}_\rh$ in the role of $p_\rh$.
\end{proof}
\begin{remark}
\label{inisolv dase}
A priori stability bound \eqref{disp11} (taken for $y=0$) implies the unique solvability of the general discrete adjoint state equation
\eqref{bar eq adj 2a}-\eqref{bar eq adj 2b}.
\end{remark}
\section{Error estimates for the state variable}\label{sec:error_opt_state}

\par We introduce the discrete adjoint control-to-state operator
$S^\star_{\rh}\colon \L\times V\times H\rightarrow V_{\rh}$, $(\phi,p^1,p^0)\mapsto p_\rh$ defined by
\[
 B_\sigma(v,p_\rh)=(\rho\phi,v)_{\L}-(\rho p^1,v(T))_{H}\ \ \forall v\in V_\rh,\ v(0)=0
\]
with $p_\rh(T)=\pi_h^0p^0$. Similarly to bound \eqref{eesv12} and Remark \ref{inisolv dase} it is well defined and satisfies
\[
\|S^\star_{\rh}(\phi,p_0,p_1)\|_{\mathcal C(\bar I,V)}
\leq c\,\big(\|\phi\|_{L^2(I\times\Omega)}+\|p^0\|_V+\|p^1\|_H\big).
\]
{
Let for brevity $W,W_h\colon\mathcal Y\rightarrow \mathcal Y^\ast$ be the duality mappings defined by}
\[
 W(y_1,y_2,y_3)=(y_1,{-}y_2,A^{-1}y_3),\ \
 W_h(y_1,y_2,y_3)=(y_1,{-}y_2,A_h^{-1}y_3)
\]
for any $(y_1,y_2,y_3)\in \mathcal Y$.
With this notation, the function
\[
p_\rh=S_\rh^*\big(y-z_1,-(y(T)-z_2),A_h^{-1}(\rho\partial_ty(T)-z_3)\big)=S_\rh^*W_h(Su-z)
\]
solves
the general discrete adjoint state equation \eqref{bar eq adj 2a}-\eqref{bar eq adj 2b}.
\begin{proposition}\label{prop:error_state}
Let $\mathbf{z}\in\mathcal{Y}$ and $\mathbf{y}\in V\times V^*$.
Then the following estimate holds
\begin{equation}
 \|S\bar u-S_\rh\bar u_\rh\|_{\mathcal Y_h}\leq \|S\bar u-S_\rh\bar u\|_{\mathcal Y_h}
 +C\|S^\star W(S\bar u-{\mathbf{z}})-S^\star_\rh W_h(S\bar u-{\mathbf{z}})\|_{\C}^{1/2}.
\label{eesv9}
\end{equation}
\end{proposition}
\begin{proof}
{We recall that $\bar{p}=S^\star W(S\bar u-\mathbf{z})$ and $\bar{p}_\rh=S^\star_\rh W_h(S_\rh\bar u_\rh-\mathbf{z})$ and}
test the continuous subgradient condition \eqref{sub_gradient_definition} with the discrete optimal control $\bar u_\rh$ and the discrete subgradient condition \eqref{discrete_sub_gradient_definition} with the continuous optimal control $\bar u$. Then we subtract the first inequality from the second one and get
\[
\langle \bar u-\bar u_\rh,{\bar{p}-\bar{p}_\rh}
\rangle_{\M,\,\C}\leq 0.
\]
We {define $\hat{p}_\rh:=S^\star_\rh W_h(S\bar u-\mathbf{z})$, insert it between $\bar{p}$ and $\bar{p}_\rh$ and obtain}
\begin{equation}\label{difference_subgradient_cond_1}
0\leq\langle \bar u_\rh-\bar u,{\bar{p}-\hat{p}_\rh}
\rangle_{\M,\,\C}\\
+\langle \bar u_\rh-\bar u,{\hat{p}_\rh-\bar{p}_\rh}
 \rangle_{\M,\,\C}.
\end{equation}
For convenience we introduce the variables $(\hat y_\rh,\hat y_\rh(T),{\rho\hat{y}_{Th}^1})=S_\rh\bar u$ and
remark that the {state} equations for $(\bar y_\rh,{\bar{y}_{Th}^1})$ and $(\hat y_\rh,{\hat{y}_{Th}^1})$ have the same initial data.
{With the help of them} we rewrite
the second term {on the right
in \eqref{difference_subgradient_cond_1} taking first the difference of
the discrete state equations \eqref{bar eq state} and \eqref{discrete_state_equation1} (taken for $(\hat y_\rh,\hat{y}_{Th}^1)$)
for $v=\hat p_\rh-\bar p_\rh$,
next the difference of the discrete adjoint state equations \eqref{bar eq adj} and \eqref{bar eq adj 2a}-\eqref{bar eq adj 2b} (taken for $\hat{p}_\rh$)
for $v=\bar y_\rh-\hat y_\rh$ and $\varphi=\bar{y}_{Th}^1-\hat{y}_{Th}^1$}
and finally using \eqref{discrete_laplace_equation}
\begin{multline*}
\langle \bar u_\rh-\bar u,\hat p_\rh-\bar p_\rh\rangle_{\M,\,\C}
=B_\sigma(\bar y_\rh-\hat y_\rh,\hat p_\rh-\bar p_\rh)
+(\rho({\bar{y}_{Th}^1-\hat{y}_{Th}^1}),(\hat p_\rh-\bar p_\rh)(T))_{H}
\\[1mm]
=(\rho(\bar y_\rh-\hat y_\rh),\bar y-\bar y_\rh)_{\L}
+(\rho(\bar y_\rh-\hat y_\rh)(T),(\bar y-\bar y_\rh)(T))_{H}\\
+{\big(\rho(\bar{y}_{Th}^1-\hat{y}_{Th}^1),A_h^{-1}\big(\rho(\partial_t\bar y(T)-\bar{y}_{Th}^1)\big)\big)_{H}}
\\[1mm]
={(S_{\rh}\bar{u}_{\rh}-S_{\rh}\bar{u},S\bar{u}-S_{\rh}\bar{u}_{\rh})_{\mathcal{Y}_h}.}
\end{multline*}
{Further we easily get}
\begin{multline*}
\langle \bar u_\rh-\bar u,\hat p_\rh-\bar p_\rh\rangle_{\M,\,\C}
=(S\bar{u}-S_{\rh}\bar{u}_{\rh},S_{\rh}\bar{u}_{\rh}-S_{\rh}\bar{u})_{{\mathcal{Y}_h}}\\
=(S\bar{u}-S_{\rh}\bar{u}_{\rh},S\bar{u}-S_{\rh}\bar{u})_{{\mathcal{Y}_h}}
-\|S\bar{u}-S_{\rh}\bar{u}_{\rh}\|_{{\mathcal{Y}_h}}^2
\\[1mm]
\leq \half\|S\bar{u}-S_{\rh}\bar{u}\|_{{\mathcal{Y}_h}}^2
-\half\|S\bar{u}-S_{\rh}\bar{u}_{\rh}\|_{{\mathcal{Y}_h}}^2.
\end{multline*}
Thus \eqref{difference_subgradient_cond_1} implies
\begin{multline*}
 \|S\bar{u}-S_{\rh}\bar{u}_{\rh}\|_{{\mathcal{Y}_h}}^2
 \leq 2\langle \bar u_\rh-\bar u,{\bar{p}-\hat{p}_\rh}
 \rangle_{\M,\,\C}
 +\|S\bar{u}-S_{\rh}\bar{u}\|_{{\mathcal{Y}_h}}^2
\\[1mm]
 \leq 2(\|\bar u_\rh\|_{\M}+\|\bar u\|_{\M})\|{\bar{p}-\hat{p}_\rh}\|_{\C}
 +\|S\bar{u}-S_{\rh}\bar{u}\|_{{\mathcal{Y}_h}}^2.
\end{multline*}
Finally by applying bounds \eqref{uopt} and \eqref{est u_h} we derive \eqref{eesv9}.
\end{proof}
This proposition is important since it allows one to derive estimates for
$\bar y-\bar y_\rh$ with the help of the above error estimates for the discrete state and adjoint state equations.
\begin{theorem}
\label{prop: est y y rh}
\begin{enumerate}
\item Let $\M = {\LwM}$,
{$\mathbf{z}\in\mathcal{Y}^{1/2}$ and $\mathbf{y}\in V\times H$.}
Then the following error estimate holds
\begin{equation}
\|\bar{y}-\bar{y}_\rh\|_{L^2(I\times\Omega)}+\|(\bar y-\bar y_\rh)(T)\|_{H}
+\|\rho\big(\partial_t\bar y(T)-{\bar{y}_{Th}^1}\big)\|_{\mathcal{V}_{\kappa,h}^*}
\leq {C(\tau+h)^{1/3}.}
\label{eesv15}
\end{equation}
\item Let $\M = \mathcal{M}(\Omega,L^2(I))$,
{$\mathbf{z}\in\mathcal{Y}^{3/2}$
and $\mathbf{y}\in H^{(3/2)}\times H^{(1/2)}$.}
Then the following {higher order} error estimate holds
\begin{equation}
\|\bar{y}-\bar{y}_\rh\|_{L^2(I\times\Omega)}+\|(\bar y-\bar y_\rh)(T)\|_{H}
+\|\rho\big(\partial_t\bar y(T)-{\bar{y}_{Th}^1}\big)\|_{\mathcal{V}_{\kappa,h}^*}\\[1mm]
\leq C(\tau+h)^{2/3}.
\label{eesv29}
\end{equation}
\end{enumerate}
\end{theorem}
\begin{proof}
1. {Let us base on Proposition \ref{prop:error_state}.
First, Proposition \ref{prop: est stat eq} (4) implies
\[
\|S\bar u-S_\rh\bar u\|_{\mathcal Y_h}\leq c(\tau+h)^{1/3}\big(\|\bar u\|_{{H^{-1/2,0;2}(Q)}}
+{\|\mathbf{y}\|_{V\times H}}\big).
\]
Second, Proposition \ref{prop_est_adjoint} (3) leads to
\begin{multline*}
\|S^\star W(S\bar u-{\mathbf{z}})-S^\star_\rh W_h(S\bar u-{\mathbf{z}})\|_{\C}\\
\leq c(\tau+h)^{2/3}\big(\|\bar{u}\|_{{H^{-1/2,0;2}(Q)}}
+\|\mathbf{z}\|_{\mathcal{Y}^{1/2}}
+\|\mathbf{y}\|_{H^{(1/2)}\times H^{(-1/2)}}\big).
\end{multline*}
Now owing to Proposition \ref{prop:error_state}, {embedding \eqref{embedding2}} and bound \eqref{uopt} for $\bar u$ error estimate \eqref{eesv15} is proved.
}
\par 2. First, Proposition \ref{prop: est stat eq} (3) implies
\[
\|S\bar u-S_\rh\bar u\|_{\mathcal Y_h}\leq c(\tau+h)^{2/3}\left(\|\bar u\|_{{SHW^{-1/2,1;2}(Q)}}
+{\|\mathbf{y}\|_{H^{(3/2)}\times H^{(1/2)}}}
\right).
\]
Second, Proposition \ref{prop_est_adjoint} (4) leads to
\begin{multline*}
\|S^\star W(S\bar u-{\mathbf{z}})-S^\star_\rh W_h(S\bar u-{\mathbf{z}})\|_{\C}\\
\leq c(\tau+h)^{4/3}\big(\|\bar{u}\|_{{SHW^{-1/2,1;2}(Q)}}
+\|\mathbf{z}\|_{\mathcal{Y}^{3/2}}
+\|\mathbf{y}\|_{H^{(3/2)}\times H^{(1/2)}}\big).
\end{multline*}
Now owing to Proposition \ref{prop:error_state}, {embedding \eqref{embedding2 1}} and Theorem \ref{thm:imp_reg_control} for $\bar u$ error estimate \eqref{eesv29} is proved too.
\end{proof}
\begin{remark}
Note that our error bounds could be better provided that one would improve the last term on the right in \eqref{eesv9} by increasing the power $1/2$.
But this seems a complicated problem.
\end{remark}

\section{Error estimate for the cost functional}\label{sec:error_cost_func}
In this section we derive error estimate for the cost functional. We first observe the inequalities
\[
j(\bar u)\leq j(\bar u_\rh){,}\ \
j_\rh(\bar u_\rh)\leq j_\rh(\bar u)
\]
which can be equivalently rewritten in the form
\begin{equation}\label{est_cost}
j(\bar u)-j_\rh(\bar u)\leq j(\bar u)-j_\rh(\bar u_\rh)\leq j(\bar u_\rh)-j_\rh(\bar u_\rh).
\end{equation}
Therefore{, to bound $|j(\bar u)-j_\rh(\bar u_\rh)|$ below we apply the following result.}
\begin{proposition}\label{prop:err cost func}
Let $\mathbf{y}\in V\times H$. Then for any $u\in\M$
\begin{multline}\label{est func diff}
|j(u)-j_\rh(u)|
\leq c\Big(\|Su-S_\rh u\|_{{\mathcal Y_h}}^2+\big(\|u\|_{\M}+{\|\mathbf{y}\|_{V\times H}}\big)
\big(\|p-p_\rh\|_{\C}
+\|p(0)-p_{\rh}(0)\|_H
\\
+h\|\partial_tp(0)\|_H
+\|\rho(\partial_t{p(0)-{p}_{0h}^1})\|_{\mathcal{V}_{\kappa,h}^*}\big)
+\|{r_hA^{-1}\big(\rho\partial_t y(T)\big)}\|_V^2
+\|{r_hA^{-1}z_3}\|_V^2\Big)
\end{multline}
with $(y,y(T),{\rho}\partial_t y(T))=Su$
and the same $p$ and $(p_\rh,{p}_{0h}^1)$ as in Proposition \ref{prop_est_adjoint}.
\end{proposition}
\begin{proof}
Let $u\in \M$. According to the definitions of the continuous and discrete cost functionals and {property \eqref{galerkin ort}
for $\tilde{w}=w$ and $\tilde{w}_h=w_h$ we get
\begin{multline}
j(u)-j_\rh(u)=\half\|Su-{\mathbf{z}}\|_{\mathcal Y}^2
-\half\|S_\rh u-{\mathbf{z}}\|_{{\mathcal Y_h}}^2
\\[1mm]
=\half(Su-S_\rh u,Su+S_\rh u-2{\mathbf{z}})_{{\mathcal Y_h}}
+\half\|A^{-1}(\rho\partial_t y(T)-z_3)\|_{{\mathcal{V}_\kappa}}^2
-\half\|A_h^{-1}(\rho\partial_ty(T)-z_3)\|_{{\mathcal{V}_\kappa}}^2
\\[1mm]
=-\half\|Su-S_\rh u\|_{{\mathcal Y_h}}^2+(Su-S_\rh u,S u-{\mathbf{z}})_{{\mathcal Y_h}}
+\half\|{r_hA^{-1}}\big(\rho\partial_ty(T)-z_3\big)\|_{{\mathcal{V}_\kappa}}^2.
\label{jujuh}
\end{multline}
\par We set $p_{Th}:=A_h^{-1}(\rho\partial_ty(T)-z_3)$.

Owing to the adjoint problem \eqref{solution_transposition} with
\[(\phi,p^1,p^0)=W(y-z_1,y(T)-z_2,\rho\partial_ty(T)-z_3)\]
we {have}
\begin{multline*}
 (Su,Su-\mathbf{z})_{\mathcal Y_h}-(A_h^{-1}(\rho\partial_ty(T)),p_{Th})_{\mathcal{V}_\kappa}
 =(\rho y,y-z_1)_{L^2(I\times\Omega)}
 +(\rho y(T),y(T)-z_2)_H\\
 =\langle u,p\rangle_{\M,\,\C}+(\rho y^1,p(0))_H
 -\langle\rho\partial_ty(T),p^0\rangle_\Omega
 -(\rho y^0,\partial_tp(0))_H.
\end{multline*}
Similarly owing to the general discrete adjoint state equation \eqref{bar eq adj 2a}-\eqref{bar eq adj 2b} for $v=y_\rh$
and the discrete state equation \eqref{discrete_state_equation1}-\eqref{discrete_state_equation2} for $v=p_\rh$ and $\varphi=p_{0h}^1$ we get
\begin{multline*}
 (S_\rh u,Su-\mathbf{z})_{\mathcal Y_h}-(A_h^{-1}(\rho y_{Th}^1),p_{Th})_{\mathcal{V}_\kappa}
 =(\rho y_\rh,y-z_1)_{L^2(I\times\Omega)}
 +(\rho y_\rh(T),y(T)-z_2)_H
\\
 =B_\sigma(y_\rh,p_\rh)-(\rho y_\rh(0),p_{0h}^1)_H
\\
 =\langle u,p_\rh\rangle_{\M,\,\C}+(\rho y^1,p_\rh(0))_H
 -(\rho y_{Th}^1,p_\rh(T))_H
 -(\rho y^0,p_{0h}^1)_H.
\end{multline*}
In addition owing to the definitions \eqref{bar eq adj 2b} of $p_\rh(T)$ and \eqref{discrete_laplace_equation} of $A_h^{-1}$, we can write
\[
 (\rho y_{Th}^1,p_\rh(T))_H=(\rho y_{Th}^1,p_{Th})_H=(A_h^{-1}(\rho y_{Th}^1),p_{Th})_{\mathcal{V}_\kappa}.
\]
Consequently we obtain}
\begin{multline}
(Su-S_\rh u,Su-{\mathbf{z}})_{{\mathcal Y_h}}=(Su,Su-{\mathbf{z}})_{\mathcal Y_h}
-(S_\rh u,Su-{\mathbf{z}})_{{\mathcal Y_h}}
\\[1mm]
=\langle u,p-p_\rh\rangle_{\M,\,\C}
-(\rho y^0,\partial_tp(0)-{p_{0h}^1})_{H}+(\rho y^1,p(0)-p_{\rh}(0))_{H}
\\[1mm]
+(A_h^{-1}(\rho\partial_ty(T)),p_{Th})_{\mathcal{V}_\kappa}-\langle\rho\partial_ty(T),p^0\rangle_\Omega.
\label{susrh}
\end{multline}
In addition using property \eqref{galerkin ort} we derive
\begin{multline}
 (A_h^{-1}(\rho\partial_ty(T)),p_{Th})_{\mathcal{V}_\kappa}-\langle\rho\partial_ty(T),p^0\rangle_\Omega\\
 =(A_h^{-1}(\rho\partial_ty(T)),p_{Th})_{\mathcal{V}_\kappa}-(A^{-1}(\rho\partial_ty(T)),p^0)_{\mathcal{V}_\kappa}
\\
 =-(r_hA^{-1}(\rho\partial_ty(T)),r_hA^{-1}(\rho\partial_ty(T)-z_3))_{\mathcal{V}_\kappa}.
\label{dahm1}
\end{multline}
Next, for the term $(\rho y^0,\partial_tp(0)-{p_{0h}^1})_{H}$ in \eqref{susrh} we have
\begin{multline}
|(\rho y^0,\partial_tp(0)-{p_{0h}^1})_{H}|
=|(\rho(y^0-\pi_h^0y^0),\partial_tp(0)-{p_{0h}^1})_{H}
+(\rho \pi_h^0y^0,\partial_tp(0)-{p_{0h}^1})_{H}|
\\
{\leq}
|(\rho(y^0-\pi_h^0y^0),\partial_tp(0))_{H}|
+{c\|\pi_h^0y^0\|_V\|\rho(\partial_tp(0)-p_{0h}^1)\|_{\mathcal{V}_{\kappa,h}^*}}
\\
\leq c_1\,\|y^0\|_V
\Big(h\|\partial_tp(0)\|_{H}+\|\rho(\partial_tp(0)-{p_{0h}^1})\|_{\mathcal{V}_{\kappa,h}^*}\Big)
\label{popoh}
\end{multline}
due to the bounds $\|y^0-\pi_h^0y^0\|_{H_\rho}\leq\|y^0-\pi_h^1y^0\|_{H_\rho}$, \eqref{H_est_poisson} and
\eqref{est s0_1}.
Clearly also $|(\rho y^1,p(0)-p_{\rh}(0))_H|\leq\|y^1\|_H\|p(0)-p_{\rh}(0)\|_H$. Finally from \eqref{jujuh}-\eqref{popoh} we derive
\eqref{est func diff}.
\end{proof}
Now we prove for the cost functional a higher order error estimate than \eqref{eesv15} for the state variable {in the case $\M = {\LwM}$.}
\begin{theorem}\label{prop:conv func}
Let $\M = {\LwM}$, $\mathbf{z}\in\mathcal{Y}^{1/2}$ and $\mathbf{y}\in V\times {H}$.
Then the following error estimate for the cost functional holds
\[
  |j(\bar u)-j_\rh(\bar u_\rh)|
 \leq {C}(\tau+h)^{2/3}.
\]
\end{theorem}
\begin{proof}
Let us base on Proposition \ref{prop:err cost func} and take any $u\in {\LwM}$.
Owing to Proposition \ref{prop: est stat eq} {(2)} we have
\begin{equation*}
\|Su-S_\rh u\|_{\mathcal Y_h}
 \leq c(\tau+h)^{1/3}{\big(\|u\|_{{H^{-1/2,0;2}(Q)}}
+\|\mathbf{y}\|_{V\times H^{(-1/2)}}\big)}.
\end{equation*}
Proposition \ref{prop_est_adjoint} (3) leads to
\[
 \|p-p_\rh\|_{\mathcal C(\bar I \times \bar{\Omega})}
 \leq c(\tau+h)^{2/3}\big(\|u\|_{{H^{-1/2,0;2}(Q)}}
+\|\mathbf{z}\|_{\mathcal{Y}^{1/2}}
+\|\mathbf{y}\|_{H^{(1/2)}\times H^{(-1/2)}}\big).
\]
Owing to Propositions \ref{prop:exist weak}(1) (applied to the adjoint state problem) and \ref{prop:exist weaker} we have
\[
 \|\partial_tp(0)\|_H\leq\|\partial_tp\|_{C(\bar{I},H)}
\leq c\big(\|u\|_{L^2(I,V^*)}+\|\mathbf{y}\|_{H\times V^*}+\|\mathbf{z}\|_\mathcal{Y}\big)
\]
(like in estimates \eqref{eesv23}-\eqref{eesv25} for $\alpha=1$).
By using estimate \eqref{V_est_poisson} for $\lambda=-1/2$ we obtain
\[
{\|{r_hA^{-1}\big(\rho\partial_t y(T)\big)}\|_V
+\|{r_hA^{-1}z_3}\|_V
\\
\leq ch^{1/2}\big(\|\partial_ty(T)\|_{H^{(-1/2)}}+\|z_3\|_{H^{(-1/2)}}\big).}
\]
\par By collecting all these estimates together with embedding \eqref{embedding2}, Proposition \ref{prop_est_adjoint} (2) to bound $\|\rho(\partial_t{p(0)-{p}_{0h}^1})\|_{\mathcal{V}_{\kappa,h}^*}$ and applying Proposition \ref{prop:err cost func}, we derive
\[
 |j(u)-j_\rh(u)|\leq c(\tau+h)^{2/3}\big(\|u\|_{{\LwM}}+\|\mathbf{z}\|_{\mathcal{Y}^{1/2}}
+\|\mathbf{y}\|_{V\times H}\big)^2.
\]
Owing to inequalities \eqref{est_cost} together with bounds \eqref{uopt} for $\bar{u}$ and \eqref{est u_h} for $\bar{u}_\rh$ the proof is complete.
\end{proof}
\begin{remark}
In the case $\M=\mathcal{M}(\Omega,L^2(I))$ we know that $\bar u \in {\mathcal{C}^1(\bar{I},\mathcal{M}(\Omega))}$ (cf. Theorem \ref{thm:imp_reg_control}).
The lack of the corresponding bound
{at least $\|\bar{u}_\rh\|_{SHW^{-1/2,1;2}(Q)}\leq C$ at the discrete level}
does not allow us to prove the error estimate $|j(\bar u)-j_\rh(\bar u_\rh)|\leq C(\tau+h)^{4/3}$.
The estimate $|j(\bar u)-j_\rh(\bar u_\rh)|\leq C(\tau+h)^{2/3}$ follows directly from~\eqref{eesv29}.
\end{remark}

\section{Time-stepping formulation}
\label{tistfo}
In this section we discuss the time-stepping formulation of the discrete state equation \eqref{discrete_state_equation1}-\eqref{discrete_state_equation2} and the discrete adjoint state equation \eqref{bar eq adj}.
{
We introduce the piecewise-linear ``hat'' functions such that
$e^\tau_m(t_k)=\delta_{m,k}$ for any $k,m=0,\ldots,M$, where $\delta_{m,k}$ is the Kroneker delta.
We recall that $e^\tau_m$ are ``half'' hat functions for $m=0,M$.
There holds $V_\tau=\spa\{e_0^\tau,\ldots,e_M^\tau\}$.
Similarly, we introduce the spatial hat functions such that $e^h_j(x_k)=\delta_{j,k}$ for any $j=1,\ldots,N-1$ and $k=0,\ldots,N$; then $V_h=\spa\{e^h_1,\ldots,e^h_{N-1}\}$.
\par Then  the approximate state variable $y_\rh\in V_\rh$ can be represented in the following forms
\begin{equation}\label{apprstvar}
y_\rh(t,x)=\sum_{m=0}^M\sum_{j=1}^{N-1}y_{m,j}e^h_j(x)e^\tau_m(t)
=\sum_{m=0}^My_{m}^h(x)e^\tau_m(t)=\sum_{j=1}^{N-1}y_{j}^\tau(t)e^h_j(x)
\end{equation}
for $(t,x)\in \bar I\times \bar \Omega$ with $y_{m,j}\in \R$, $y_{m}^h\in V_h$ and $y_{j}^\tau\in V_\tau$.

\par We also}
define the forward and backward difference {quotients} and the average in time {operator}
\begin{gather*}
 {\delta_t} v_m=\frac{v_{m+1}-v_m}{\tau},~~{\bar{\delta}_t} v_m=\frac{v_m-v_{m-1}}{\tau},
\\
 B^\tau v_m=\frac16v_{m-1}+\frac23v_m+\frac16v_{m+1},~1\leq m\leq M-1,
\\
 B^\tau v_0=\frac13v_0+\frac16v_1,~~B^\tau v_M=\frac16v_{M-1}+\frac13v_M.
\end{gather*}
We define the self-adjoint positive-definite operators $B_h$ and $L_h$ acting in $V_h$ (in other words, the mass and stiffness matrices) such that
\[
 (B_h\varphi_h,\psi_h)_{V_h}=(\rho\varphi_h,\psi_h)_{H},~~
 (L_h\varphi_h,\psi_h)_{V_h}=(\kappa\partial_x\varphi_h,\partial_x\psi_h)_{H}\quad\forall \varphi,\psi\in V_h.
\]
For $w\in V^\ast$ and ${u}\in L^2(I,V^\ast)$ we define the vectors $w^h=\{\langle w,e_{j}^h\rangle_{\Omega}\}_{j=1}^{N-1}$ and
\begin{align*}
 {u}^\rh_m
 &=\frac{1}{\tau}\Big\{\big(\langle {u},e_{j}^h\rangle_{\Omega},e_m^\tau\big)_{L^2(I)}\Big\}_{j=1}^{N-1},~1\leq m\leq M-1,\\
 {u}^\rh_m
 &=\frac{2}{\tau}\Big\{\big(\langle {u},e_{j}^h\rangle_{\Omega},e_m^\tau\big)_{L^2(I)}\Big\}_{j=1}^{N-1},~m=0,M.
\end{align*}
We recall the form of the discrete state \eqref{apprstvar}.
\par\textit{The forward time-stepping} is implemented as follows. The integral identities \eqref{discrete_state_equation1}-\eqref{discrete_state_equation2} are equivalent to the operator equations
\begin{gather}
 (B_h+\sigma\tau^2L_h){\delta_t\bar{\delta}_t}y_{\rh,m}+L_hy_{\rh,m}=u_m^{\rh},~~m=2,\ldots,M-1,
\label{fts1}\\[1mm]
\textstyle{(B_h+\sigma\tau^2L_h){\delta_t}y_{\rh,1}+\frac{\tau}{2}L_hy_{\rh,0}=(\rho y^1)^h+\frac{\tau}{2}u_0^{\rh},}
\label{fts2}\\[1mm]
 B_hy_{\rh,0}=(\rho y^0)^h
\label{fts3}
\end{gather}
{followed by the counterpart of \eqref{fts2} at time $T$ for $y_{Th}^1$:}
\begin{gather}
\textstyle{  B_h{y_{Th}^1}
 =(B_h+\sigma\tau^2L_h){\bar{\delta}_t}y_{\rh,M}-\frac{\tau}{2}L_hy_{\rh,M}+\frac{\tau}{2}u_M^{\rh}.}
\label{fts4}
\end{gather}
\par Next \textit{the adjoint (backward) time-stepping} is implemented in a similar manner.
Namely, the integral identities \eqref{bar eq adj} are equivalent to the operator equations
\begin{gather}
\hspace{-4pt}
 (B_h+\sigma\tau^2L_h){\delta_t\bar{\delta}_t}p_{\rh,m}+L_hp_{\rh,m}=B_hB^\tau y_{\rh,m}-(\rho z_1)_m^{\rh},\ m=M-1,\dots,1,
\label{ats1}\\[1mm]
\textstyle{  -(B_h+\sigma\tau^2L_h){\bar{\delta}_t}p_{\rh,M}+\frac{\tau}{2}L_hp_{\rh,M}}
 =B_hy_{\rh,M}-(\rho z_2)^h+\frac{\tau}{2}\big(B_hB^\tau y_{\rh,M}-(\rho z_1)_M^{\rh}\big),
\label{ats2}\\[1mm]
 L_hp_{\rh,M}=B_h{y_{Th}^1}-z_3^h,
\label{ats3}
\end{gather}
{followed by the counterpart of \eqref{fts4} for $p_{0h}^1$:}
\begin{gather}
\textstyle{  B_h{p_{0h}^1}=(B_h+\sigma\tau^2L_h)\delta_tp_{\rh,0}+\frac{\tau}{2}L_hp_{\rh,0}
 -\frac{\tau}{2}\big(B_hB^\tau y_{\rh,0}-(\rho z_1)_0^{\rh}\big).}
\label{ats4}
\end{gather}
\begin{remark}\label{Crank Nicolson}
For $\sigma=1/4$ the three-level time stepping scheme \eqref{fts1}-\eqref{fts4} is closely related to the well-known two-level Crank-Nicolson method applied to the first order in time system
\begin{equation*}
\left\{\begin{aligned}
\partial_ty=v,~\rho\partial_{t}v-\partial_x(\kappa\partial_{x}y)&=u&&\text{in}~I\times \Omega\\
y&=0&&\text{on}~I\times\partial\Omega\\
y=y^0,~v&=y^1&&\text{in}~\{0\}\times\Omega,\\
\end{aligned}\right.
\end{equation*}
see \cite[Section 8]{Zlotnik94} for details,
as well as to the Petrov-Galerkin method described in \cite{KroenerKunischVexler11}.
After the mass lumping, for $\sigma=0$ our method becomes explicit and is related to the Leap-Frog method;
{moreover, for any $\sigma$ it becomes close to three-level finite-difference schemes with such weight in time,
eg. see \cite{Samarskii01}.
}
\end{remark}
\section{Control discretization. {Solution process and $L^2(I\times \Omega)$-regularization}}\label{sec:control disc}
Now we discuss in more detail solving of the semi-discrete optimization problem \eqref{semi_discrete_problem} in the case $\M=\mathcal{M}(\Omega,L^2(I))$.
\par An important point is that we can seek its solution in the form
\[
\bar u_\rh\in \mathcal M_\rh:=V_\tau\otimes\mathcal M_h,\ \
\mathcal M_h:=\spa\{\delta_{x_1},\ldots,\delta_{x_{N-1}}\}\subset\mathcal M(\Omega).
\]
{To show that, let} $\pi_\tau^0$  be the projector in $L^2(I)$ on $V_{\tau}$. Note that, for $\eta\in L^2(I)$, it satisfies
\begin{align*}
 (B^\tau\pi_\tau^0\eta)_m&=\frac{1}{\tau}(\eta,e_m^\tau)_{L^2(I)}~~\text{for}~~1\leq m\leq M-1,\\
 (B^\tau \pi^0_\tau\eta)_m&=\frac{2}{\tau}(\eta,e_m^\tau)_{L^2(I)}~~\text{for}~~m=0,M.
\end{align*}
Then we define { $\Pi_{h}$: $\mathcal{M}(\Omega)\to \mathcal{M}_h$ by
$\Pi_hw:=\sum_{j=1}^{N-1}\langle w,e_{j}^h\rangle_{\Omega}
 \delta_{x_j}$ and $\Pi_\rh=\pi_\tau^0\Pi_h$.}
The following identity holds
\[
 \langle\Pi_\rh u,v\rangle_{\M,\,\C}=\langle u,\pi_\tau^0i_hv\rangle_{\M,\,\C}\quad\forall u\in\M, v\in\C
\]
with {the interpolation operator $i_h$: $\mathcal C_0(\Omega)\to V_{h}$ such that $i_hw(x_j)=w(x_j)$ for all $j=0,\ldots,N$.}
In particular, if $v\in V_\rh$, then
\[
 \langle\Pi_\rh u,v\rangle_{\M,\,\C}=\langle u,v\rangle_{\M,\,\C},
\]
and consequently (like in \cite[Lemma 3.11]{KunischPieperVexler2014}) we have $S_\rh=S_\rh\circ\Pi_\rh$ as well as
{$\|\Pi_\rh u\|_{\M}\leq\|u\|_{\M}$.}
Thus for each solution $\tilde u_\rh$ of problem \eqref{semi_discrete_problem}, \textit{the discrete} \textit{control} $\Pi_\rh\tilde u_\rh$
satisfies
\[
 j_\rh(\tilde u_\rh)=j_\rh(\Pi_\rh\tilde u_\rh).
\]
Therefore $\Pi_\rh\tilde u_\rh$ is also a solution of \eqref{semi_discrete_problem}. This is a justification for solving the fully discrete problem
\begin{equation}\label{fully_discrete_problem}
j_\rh(u_\rh)=\half\left\|S_\rh u_\rh-{\mathbf{z}}\right\|_{{\mathcal Y_h}}^{2}+\alpha\|u_\rh\|_{\mathcal M(\Omega,L^2(I))}\to\min_{u_\rh\in \mathcal M_\rh}
\end{equation}
in order to get a solution of \eqref{semi_discrete_problem}.

\par The direct solution of \eqref{fully_discrete_problem} by means of a generalized Newton type method is a challenging problem since a proper globalization strategy is needed, see \cite{milzarekulbrich14}.
Thus we propose a solution strategy based on an additional $L^2(I\times\Omega)$-regularization of \eqref{fully_discrete_problem}
with a parameter $\gamma>0$ and a continuation method. For high values of $\gamma$
the corresponding Newton type method converges independently of the initial guess in numerical practice. Thus the continuation strategy can be seen as simple globalization strategy.
\par On the continuous level we consider the following regularized problem
\begin{equation}\label{regularized problem}
\textstyle{ j_\gamma(u)=\frac12\left\|Su-{\mathbf{z}}\right\|_{\mathcal Y}^{2}+\alpha\|u\|_{\mathcal M(\Omega,L^2(I))}
+\frac \gamma 2\|u\|_{L^2(I\times\Omega)}^2\to\min_{u\in L^2(I\times\Omega)}.}
\end{equation}
It is possible to formulate a semi-smooth Newton method for this problem on the continuous level which is based on the following necessary and sufficient optimality condition
\begin{equation}\label{reg_opt_cond}
\bar u_\gamma(t,x)=-\frac 1 \gamma\max\left(0,1-\frac{\alpha}{\|\bar p(\cdot,x)\|_{L^2(I)}}\right)\bar p(t,x),\quad (t,x)\in I\times \Omega,
\end{equation}
with $\bar p= S^\star{W_h(S\bar u_\gamma-\mathbf{z}})$. Moreover, this semi-smooth Newton method is superlinear convergent.
{Let $\bar u_\gamma$ and $\bar u$ be the unique solutions of \eqref{regularized problem} and \eqref{measure_control_problem}.}
Then we have $\bar u_\gamma\rightharpoonup^\ast \bar u$ in $\mathcal M(\Omega,L^2(I))${, see
\cite{KunischPieperVexler2014,Pieper:2015,HerzogStadlerWachsmuth2012}.}
This justifies the use of a continuation strategy in $\gamma$. The control discretization described
{above}
can not be used for \eqref{regularized problem}. Instead we propose to use discrete controls from $V_\rh$, i.e.,
\[
u_\rh(t,x)=\sum_{m=0}^M\sum_{j=1}^{N-1}u_{m,j}e^\tau_m(t)e^h_j(x)=\sum_{j=1}^{N-1}u_j(t)e^h_j(x)
=\sum_{m=0}^{M}u_m(x)e^\tau_{m}(t),
\]
cf. \eqref{apprstvar}. In particular, we solve the following fully discrete regularized problem
\begin{equation}\label{fully_discrete_regularized problem}
\textstyle{ j_\rh^\gamma(u_\rh)=\frac12\left\|S_\rh(l_\rh u_\rh)-{\mathbf{z}}\right\|_{{\mathcal Y_h}}^{2}+\alpha\|u_\rh\|_{\mathcal M(\Omega,L^2(I)),h}+\frac \gamma 2\|u_\rh\|_{L^2(I\times\Omega),h}^2\to\min_{u_\rh\in V_\rh}}
\end{equation}
with
\[
\|u_\rh\|_{\mathcal M(\Omega,L^2(I)),h}=\sum_{j=1}^{N-1}d_j\|u_j\|_{L^2(I)},\quad \|u_\rh\|_{L^2(I\times\Omega),h}^2= \sum_{m=0}^M (B^\tau u_m)^tD(B^\tau u_m)
\]
where $D=\diag(d_1,\ldots,d_{N-1})$
is the lumped mass matrix.
Moreover, the operator $l_\rh$ is defined by
\[
(l_\rh u_\rh,v_\rh)_{L^2(I\times \Omega),h}=\sum_{m=0}^M(B^\tau u_m)^tD(B^\tau v_m)\quad \forall u_\rh,v_\rh\in V_\rh.
\]
The use of $D$
allows us to derive the following optimality conditions for \eqref{fully_discrete_regularized problem}
\begin{equation}\label{discret_reg_opt_cond}
\bar u_{m,j}^\gamma
=-\frac 1 \gamma\max\left(0,1-\frac{\alpha}{\|\bar p_{\rh\,\cdot,j}\|_{L^2(I)}}\right)\bar p_{\rh\,m,j},
\end{equation}
for all $m$ and $j$,
with $\bar p_\rh=S^\star_\rh{W_h}(S_\rh \bar u_\rh-{\mathbf{z}})$, cf. \eqref{reg_opt_cond}.
Based on \eqref{discret_reg_opt_cond} we can set up a semi-smooth Newton method. Since problem \eqref{fully_discrete_regularized problem}
is a discretization of \eqref{regularized problem}, we can expect that this method behaves mesh independently. Let $\bar u_\rh^\gamma=\sum_{j=1}^{N-1}u_j(t)e^h_j$ be the solution of \eqref{fully_discrete_regularized problem} and we define
\[
\tilde u_\rh^\gamma=\sum_{j=1}^{N-1}d_ju_j(t)\delta_{x_j}.
\]
As $\gamma\rightarrow 0$ the control
$\tilde u_\rh^\gamma$ {tends to}
a solution of \eqref{fully_discrete_problem} justifying the use of this control discretization and the continuation strategy. For more details see \cite{Pieper:2015}.
\section{Numerical results}
\label{sec:numerics}
In this section, we present results of numerical experiments and consider two examples both involving zero initial data $y^0=y^1=0$, the control space $\M=\ML$ and the tracking functional
\[
F(y)={\half}\|y-z\|_{\L}^2,\ \ \textstyle{ z(x):=\frac{1}{\sqrt{2\pi\rho}}e^{-\frac{(x-\lambda)^2}{2\rho^2}}}
\]
with the time independent desired state $z$ which is a Gaussian centered at $x=\lambda$.
We choose $\rho=0.1$ and $\lambda$ as an irrational parameter.

For sufficiently large $\alpha$ ($\alpha = 0.1$), we expect that the optimal control {$\bar{u}$} consist{s} of one point source with a position close to {$\lambda$.}
If the Gaussian would move through the domain, a point source shaped {$\bar{u}$} is not able to follow the center of the Gaussian since
$\ML$ contains no moving point sources.
The optimal control would rather consist of some additional fixed point sources. This would not lower the regularity of the state whereas a moving point source can cause it.
\par The domain $\Omega$ and the time segment $\bar{I}$ are discretized by the uniform grids for $N=2^{r_h}$ and $M=2^{r_\tau}$
where $r_\tau,r_h=2,3,\ldots,r^{\max}$ with $r^{\max}=10$.
The stability parameter is fixed to its lowest value $\sigma=1/4$ ensuring unconditional stability of the time-stepping method.
The discrete control problem is solved for $r_h=2,3,\ldots,{r^{\max}}$ and the fixed $r_\tau=r^{\max}$ and then vice versa.
The solution process {has been described above in} Section \ref{sec:control disc}.
Numerically the desired state $z$ is replaced by $i_hz$ for simplicity, moreover the corresponding error $\mathcal{O}(h^2)$ is negligible. Since the optimal pairs $(\bar u,\bar y)$ are not known in our examples, we replace them by reference solutions $(\hat u,\hat y)$ which are taken as the approximate solutions on the finest grid level.
\smallskip\par \textbf{Example 1}.
We first take the constant coefficients $\rho\equiv 1$ and $\kappa\equiv 1$ and set $\lambda=\pi/20$.
We depict the reference solution $(\hat u,\hat y)$ in Figure \ref{fig:ref solution}.

\begin{figure}[!htb]
\centering
\begin{subfigure}{0.48\textwidth}
\centering
\includegraphics[width=5cm,height=5cm]{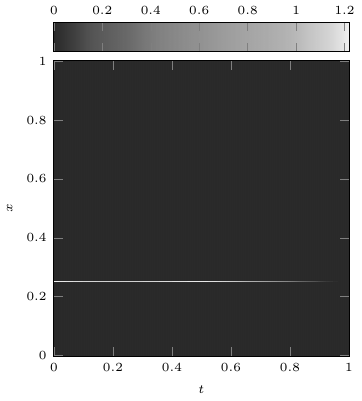}
\caption{$\hat u$ (on a coarser grid)}
\end{subfigure}
\quad
\begin{subfigure}{0.48\textwidth}
\centering
\includegraphics[width=5cm,height=5cm]{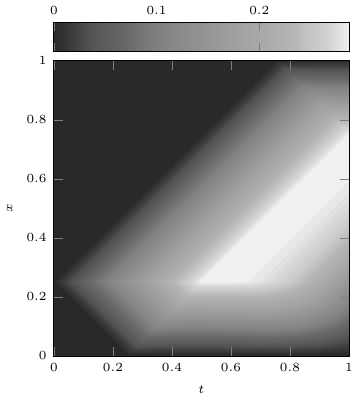}
\caption{$\hat y$}
\end{subfigure}
\caption{Example 1: the reference solution $(\hat u,\hat y)$}\label{fig:ref solution}
\end{figure}
As expected, the optimal control $\hat u$ consists only of one point source positioned in the vicinity of {$\lambda$.}
Thus, the state $\hat y$ has a kink at this position. Due to reflections at the boundary, $\hat y$ has also kinks at other positions.
\par Next, we discuss the convergence results.
In Figure \ref{fig:h refine}, we see the convergence rate of $\|\bar y_\sigma-\hat y\|_{\L}$ (left) and the objective functional (right) as $h$ refines.
The state error behaves mostly in a linear way and the rate for the functional is close to two; as usual the latter is approximately the doubled rate of the former, and fortunately both are better than the above proved theoretical rates.
\begin{figure}[!htb]
\centering
\begin{subfigure}{0.48\textwidth}
  \centering
%\begin{tikzpicture}
%
%\begin{axis}[
%width=5cm,
%height=6cm,
%at={(0.0in,0.0in)},
%scale only axis,
%xmode=log,
%xmin=1,
%xmax=10000,
%xminorticks=true,
%ymode=log,
%ymin=1e-07,
%ymax=1,
%yminorticks=true,
%legend style={at={(0.02,0.02)},anchor=south west,legend cell align=left,align=left,draw=white!15!black,font=\tiny,draw=none,fill=none}
%]
%
%\addplot [color=black!60,solid,line width=1.0pt,mark=square]
%  table[row sep=crcr]{
%5	0.25\\
%9	0.125\\
%17	0.0625\\
%33	0.03125\\
%65	0.015625\\
%129	0.0078125\\
%257	0.00390625\\
%513	0.001953125\\
%1025	0.0009765625\\
%2049	0.00048828125\\
%};
%\addlegendentry{$O(h)$};
%
%\addplot [color=black!80,solid,line width=1.0pt,mark=o]
%  table[row sep=crcr]{
%5	0.629960524947437\\
%9	0.396850262992050\\
%17	0.250000000000000\\
%33	0.157490131236859\\
%65	0.0992125657480125\\
%129	0.0625000000000000\\
%257	0.0393725328092148\\
%513	0.0248031414370031\\
%1025	0.0156250000000000\\
%2049	0.00984313320230370\\
%};
%\addlegendentry{$O(h^{2/3})$};
%
%\addplot [color=black,solid,line width=1.0pt,mark=x]
%  table[row sep=crcr]{
%5	0.0537111878708385\\
%9	0.0187956308574426\\
%17	0.00626703924577896\\
%33	0.00249686254392515\\
%65	0.00115224828062746\\
%129	0.000563238770941444\\
%257	0.000279586318541606\\
%513	0.000139293558635259\\
%1025	6.95003055326511e-05\\
%2049	2.9339021878709e-05\\
%};
%\addlegendentry{$\|\bar y_\sigma-\hat y\|_{\L}$};
%
%\end{axis}
%\end{tikzpicture}
\includegraphics[width=5cm,height=6cm]{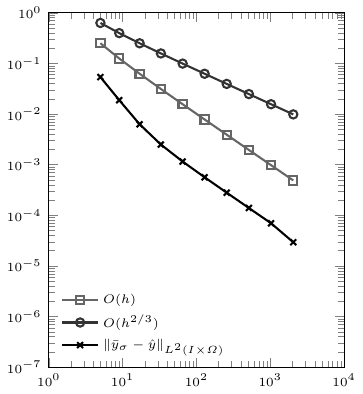}
\end{subfigure}
\quad
\begin{subfigure}{0.48\textwidth}
%  \centering
%\begin{tikzpicture}
%
%\begin{axis}[
%width=5cm,
%height=6cm,
%at={(0.0in,0.0in)},
%scale only axis,
%xmode=log,
%xmin=1,
%xmax=10000,
%xminorticks=true,
%ymode=log,
%ymin=1e-07,
%ymax=1,
%yminorticks=true,
%legend style={at={(0.02,0.02)},anchor=south west,legend cell align=left,align=left,draw=white!15!black,font=\tiny,draw=none,fill=none}
%]
%\addplot [color=black!60,solid,line width=1.0pt,mark=square]
%  table[row sep=crcr]{
%5	0.0625\\
%9	0.015625\\
%17	0.00390625\\
%33	0.0009765625\\
%65	0.000244140625\\
%129	6.103515625e-05\\
%257	1.52587890625e-05\\
%513	3.814697265625e-06\\
%1025	9.5367431640625e-07\\
%2049	2.38418579101563e-07\\
%0	0\\
%};
%\addlegendentry{$O(h^2)$};
%
%
%\addplot [color=black!80,solid,line width=1.0pt,mark=o]
%  table[row sep=crcr]{
%5	0.39685026299205\\
%9	0.25\\
%17	0.157490131236859\\
%33	0.0992125657480125\\
%65	0.0625\\
%129	0.0393725328092148\\
%257	0.0248031414370031\\
%513	0.015625\\
%1025	0.0098431332023037\\
%2049	0.00620078535925078\\
%0	0\\
%};
%\addlegendentry{$O(h^{2/3})$};
%
%\addplot [color=black,solid,line width=1.0pt,mark=x]
%  table[row sep=crcr]{
%5	0.643488818550255\\
%9	0.16140700405639\\
%17	0.0459958357303263\\
%33	0.011887029698799\\
%65	0.00299608381948802\\
%129	0.000750051257941697\\
%257	0.000187079096528597\\
%513	4.624466960923e-05\\
%1025	1.10304004201645e-05\\
%2049	2.21368764585783e-06\\
%0	0\\
%};
%\addlegendentry{$|J_\sigma(\bar u_\sigma,\bar y_\sigma)-J(\hat u,\hat y)|$};
%
%\end{axis}
%\end{tikzpicture}
\includegraphics[width=5cm,height=6cm]{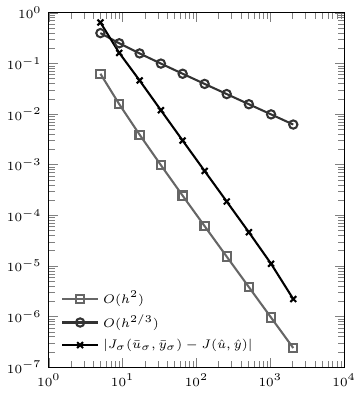}
\end{subfigure}
\caption{Example 1: errors as $h$ refines and $M=2^{10}$}
\label{fig:h refine}
\end{figure}
In Figure \ref{fig:tau refine}, we see the similar results as $\tau$ refines.
The error of the functional stagnates at the last $\tau$ refinement that is caused by a too coarse space grid.
Nevertheless, we observe reduced rates {for $\hat{y}$ much less than two} caused by its reduced regularity (kinks).
\begin{figure}[!htb]
\centering
\begin{subfigure}[t]{0.48\textwidth}
  \centering
%\begin{tikzpicture}
%
%\begin{axis}[
%width=5cm,
%height=6cm,
%at={(0.0in,0.0in)},
%scale only axis,
%xmode=log,
%xmin=1,
%xmax=10000,
%xminorticks=true,
%ymode=log,
%ymin=1e-07,
%ymax=1,
%yminorticks=true,
%legend style={at={(0.02,0.02)},anchor=south west,legend cell align=left,align=left,draw=white!15!black,font=\tiny,draw=none,fill=none}
%]
%
%\addplot [color=black!60,solid,line width=1.0pt,mark=square]
%  table[row sep=crcr]{
%5	0.25\\
%9	0.125\\
%17	0.0625\\
%33	0.03125\\
%65	0.015625\\
%129	0.0078125\\
%257	0.00390625\\
%513	0.001953125\\
%1025	0.0009765625\\
%2049	0.00048828125\\
%0	0\\
%};
%\addlegendentry{$O(\tau)$};
%
%\addplot [color=black!80,solid,line width=1.0pt,mark=o]
%  table[row sep=crcr]{
%5   0.629960524947437\\
%9   0.396850262992050\\
%17  0.250000000000000\\
%33  0.157490131236859\\
%65  0.0992125657480125\\
%129 0.0625000000000000\\
%257 0.0393725328092148\\
%513 0.0248031414370031\\
%1025    0.0156250000000000\\
%2049 0.00984313320230370\\
%};
%\addlegendentry{$O(\tau^{2/3})$};
%
%\addplot [color=black,solid,line width=1.0pt,mark=x]
%  table[row sep=crcr]{
%5	0.0295920414161701\\
%9	0.0108403724661214\\
%17	0.00570494878610595\\
%33	0.00294485132038744\\
%65	0.00150719262603264\\
%129	0.000766606778394984\\
%257	0.00038743004605686\\
%513	0.000194197286891723\\
%1025	9.53057407957076e-05\\
%2049	4.71191629858492e-05\\
%0	0\\
%};
%\addlegendentry{$\|\bar y_\sigma-\hat y\|_{\L}$};
%
%\end{axis}
%\end{tikzpicture}
\includegraphics[width=5cm,height=6cm]{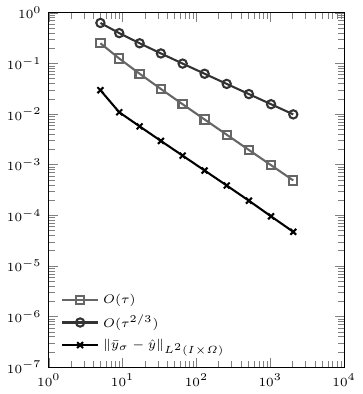}
\end{subfigure}
\quad
\begin{subfigure}[t]{0.48\textwidth}
  \centering
%\begin{tikzpicture}
%
%\begin{axis}[
%width=5cm,
%height=6cm,
%at={(0.0in,0.0in)},
%scale only axis,
%xmode=log,
%xmin=1,
%xmax=10000,
%xminorticks=true,
%ymode=log,
%ymin=1e-08,
%ymax=1,
%yminorticks=true,
%legend style={at={(0.02,0.02)},anchor=south west,legend cell align=left,align=left,draw=white!15!black,font=\tiny,draw=none,fill=none}
%]
%
%\addplot [color=black!60,solid,line width=1.0pt,mark=square]
%  table[row sep=crcr]{
%5	0.0625\\
%9	0.015625\\
%17	0.00390625\\
%33	0.0009765625\\
%65	0.000244140625\\
%129	6.103515625e-05\\
%257	1.52587890625e-05\\
%513	3.814697265625e-06\\
%1025	9.5367431640625e-07\\
%2049	2.38418579101563e-07\\
%0	0\\
%};
%\addlegendentry{$O(\tau^2)$};
%
%\addplot [color=black!80,solid,line width=1.0pt,mark=o]
%  table[row sep=crcr]{
%5	0.39685026299205\\
%9	0.25\\
%17	0.157490131236859\\
%33	0.0992125657480125\\
%65	0.0625\\
%129	0.0393725328092148\\
%257	0.0248031414370031\\
%513	0.015625\\
%1025	0.0098431332023037\\
%2049	0.00620078535925078\\
%0	0\\
%};
%\addlegendentry{$O(\tau^{2/3})$};
%
%\addplot [color=black,solid,line width=1.0pt,mark=x]
%  table[row sep=crcr]{
%5	0.00330245631452453\\
%9	0.000466692350604481\\
%17	0.000133796355854532\\
%33	3.4191772143366e-05\\
%65	8.67643881408142e-06\\
%129	2.17822564185788e-06\\
%257	5.45107018634639e-07\\
%513	1.34986493982581e-07\\
%1025	3.21824533688897e-08\\
%2049	4.34465561305331e-08\\
%0	0\\
%};
%\addlegendentry{$|J_\sigma(\bar u_\sigma,\bar y_\sigma)-J(\hat u,\hat y)|$};
%
%\end{axis}
%\end{tikzpicture}
\includegraphics[width=5cm,height=6cm]{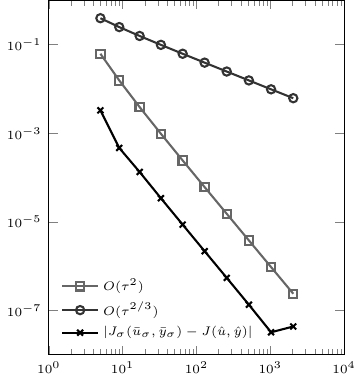}
\end{subfigure}
\caption{Example 1: errors as $\tau$ refines and $N=2^{10}$}
\label{fig:tau refine}
\end{figure}
\smallskip\par\textbf{Example 2}.
Now we take the variable coefficient
\[
\kappa(x)=
\begin{cases}
1.2&\quad 0.25<x\leq 1,\\
0.2&\quad 0\leq x<0.25
\end{cases}
\]
and set $\lambda = \pi/6$.
Our analysis does not cover discontinuous coefficients, but they are of great importance in applications, for example, in seismic tomography.
A jump discontinuity in $\kappa$ translates to a jump in the wave speed which can be related to two different material characteristics changing at the point of discontinuity.
Note that the point of discontinuity is a grid point for all grid levels.
\begin{figure}[!htb]
\centering
\begin{subfigure}{0.48\textwidth}
\centering
\includegraphics[width=5cm,height=5cm]{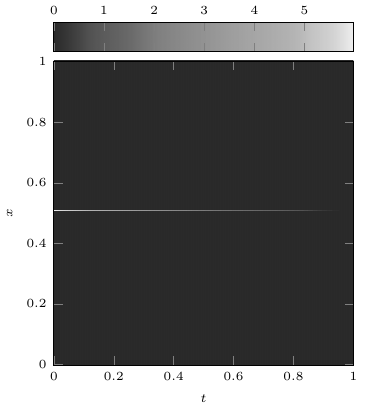}
\caption{$\hat u$ (on a coarser grid)}
\end{subfigure}
\quad
\begin{subfigure}{0.48\textwidth}
\centering
\includegraphics[width=5cm,height=5cm]{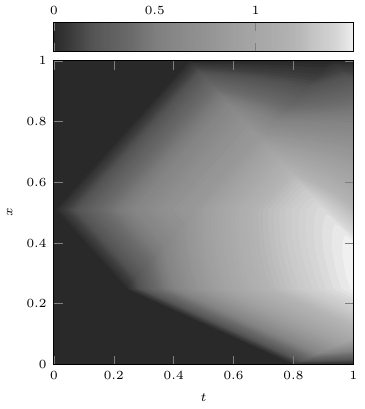}
\caption{$\hat y$}
\end{subfigure}
\caption{Example 2: the reference solution $(\hat u,\hat y)$}\label{fig:ref solution_kappa}
\end{figure}
The reference solution $(\hat u,\hat y)$ is displayed in Figure \ref{fig:ref solution_kappa}.
Once again $\hat u$ consists of one Dirac measure with a time-dependent intensity located in the vicinity of $\lambda=\pi/6$ and
thus $\hat y$ has a kink at this position.
Moreover, we can clearly see that at $x=0.25$ the wave speed changes and the wave propagation becomes slower.
\begin{figure}[!htb]
\centering
\begin{subfigure}{0.48\textwidth}
  \centering
%\begin{tikzpicture}
%
%\begin{axis}[
%width=5cm,
%height=6cm,
%at={(0.0in,0.0in)},
%scale only axis,
%xmode=log,
%xmin=1,
%xmax=10000,
%xminorticks=true,
%ymode=log,
%ymin=0.0001,
%ymax=1,
%yminorticks=true,
%legend style={at={(0.02,0.02)},anchor=south west,legend cell align=left,align=left,draw=white!15!black,font=\tiny,draw=none,fill=none}
%]
%
%\addplot [color=black!60,solid,line width=1.0pt,mark=square]
%  table[row sep=crcr]{
%5	0.25\\
%9	0.125\\
%17	0.0625\\
%33	0.03125\\
%65	0.015625\\
%129	0.0078125\\
%257	0.00390625\\
%513	0.001953125\\
%1025	0.0009765625\\
%2049	0.00048828125\\
%};
%\addlegendentry{$\mathcal O(h)$};
%
%\addplot [color=black!80,solid,line width=1.0pt,mark=o]
%  table[row sep=crcr]{
%5	0.39685026299205\\
%9	0.25\\
%17	0.157490131236859\\
%33	0.0992125657480125\\
%65	0.0625\\
%129	0.0393725328092148\\
%257	0.0248031414370031\\
%513	0.015625\\
%1025	0.0098431332023037\\
%2049	0.00620078535925078\\
%};
%\addlegendentry{$\mathcal O(h^{2/3})$};
%
%\addplot [color=black,solid,line width=1.0pt,mark=x]
%  table[row sep=crcr]{
%5	0.121802035986403\\
%9	0.053406713980365\\
%17	0.0294909014947853\\
%33	0.0189553783508336\\
%65	0.00858202862715643\\
%129	0.004196949311815\\
%257	0.00210736506463611\\
%513	0.000944132841358249\\
%1025	0.000418514116188891\\
%2049	0.000172804719168907\\
%};
%\addlegendentry{$\|\bar y_\sigma-\hat y\|_{\L}$};
%
%\end{axis}
%\end{tikzpicture}
\includegraphics[width=5cm,height=6cm]{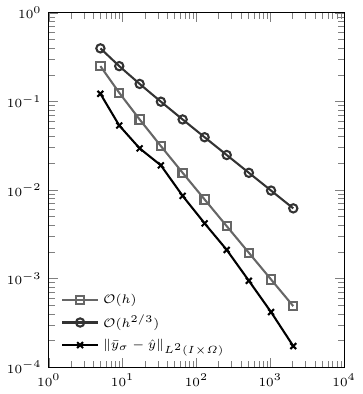}
\end{subfigure}
\quad
\begin{subfigure}{0.48\textwidth}
  \centering
%\begin{tikzpicture}
%
%\begin{axis}[
%width=5cm,
%height=6cm,
%at={(0.0in,0.0in)},
%scale only axis,
%xmode=log,
%xmin=1,
%xmax=10000,
%xminorticks=true,
%ymode=log,
%ymin=1e-07,
%ymax=1,
%yminorticks=true,
%legend style={at={(0.02,0.02)},anchor=south west,legend cell align=left,align=left,draw=white!15!black,font=\tiny,draw=none,fill=none}
%]
%\addplot [color=black!60,solid,line width=1.0pt,mark=square]
%  table[row sep=crcr]{
%5	0.0625\\
%9	0.015625\\
%17	0.00390625\\
%33	0.0009765625\\
%65	0.000244140625\\
%129	6.103515625e-05\\
%257	1.52587890625e-05\\
%513	3.814697265625e-06\\
%1025	9.5367431640625e-07\\
%2049	2.38418579101563e-07\\
%};
%\addlegendentry{$\mathcal O(h^2)$};
%
%
%\addplot [color=black!80,solid,line width=1.0pt,mark=o]
%  table[row sep=crcr]{
%5	0.39685026299205\\
%9	0.25\\
%17	0.157490131236859\\
%33	0.0992125657480125\\
%65	0.0625\\
%129	0.0393725328092148\\
%257	0.0248031414370031\\
%513	0.015625\\
%1025	0.0098431332023037\\
%2049	0.00620078535925078\\
%};
%\addlegendentry{$\mathcal O(h^{2/3})$};
%
%\addplot [color=black,solid,line width=1.0pt,mark=x]
%  table[row sep=crcr]{
%5	0.119929606060846\\
%9	0.135458627959623\\
%17	0.0398848443359388\\
%33	0.0101071621557185\\
%65	0.00247996114003701\\
%129	0.000647816410938695\\
%257	0.000155387685766506\\
%513	4.10171165510231e-05\\
%1025	9.7204160036668e-06\\
%2049	1.91944149241507e-06\\
%};
%\addlegendentry{$|J_\sigma(\bar u_\sigma,\bar y_\sigma)-J(\hat u,\hat y)|$};
%\end{axis}
%\end{tikzpicture}
\includegraphics[width=5cm,height=6cm]{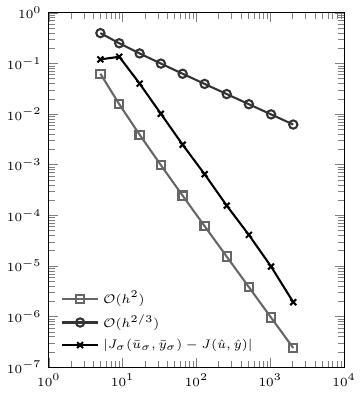}
\end{subfigure}
\caption{Example 2: errors as $h$ refines and $M=2^{10}$}
\label{fig:h refine_kappa}
\end{figure}
In Figure \ref{fig:h refine_kappa} we observe that the error of the state variable converges  in a linear way whereas the error measured in the objective functional behaves quadratically.
\begin{figure}[!htb]
\centering
\begin{subfigure}{0.48\textwidth}
  \centering
%\begin{tikzpicture}
%
%\begin{axis}[
%width=5cm,
%height=6cm,
%at={(0.0in,0.0in)},
%scale only axis,
%xmode=log,
%xmin=1,
%xmax=10000,
%xminorticks=true,
%ymode=log,
%ymin=0.0001,
%ymax=1,
%yminorticks=true,
%legend style={at={(0.02,0.02)},anchor=south west,legend cell align=left,align=left,draw=white!15!black,font=\tiny,draw=none,fill=none}
%]
%
%\addplot [color=black!60,solid,line width=1.0pt,mark=square]
%  table[row sep=crcr]{
%5	0.25\\
%9	0.125\\
%17	0.0625\\
%33	0.03125\\
%65	0.015625\\
%129	0.0078125\\
%257	0.00390625\\
%513	0.001953125\\
%1025	0.0009765625\\
%2049	0.00048828125\\
%};
%\addlegendentry{$\mathcal O(h)$};
%
%\addplot [color=black!80,solid,line width=1.0pt,mark=o]
%  table[row sep=crcr]{
%5	0.39685026299205\\
%9	0.25\\
%17	0.157490131236859\\
%33	0.0992125657480125\\
%65	0.0625\\
%129	0.0393725328092148\\
%257	0.0248031414370031\\
%513	0.015625\\
%1025	0.0098431332023037\\
%2049	0.00620078535925078\\
%};
%\addlegendentry{$\mathcal O(h^{2/3})$};
%
%\addplot [color=black,solid,line width=1.0,mark=x]
%  table[row sep=crcr]{
%5	0.0812387247443535\\
%9	0.0474156345449169\\
%17	0.025508580267722\\
%33	0.0134792313672558\\
%65	0.00697458539106932\\
%129	0.00356597430246682\\
%257	0.00180313617468452\\
%513	0.000900510473414845\\
%1025	0.000437050648869284\\
%2049	0.000190076619290246\\
%};
%\addlegendentry{$\|\bar y_\sigma-\hat y\|_{\L}$};
%\end{axis}
%\end{tikzpicture}
\includegraphics[width=5cm,height=6cm]{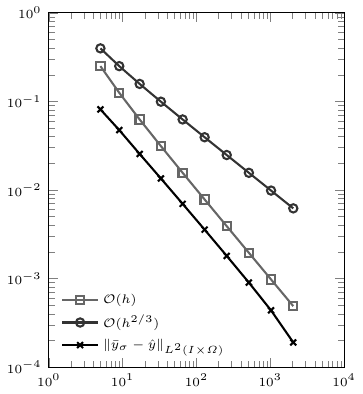}
\end{subfigure}
\quad
\begin{subfigure}{0.48\textwidth}
  \centering
%\begin{tikzpicture}
%
%\begin{axis}[
%width=5cm,
%height=6cm,
%at={(0.0in,0.0in)},
%scale only axis,
%xmode=log,
%xmin=1,
%xmax=10000,
%xminorticks=true,
%ymode=log,
%ymin=1e-08,
%ymax=1,
%yminorticks=true,
%legend style={at={(0.02,0.02)},anchor=south west,legend cell align=left,align=left,draw=white!15!black,font=\tiny,draw=none,fill=none}
%]
%
%\addplot [color=black!60,solid,line width=1.0pt,mark=square]
%  table[row sep=crcr]{
%5	0.0625\\
%9	0.015625\\
%17	0.00390625\\
%33	0.0009765625\\
%65	0.000244140625\\
%129	6.103515625e-05\\
%257	1.52587890625e-05\\
%513	3.814697265625e-06\\
%1025	9.5367431640625e-07\\
%2049	2.38418579101563e-07\\
%};
%\addlegendentry{$\mathcal O(h^2)$};
%
%\addplot [color=black!80,solid,line width=1.0pt,mark=o]
%  table[row sep=crcr]{
%5	0.39685026299205\\
%9	0.25\\
%17	0.157490131236859\\
%33	0.0992125657480125\\
%65	0.0625\\
%129	0.0393725328092148\\
%257	0.0248031414370031\\
%513	0.015625\\
%1025	0.0098431332023037\\
%2049	0.00620078535925078\\
%};
%\addlegendentry{$O(h^{2/3})$};
%
%\addplot [color=black,solid,line width=1.0pt,mark=x]
%  table[row sep=crcr]{
%5	0.00497000069780107\\
%9	0.00063101617762884\\
%17	0.000312773515589448\\
%33	0.000219047937272565\\
%65	6.58655857239676e-05\\
%129	1.70903863359051e-05\\
%257	4.30737503109846e-06\\
%513	1.06828988100105e-06\\
%1025	2.54739038307861e-07\\
%2049	5.09694246808579e-08\\
%};
%\addlegendentry{$|J_\sigma(\bar u_\sigma,\bar y_\sigma)-J(\hat u,\hat y)|$};
%\end{axis}
%\end{tikzpicture}
\includegraphics[width=5cm,height=6cm]{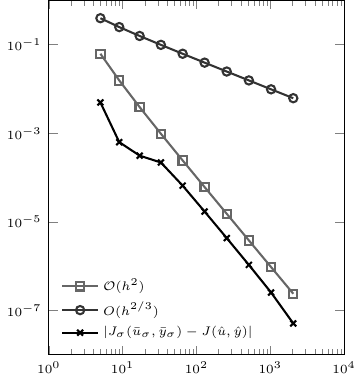}
\end{subfigure}
\caption{Example 2: errors as $\tau$ refines and $N=2^{10}$}
\label{fig:tau refine_kappa}
\end{figure}
Finally in Figure \ref{fig:tau refine_kappa} we study the error behaviour for $\tau$-refinement and find the similar rates of convergence.
So somewhat surprisingly we find that the convergence behavior of the error is comparable to the previous Example 1 with $\kappa\equiv1$ that stimulates further possible studies.}

%\medskip
%Received \Blue{ February, 10 2018.}
%\medskip
\end{document}